\documentclass[11pt,a4paper,final]{amsart}
\usepackage[utf8]{inputenc}
\usepackage[T1]{fontenc}
\usepackage[UKenglish]{babel}
\usepackage[a4paper,margin=1in]{geometry}

\usepackage{amsmath,amsthm,amsfonts,amssymb}
\usepackage{mathrsfs,dsfont}
\usepackage{graphicx}
\usepackage{url}
\usepackage{verbatim}
\usepackage{xcolor}

\usepackage{hyperref}

\usepackage[notcite,notref]{showkeys}

\makeatletter
\def\namedlabel#1#2{\begingroup
    #2%
    \def\@currentlabel{#2}%
    \label{#1}\endgroup
}
\makeatother

\theoremstyle{plain}
\newtheorem{theorem}{Theorem}[section]
\newtheorem{corollary}[theorem]{Corollary}
\newtheorem{lemma}[theorem]{Lemma}

\theoremstyle{definition}
\newtheorem{remark}[theorem]{Remark}
\newtheorem{example}[theorem]{Example}
\newtheorem{definition}[theorem]{Definition}

\newtheorem*{notation}{Notation}

\numberwithin{equation}{section}

\renewcommand\labelenumi{\textup{\alph{enumi})}}

\renewcommand\theenumi\labelenumi
\makeatletter\renewcommand{\p@enumii}{}\makeatother 

\renewcommand{\leq}{\leqslant}

\renewcommand{\geq}{\geqslant}

\newcommand{\et}{\quad\text{and}\quad}
\newcommand{\loc}{\mathrm{loc}}
\newcommand{\bee}{\mathrm{b}}
\newcommand{\ttau}[1]{\Lambda(#1)}
\newcommand{\ittau}{\rho}

\DeclareMathOperator{\supp}{supp}
\DeclareMathOperator{\spec}{spec}

\DeclareMathOperator{\dist}{dist}

\newcommand{\cC}{\mathcal{C}}

\newcommand{\real}{\mathds{R}}
\newcommand{\Pp}{\mathds{P}}
\newcommand{\Ee}{\mathds{E}}
\newcommand{\I}{\mathds{1}}

\newcommand{\pr}{\mathds{P}}

\newcommand{\ex}{\mathds{E}}
\newcommand{\dsphere}{ \mathds{S}^{d-1}}
\newcommand{\nat}{\mathds{N}}
\newcommand{\re}{\mathop{\mathrm{Re}}}

\newcommand{\A}[2]{A_{#1}^{#2}}
\newcommand{\casymp}[1]{\stackrel{#1}{\asymp}}

\newcommand{\Ci}{C'}
\newcommand{\Cii}{C''}
\newcommand{\Ciii}{C^*}
\newcommand{\tauAninfty}{\tau_{n-2}^{\infty}}
\newcommand{\tauAkinfty}{\tau_{k-2}^{\infty}}
\newcommand{\sigmaAkk}{\sigma_{k-1}^{k}}
\newcommand{\sigmaApp}{\sigma_{p-1}^{p}}
\newcommand{\tauAnOinfty}{\tau_{n_0-2}^{\infty}}
\newcommand{\tauAnOOinfty}{\tau_{n_0-1}^{\infty}}

\begin{document}

\title[Heat kernel estimates for non-local Schr\"odinger operators]%
{Progressive intrinsic ultracontractivity and heat kernel estimates for non-local Schr\"odinger operators}

\date{\today}

\author[K.~Kaleta]{Kamil Kaleta}
\address[K.~Kaleta]{Faculty of Pure and Applied Mathematics\\ Wroc{\l}aw University of Science and Technology\\ ul. Wybrze{\.z}e Wyspia{\'n}skiego 27, 50-370 Wroc{\l}aw, Poland}
\email{kamil.kaleta@pwr.edu.pl}

\thanks{K.~Kaleta gratefully acknowledges support through the Alexander von Humboldt Foundation (Germany) and by the National Science Centre (Poland) grant no.\ 2015/18/E/ST1/00239. Part of this research has been carried out during K.~Kaleta's stay as Humboldt fellow at TU Dresden.}

\author[R.L.~Schilling]{Ren\'e L.\ Schilling}
\address[R.L.~Schilling]{TU Dresden\\ Fakult\"{a}t Mathematik\\ Institut f\"{u}r Mathematische Stochastik\\ 01062 Dresden, Germany}
\email{rene.schilling@tu-dresden.de}

\begin{abstract}
    We study the long-time asymptotic behaviour of semigroups generated by non-local Schr\"odinger operators of the form $H = -L+V$; the free operator $L$ is the generator of a symmetric L\'evy process in $\mathds R^d$, $d > 1$ (with non-degenerate jump measure) and $V$ is a sufficiently regular confining potential. We establish sharp two-sided estimates of the corresponding heat kernels for large times and identify a new general regularity property, which we call progressive intrinsic ultracontractivity, to describe the large-time evolution of the corresponding Schr\"odinger semigroup. We discuss various examples and applications of these estimates, for instance we characterize the heat trace and heat content. Our examples cover a wide range of processes and we have to assume only mild restrictions on the growth, resp.\ decay, of the potential and the jump intensity of the free process. Our approach is based on a combination of probabilistic and analytic methods; our examples include fractional and quasi-relativistic Schr\"odinger operators.
\end{abstract}



\maketitle

\section{Introduction, assumptions and statement of main results}\label{sec1}

Over the past decade, there has been an increasing interest in non-local models involving Schr\"odinger operators associated with generators of L\'evy processes with non-degenerate jump measures. Recent investigations include heat kernel and heat trace estimates~\cite{bib:AVB,bib:BYY,bib:BGJP,bib:W}, gradient estimates of harmonic functions~\cite{bib:K2013}, properties of radial solutions, ground states, eigenfunctions and eigenvalues, and spectral bounds~\cite{bib:BL,bib:CMS,FLS2015,FLS2008,bib:JW,bib:Kal2012,KL2015,bib:KM,bib:LM,bib:T}, intrinsic hyper- and ultracontractivity properties of Schr\"odinger semigroups~\cite{bib:ChW2015,bib:ChW2016,bib:KaKu,bib:KL15,bib:KKL2018} as well as applications in quantum field theory~\cite{BSSch2, BSSch1, bib:HH, HIL}.

Typically, the operator is of the form $H = -L + V$; throughout we assume that the potential $V$ is locally bounded and $L$ is the generator of a symmetric L\'evy process. The corresponding semigroup $\left\{e^{tL}: t \geq 0 \right\}$ is a semigroup of convolution operators in $L^2=L^2(\real^d, dx)$ mapping $L^2$ into $L^{\infty}$ (if $t$ is large enough). Since $e^{tL}$ is positivity preserving this mapping property is equivalent to either of the following statements: (i) $e^{tL} : L^2 \to L^\infty$ is continuous for large $t$ (see e.g.~\cite[Corollary 1.3]{JSch2006}) or (ii) $X_t$ has a bounded probability density for large $t$. The symmetry of the L\'evy process is equivalent to the symmetry of the semigroup operators $e^{-tL}$, $t\geq 0$, and the self-adjointness of the generator $L$; $L$, $\left\{e^{tL}: t \geq 0 \right\}$ and the corresponding L\'evy process are called \emph{free} generator, semigroup and process. The Schr\"odinger operator $H$ generates a semigroup of symmetric operators $\left\{e^{-tH}: t \geq 0 \right\}$ on $L^2$ such that $e^{-tH}:L^2 \to L^{\infty}$ are bounded for large values of $t$. If the free L\'evy process has a sufficiently regular transition density, then $U_t$ is an integral operator with kernel $u_t(x,y)$, i.e.\ $U_tf(x) = \int u_t(x,y) f(y)\,dy$ \cite[Chapter 2.B]{bib:DC}. For a confining potential $V$, i.e.\ $\lim_{|x|\to\infty} V(x)=\infty$, each $U_t = e^{-tH}$, $t > 0$, is a compact operator in $L^2$ --- see e.g.\ \cite[Lemmas 1 and 9]{bib:KaKu} for a general argument --- and the spectra of $H$ and $U_t$ are purely discrete. We denote by $\lambda_0 = \inf\spec(H)\in\real$ the ground state eigenvalue and by $\varphi_0\in L^2$ the corresponding ground state eigenfunction. In particular, it makes sense to study the spectral regularity --- the heat trace or the heat content, and the Hilbert-Schmidt property --- and large-time smoothness properties of the semigroup $\left\{U_t: t \geq 0 \right\}$ such as intrinsic hyper- and ultracontractivity \cite{bib:D}. Related to that, it is also a natural question to ask for the behaviour of $U_tf$ and $u_t(x,y)$ as $t\to\infty$.

In the literature \cite{bib:B,bib:DS,bib:KL12} the (asymptotic) intrinsic ultracontractivity condition (a)IUC is used to describe the large time behaviour of $u_t(x,y)$. These are conditions on $U_t$ which can be equivalently stated in the following form
\begin{gather}
\tag{IUC}
    \forall t_0 > 0 \ \exists C=C(t_0) \geq 1 \ \forall t \geq t_0\ \forall x,y\in\real^d \,:  u_t(x,y) \casymp{C} e^{-\lambda_0 t}\varphi_0(x)\varphi_0(y), \\
\tag{aIUC}
    \exists t_0 > 0 \ \exists C=C(t_0) \geq 1 \ \forall t \geq t_0\ \forall x,y\in\real^d \,: u_t(x,y) \casymp{C} e^{-\lambda_0 t}\varphi_0(x)\varphi_0(y).
\end{gather}
(``$\casymp{C}$'' denotes a two-sided comparison with the constants $0 < C^{-1} \leq 1 \leq C < \infty$.)

In the present paper we are mainly interested in the case where $U_t = e^{-tH}$ fails to be (asymptotically) IUC and we want to study the asymptotics of $u_t(x,y)$ as $t\to\infty$ in the general case. Our main result are sharp two-sided large-time estimates for the kernel $u_t(x,y)$. Let us first state the result and then discuss the assumptions  \eqref{A1}--\eqref{A3} needed therein.
\begin{theorem} \label{th:main_th}
    Let $L$ be the generator of a symmetric L\'evy process with L\'evy measure $\nu(dx) = \nu(x)\,dx$ and diffusion matrix $A=(a_{ij})_{i,j=1\dots, n}$, and let $V$ be a confining potential. Denote by $H = -L+V$ the Schr\"{o}dinger operator and assume \eqref{A1}--\eqref{A3} with $t_{\bee}>0$, $R_0>0$ and the profile functions $f(|x|)$ and $g(|x|)$ which control $\nu(x)$ and $V(x)$, respectively. Write $\lambda_0$ and $\varphi_0$ for the ground-state eigenvalue and eigenfunction, and $u_t(x,y)$ for the density of the operator $U_t = e^{-tH}$. There exist constants $C \geq 1$ and $R > R_0$ such that for every $t > 30 t_{\bee}$ the following assertions hold.
    \begin{enumerate}
    \item\label{th:main_th-a}
        If $|x|, |y| \leq R$, then
        \begin{align*}
            \frac{1}{C} e^{-\lambda_0 t}
            \leq u_t(x,y)
            \leq C e^{-\lambda_0 t}.
        \end{align*}

    \item\label{th:main_th-b}
        If $|x| > R$ and $|y| \leq R$, then
        \begin{align*}
            \frac{1}{C} e^{-\lambda_0 t} \frac{\nu(x)}{V(x)}
            \leq u_t(x,y)
            \leq C e^{-\lambda_0 t} \frac{\nu(x)}{V(x)};
        \end{align*}
        by symmetry, if $|x| \leq R$ and $|y| > R$, then
        \begin{align*}
            \frac{1}{C} e^{-\lambda_0 t} \frac{\nu(y)}{V(y)}
            \leq u_t(x,y)
            \leq C e^{-\lambda_0 t} \frac{\nu(y)}{V(y)}.
        \end{align*}

    \item\label{th:main_th-c}
        If $|x|, |y| > R$, then
        \begin{gather*}
            \frac{1}{C}\frac{F(\mathsf{K}t,x,y) \vee e^{-\lambda_0 t} \nu(x) \nu(y)}{V(x) V(y)}
            \leq u_t(x,y)
            \leq C\frac{F\left(\frac{t}{\mathsf{K}},x,y\right)\vee e^{-\lambda_0 t}\nu(x) \nu(y)}{V(x) V(y)},
        \end{gather*}
        where $\mathsf{K} = 4 C_6 C_7^2$ --- the constants $C_6,C_7$ are from \eqref{A3} --- and
        \begin{align*} 
            F(\tau,x,y)
            := \int_{R-1 < |z| < |x| \vee |y|} \left(f(|x-z|) \wedge 1\right) \left(f(|z-y|) \wedge 1\right)  e^{- \tau g(|z|)}\,dz.
        \end{align*}
    \end{enumerate}
\end{theorem}
Under the additional condition $\inf_{x \in \real^d} V(x) >0$, the cases~\ref{th:main_th-a} and~\ref{th:main_th-b} can be combined in a single estimate: \emph{if $|x| \leq R$ or $|y| \leq R$, then
\begin{align*}
    u_t(x,y)
    \casymp{C} e^{-\lambda_0 t} \left(1 \, \wedge \, \frac{\nu(x)}{V(x)}\right) \left(1 \, \wedge \, \frac{\nu(y)}{V(y)}\right).
\end{align*}}

\bigskip
Let us now discuss the assumptions and the set-up of Theorem~\ref{th:main_th}. Recall that a L\'evy process on $\real^d$ is a stochastic process $(X_t)_{t\geq 0}$ with values in $\real^d$, independent and stationary increments, and c\`adl\`ag (right-continuous, finite left limits) paths. It is well-known, cf.~\cite{bib:J,JSch} or~\cite{BSchW}, that a L\'evy process is a Markov process whose transition semigroup is a semigroup of convolution operators
\begin{gather*}
    P_tu(x) = \Ee u(X_t + x) = u*\tilde\mu_t(x),
    \quad
    \tilde\mu_t(dy) = \Pp(-X_t\in dy)
\end{gather*}
which is a strongly continuous contraction semigroup on $L^2 = L^2(\real^d, dx)$.  Using the Fourier transform we can describe $P_t$ as a Fourier multiplication operator
\begin{gather*}
    P_t u(x) = \mathcal{F}^{-1}\left(e^{-t\psi} \mathcal{F} u\right)(x)
\end{gather*}
with symbol (multiplier) $e^{-t\psi(\xi)}$. The semigroup $\big\{P_t:t \geq 0\big\}$ is symmetric in $L^2$ if, and only if, $X_t$ is a symmetric L\'evy process (i.e.\ $\Pp(X_t \in dy) = \Pp(-X_t\in dy)$, $t\geq 0$) which is equivalent to $e^{-t\psi}$ or $\psi$ being real. All real characteristic exponents are given  by the L\'evy--Khintchine formula
\begin{gather} \label{eq:Lchexp}
    \psi(\xi) = \xi \cdot A \xi + \int_{\real^d \setminus \left\{0\right\}} \left(1-\cos(\xi \cdot z)\right) \nu(dz), \quad \xi \in \real^d.
\end{gather}
where $A=(a_{ij})_{1\leq i,j \leq d}$ is a symmetric non-negative definite matrix, and $\nu$ is a symmetric L\'evy measure, i.e.\ a Radon measure on $\real^d \setminus \left\{0\right\}$ satisfying $\nu(E)= \nu(-E)$ and $\int_{\real^d\setminus\{0\}} (1 \wedge |z|^2) \nu(dz) < \infty$. The matrix $A$ describes the diffusion part of $(X_t)_{t\geq 0}$ while $\nu$ is the jump measure. Throughout this paper we assume that the jump activity is infinite and $\nu$ is absolutely continuous with respect to Lebesgue measure, i.e.\
\begin{align} \label{eq:nuinf}
    \nu(\real^d \setminus \left\{0\right\}) = \infty
    \quad\text{and}\quad
    \nu(dx)=\nu(x)\,dx.
\end{align}

The generator $L$ is a non-local self-adjoint pseudo-differential operator given by
\begin{align} \label{def:gen}
    \mathcal{F}[L u](\xi) = - \psi(\xi) \mathcal{F} u(\xi), \quad \xi \in \real^d, \quad u \in \mathcal D(L):=\left\{v \in L^2(\real^d): \psi \mathcal{F} v \in L^2(\real^d) \right\},
\end{align}
Prominent examples of non-local operators (and related jump processes) are \emph{fractional Laplacians} $L=-(-\Delta)^{\alpha/2}$, $\alpha \in (0,2)$ (\emph{isotropic $\alpha$-stable processes}) and \emph{quasi-relativistic operators} $L=-(-\Delta+m^{2/\alpha})^{\alpha/2} + m$, $\alpha \in (0,2)$, $m>0$ (\emph{isotropic relativistic $\alpha$-stable processes}) which play an important role in mathematical physics. These and further examples are discussed in Section~\ref{sec5}.

Under \eqref{eq:nuinf} the process $(X_t)_{t \geq 0}$ is a strong Feller process, i.e.\ $P_t$ maps bounded measurable functions into continuous functions; equivalently, this means that its one-dimensional distributions are absolutely continuous with respect to Lebesgue measure, i.e.\ there exists a transition density $p_t(x,y) = p_t(y-x)$ such that $\pr^0(X_t \in E) = \int_E p_t(x)\,dx$ for every Borel set $E \subset \real^d$, see e.g.~\cite[Th.~27.7]{bib:Sat}. Further details on the existence and regularity of transition densities can be found in~\cite{bib:KSch}.

We need the following additional regularity assumptions \eqref{A1}--\eqref{A2} for the density $\nu(x)$ of the L\'evy measure and the transition density $p_t(x)$.
\begin{itemize}
\item[\bfseries(\namedlabel{A1}{A1})] \textbf{L\'evy density.}
    There exists a profile function $f:(0,\infty) \to (0,\infty)$ such that
    \begin{enumerate}
    \item
    there is a constant $C_1 \geq 1$ such that $C_1^{-1} f(|x|) \leq \nu(x) \leq C_1 f(|x|)$ for all $x \in \real^d \setminus \left\{0\right\}$;
    \item
    $f$ is decreasing and $\lim_{r\to \infty}f(r) = 0$;
    \item
    there is a constant $C_2 \geq 1$ such that $f(r) \leq C_2 f(r+1)$ for all $r \geq 1$;
    \item
    $f$ has the \emph{direct jump property}: there exists a constant $C_3 > 0$ such that
    \begin{gather*}
        \int_{\substack{|x-y| > 1\\|y|> 1}} f(|x-y|) f(|y|)\,dy \leq C_3 f(|x|), \quad |x| \geq 1.
    \end{gather*}
    \end{enumerate}
\end{itemize}
Some parts of \eqref{A1} are redundant but we prefer to keep it that way for clarity and reference purposes. For instance, under (\ref{A1}.b) the condition (\ref{A1}.d) implies (\ref{A1}.c). Similarly, in (\ref{A1}.b), $\lim_{r\to\infty} f(r) = 0$ readily follows from the monotonicity of $f$ and (\ref{A1}.a).

The convolution property (\ref{A1}.d) is fundamental for our investigations. It has a very suggestive probabilistic interpretation: \emph{the probability to move from $0$ to $x$ in ``two large jumps in a row'' is smaller than with a ``single direct jump''.}
For this reason we call this condition the \emph{direct jump property}.\footnote{Previously \cite{KL2015,bib:KKL2018} this condition has also been called \emph{jump paring condition} but we prefer the present name as it captures the probabilistic meaning in a more concise way.}
\begin{itemize}
\item[\bfseries(\namedlabel{A2}{A2})] \textbf{Transition density of the free process.}
    The function $(t,x) \mapsto p_t(x)$ is continuous on $(0,\infty) \times \real^d$ and there exists some $t_{\bee} >0$ such that the following conditions hold.
    \begin{enumerate}
    \item
    There are constants $C_4, C_5 > 0$ such that
    \begin{gather*}
        p_t(x) \leq C_4 \left(\left[e^{C_5 t} f(|x|)\right]\wedge 1\right),
        \quad x \in \real^d \setminus \left\{0\right\},\; t \geq t_{\bee};
    \end{gather*}
    \item
    For every $r \in (0,1]$ we have
    \begin{gather*}
        \sup_{t \in (0,t_{\bee}]} \sup_{r \leq |x| \leq 2} p_t(x) < \infty.
    \end{gather*}
    \end{enumerate}
\end{itemize}
An easy-to-check sufficient condition for the time-space continuity of the density $p_t(x)$ is
\begin{gather*}
    e^{-t\psi(\xi)}\in L^1(\real^d,d\xi) \quad \text{for all\ \ } t > 0,
\end{gather*}
see Lemma \ref{lem:density} in Section \ref{sec2} and the discussion following that lemma. Notice that this condition trivially holds as soon as $\psi$ has a nondegenerate Gaussian part, i.e.\ $\det A \neq 0$ in \eqref{eq:Lchexp}. The other assumptions in \eqref{A2} govern the asymptotic behaviour of the transition density $p_t(x)$ for the free operator $L$ and they should be seen as the minimal regularity requirement for the density of the free semigroup. The upper bound on $p_t(x)$ in (\ref{A2}.a) is known for a wide range of operators $L$ whose L\'evy measures satisfy \eqref{A1}, cf.~\cite{bib:BGR, bib:KS2013, bib:KS2017, Knop, KnopKul}. Similarly, the condition (\ref{A2}.b) is a small time off-diagonal boundedness property which holds for a large class of semigroups, see e.g.~\cite[Th.~5.6 and Rem.~5.7]{bib:GS2017}. Under (\ref{A2}.a) we know that $\sup_{x \in \real^d} p_{t_{\bee}}(x) = p_{t_{\bee}}(0) < \infty$ --- this extends to all $t \geq t_{\bee}$ --- and the function $p_t(\cdot)$ is smooth for all $t > t_{\bee}$; this is a consequence of the fact that $p_t$ is the convolution of $p_{t-t_{\bee}} \in L^1(\real^d)$ and $p_{t_{\bee}} \in L^{\infty}(\real^d)$.

We will now introduce the class of potentials which we consider in this paper.
\begin{itemize}
\item[\bfseries(\namedlabel{A3}{A3})] \textbf{Confining potential.}
    Let $V \in L^{\infty}_{\loc}(\real^d)$ be such that $\lim_{|x|\to\infty} V(x) = \infty$ and assume that there exist constants $C_6 \geq 1$ and $R_0 >0$, and a profile function $g:[0,\infty) \to (0,\infty)$ such that
    \begin{enumerate}
    \item
    $g|_{[0,R_0)} \equiv 1$ and $C_6^{-1} g(|x|) \leq V(x) \leq C_6 g(|x|)$, $|x| \geq R_0$;
    \item%
        $g$ is increasing on $[R_0,\infty)$;
    \item%
        there exists a constant $C_7 \geq 1$ such that $g(r+1) \leq C_7 g(r), \quad r \geq R_0$.
    \end{enumerate}
\end{itemize}
The uniform growth condition (\ref{A3}.c) excludes profiles growing like $\exp\left(r^2\right)$ or $\exp\left(e^r\right)$, but exponentially and slower growing potentials --- for example growth orders $\log \log r$, $\log(r)^{\beta}$, $r^{\beta}$ and $e^{\beta r}$, with $\beta>0$ --- are admissible.

Under \eqref{A1}--\eqref{A3} $H=-L+V$ is well defined, bounded below, and self-adjoint in $L^2=L^2(\real^d,dx)$. Our standard reference for Schr\"{o}dinger operators is the monograph \cite{bib:DC} by Demuth and van Casteren. The corresponding Schr\"odinger semigroup $\left\{e^{-tH}: t \geq 0 \right\}$ has the following probabilistic \emph{Feynman--Kac representation}
\begin{gather*}
    e^{-tH} f(x)
    = U_t f(x)
    := \ex^x\left[e^{-\int_0^t V(X_s)\,ds} f(X_t) \right],
    \quad f \in L^2(\real^d),\; t>0,
\end{gather*}
which allows us to use methods from probability theory. Under \eqref{A3} the semigroup operators $U_t$, $t>0$, are compact and the spectrum of $H$ consists of eigenvalues of finite multiplicity without accumulation points. The ground state eigenvalue $\lambda_0 := \inf \spec(H)$ is simple and the corresponding --- unique (when normalized) and positive --- $L^2$-eigenfunction is denoted by $\varphi_0$. The operators $U_t$ have integral kernels, i.e.
\begin{gather*}
    U_t f(x)
    = \int_{\real^d} u_t(x,y) f(y)\,dy,
    \quad f \in L^2(\real^d),\; t>0,
\end{gather*}
and the kernels $u_t(\cdot,\cdot)$, $t>0$, are continuous, positive and symmetric functions on $\real^d \times \real^d$. We call the kernel $u_t(x,y)$ the \emph{Schr\"{o}dinger heat kernel}. Because of (\ref{A2}.a), $u_t(x,y)$ is a bounded function for all $t \geq t_{\bee}$. Further properties of $u_t(x,y)$ will be discussed in detail in the next section.

The conditions \eqref{A1}--\eqref{A3} are needed to prove Theorem~\ref{th:main_th}. If we make a further structural assumption on the profile function $g$, we can improve the results of Theorem~\ref{th:main_th}, splitting the estimates in two distinct scenarios: the \emph{aIUC regime} (including the \emph{IUC regime}) and the \emph{non-aIUC regime}, see Remark~\ref{rem:regimes}.
\begin{itemize}
\item[\bfseries(\namedlabel{A4}{A4})]
    $V$ is a potential satisfying \eqref{A3} with the profile $g$ and $R_0 >0$ such that $f(R_0) < 1$ and
    \begin{align} \label{eq:h_log}
        g(r) = h\left(|\log f(r)|\right), \quad r \geq R_0,
    \end{align}
    for some increasing function $h:[|\log f(R_0)|,\infty) \to (0,\infty)$ such that $h(s)/s$ is monotone.
\end{itemize}
Examples of such profiles $h$ can be found among functions which are regularly varying at infinity, see \cite{bib:BGT}. Let us remark that \eqref{A4} is, when we compare it with existing results on the asymptotic behaviour of aIUC Schr\"odinger semigroups~\cite{bib:KL15, bib:KKL2018}, a very natural condition.

The large time estimates of the heat kernel $u_t(x,y)$ in the the aIUC vs.\ the non-aIUC regime are substantially different. This is due to the intricate asymptotic behaviour of the function $F(\tau,x,y)$. In the non-aIUC regime the following result holds true (Corollary~\ref{cor:cor_prog}): \emph{For every confining potential -- no matter how slowly $V$ grows at infinity -- there is an increasing function $r:(0,\infty) \to (0,\infty]$ such that $\lim_{t\to\infty}r(t) = \infty$} (cf.~Lemma~\ref{lem:non-aIUC}) \emph{and such that the following estimate holds: There is a constant $C \geq 1$ such that for sufficiently large values of $t$ we have}
\begin{gather*}
    \frac{1}{C} e^{-\lambda_0 t} \frac{f(|x|)}{g(|x|)} \frac{f(|y|)}{g(|y|)}
    \leq u_t(x,y)
    \leq C e^{-\lambda_0 t} \frac{f(|x|)}{g(|x|)}\frac{f(|y|)}{g(|y|)},
    \quad |x|,|y| > R, \;  |x| \wedge |y| < r(t).
\end{gather*}
These estimates are equivalent to saying that there is a constant $\widetilde{C} \geq 1 $ such that for sufficiently large values of $t$ we have
\begin{gather*}
    \frac{1}{\widetilde{C}} e^{-\lambda_0 t} \varphi_0(x) \varphi_0(y)
    \leq u_t(x,y)
    \leq \widetilde{C} e^{-\lambda_0 t} \varphi_0(x) \varphi_0(y),
    \quad |x| \wedge |y| < r(t).
\end{gather*}
The estimates for $u_t(x,y)$ are essentially different if $|x|,|y| > r(t)$.  This means that the regularity of a non-aIUC Schr\"odinger semigroup improves as soon as the time parameter $t$ increases; note that the constants $C, \widetilde{C}$ do not depend on $t$. We believe that this surprising property has not been observed before. In analogy to the \emph{asymptotic} IUC property, we propose to call this property \emph{progressive intrinsic ultracontractivity}, \emph{pIUC}, for short. It seems that pIUC is a regularity property for compact semigroup, in general, and merits a formal definition.

\begin{definition}
\label{def:piuc}
    Let $\left\{U_t:t \geq 0\right\}$ be a semigroup of compact operators on $L^2(\real^d)$ with integral kernels $u_t(x,y)$, ground state eigenvalue $\lambda_0$ and ground state eigenfunction $\varphi_0$. The semigroup $\left\{U_t:t \geq 0\right\}$ is said to be
    \begin{enumerate}
    \item\label{piuc-a}
    \emph{intrinsically ultracontractive} (IUC) if for every $t>0$ there exists a constant $C>0$ such that
    \begin{gather*}
        u_t(x,y) \casymp{C} e^{-\lambda_0 t} \varphi_0(x) \varphi_0(y),
        \quad  x, y \in \real^d, \;t >0.
    \end{gather*}

    \item\label{piuc-b}	
    \emph{asymptotically intrinsically ultracontractive} (aIUC) if there exist some $t_0>0$ and a constant $C>0$ such that
    \begin{gather*}
        u_t(x,y) \casymp{C} e^{-\lambda_0 t} \varphi_0(x) \varphi_0(y),
        \quad x, y \in \real^d, \;t \geq t_0.
    \end{gather*}

    \item\label{piuc-c}
    \emph{progressively intrinsically ultracontractive} \textup{(pIUC)} if there exist some $t_0>0$, an increasing function $r:[t_0,\infty) \to (0,\infty]$ such that $r(t) \to \infty$, and a constant $C>0$ such that
    \begin{gather*}
        u_t(x,y) \casymp{C} e^{-\lambda_0 t} \varphi_0(x) \varphi_0(y),
        \quad |x| \wedge |y| < r(t), \;t \geq t_0.
    \end{gather*}
    \end{enumerate}
\end{definition}
Note that IUC always implies aIUC, and aIUC always implies pIUC (with threshold function $r \equiv \infty)$.

Our paper is organized in the following way. Section~\ref{sec2} contains some basic probabilistic potential theory which is needed in the subsequent sections. In~\ref{sec3:djp} we discuss the properties of the profile function $f$ from \eqref{A1}, in particular we provide some sufficient conditions such that the direct jump property (\ref{A1}.d) holds. Section~\ref{sec3:decomposition} is about the basic decomposition of the trajectories of the free process: we use this to obtain estimates of the localized (in space) Feynman--Kac representation. These bounds are essential for the upper and lower estimates in Sections~\ref{sec4:upper} and~\ref{sec4:lower}. They are combined to give sharp two-sided estimates for $u_t(x,y)$, see Section~\ref{sec4:sharp}. Theorem~\ref{th:th1} is an extended version of our main Theorem~\ref{th:main_th}. Based on these estimates we discuss several applications: we study the decay properties of the functions $U_t \I_{\real^d}(x)$ (Section~\ref{sec4:Ut}),
we show that the present estimates are potent enough to recover known results on aIUC (Section~\ref{sec4:aIUC}), and finally we show (Section~\ref{sec4:spectre}) that --- for $t\gg 1$ --- the notions of ``operator with finite heat content'', ``trace-class operator'' and ``Hilbert--Schmidt operator'' coincide in this setting, and they are equivalent to the condition that $\int_{|x|>R} e^{-t V(x)} \,dx < \infty$, for some $R, t>0$. Our second main result is presented as Corollary~\ref{cor:cor_prog} in Section~\ref{sec5:piuc}. This is about improved heat-kernel estimates in the pIUC regime. This is only possible if we know about the dependence of the growth at infinity of $f$ (resp.\ the jump density $\nu$) and $g$ (resp.\ the potential $V$). Here we need the additional assumption \eqref{A4}, see Section~\ref{sec5:log}. The last two sections contain examples: Section~\ref{sec5:doubling} is about doubling L\'evy measures with $f$ of polynomial type $f(r) = r^{-d-\alpha}(e\vee r)^{-\gamma}$ while Section~\ref{sec5:exp} considers exponentially decaying L\'evy measures with $f$ of the form $f(r) = r^{-\gamma} e^{-\kappa r}$, $r\geq 1$.

\begin{notation}
    Throughout the paper lower case letters $c,c_1,c_2\dots$ denote generic constants; within a proof we indicate changes to constants by increasing their running index. Upper case letters $C_1,\dots,C_7$ (they appear in the assumptions \eqref{A1}--\eqref{A3}) and $C_{n.m}$ ($n.m$ refers to the Theorem, Lemma etc.\ where $C_{n.m}$ appears for the first time) denote important constants. This is for cross-referencing and to help keeping track of the dependence of constants in our calculations. The constant $\tilde C_2$ is from \eqref{eq:loc_comp_f1} on page~\pageref{eq:loc_comp_f1}; it serves as an alternative of the constant $C_2$ in \eqref{A1} and it determines the growth of the functions of class $\cC$ defined in \eqref{eq:W_prop}.

    Our basic assumptions \eqref{A1}, \eqref{A2}, \eqref{A3} and \eqref{A4} can be found on pages~\pageref{A1}--\pageref{A4}. The constant $t_\bee$ is from \eqref{A2} and $R_0$ is from \eqref{A3};

    Two-sided estimates between functions are sometimes indicated by
    \begin{gather*}
        f(x)\asymp g(x),\; x\in A
        \iff
        \exists C\geq 0 \;\forall x\in A \::\: C^{-1}f(x)\leq g(x)\leq C f(x).
    \end{gather*}
    The notation $f(x)\casymp{C}g(x)$ is used to highlight the comparison constant $C$.

    By $\lambda_0$ and $\varphi_0$ we denote the ground-state eigenvalue and eigenfunction, see page~\pageref{def:piuc}. The Fourier transform and its inverse are defined as
    \begin{gather*}
        \mathcal{F}u(\xi) = (2\pi)^{-d}\int_{\real^d} e^{-i\xi x} u(x)\,dx
        \et
        \mathcal{F}^{-1}v(x) = \int_{\real^d} e^{ix\xi} v(\xi)\,d\xi.
    \end{gather*}
    As usual, we write  $\ex(Y; A,B) := \int_{A\cap B} Y\,d\Pp$, $|B|$ denotes the Lebesgue measure of a Borel set $B\subset\real^d$, $a\wedge b$ and $a\vee b$ are the minimum and maximum of $a$ and $b$.
\end{notation}

\section{Preliminaries}\label{sec2}

We begin with a sufficient condition for the joint continuity of the functions $p_t(x)$ on the sets $(t_0,\infty) \times \real^d$, $t_0>0$.
\begin{lemma}\label{lem:density}
    Let $(X_t)_{t\geq 0}$ be a L\'evy process with values in $\real^d$ and characteristic exponent $\psi:\real^d\to \mathds C$. If there exists some $t_0>0$ such that $e^{-t\psi(\xi)}\in L^1(\real^d,d\xi)$ for all $t\geq t_0$, then $X_t$ admits for $t\geq t_0$ a probability density $p_t(x)$ such that $(t,x)\mapsto p_t(x)$ is continuous for all $(t,x)\in (t_0,\infty)\times\real^d$.
\end{lemma}
\begin{proof}
    Fix $\epsilon>0$. Since $\xi\mapsto e^{-t_0\psi(\xi)}$ is integrable and $|e^{-t_0\psi}| = e^{-t_0\re\psi}$, we can use the dominated convergence theorem to ensure that for some $\delta>0$
    \begin{gather*}
        2(2\pi)^{-d} \int_{\real^d} e^{-t_0\re\psi(\xi)} \left|\sin\tfrac 12 \xi\cdot(x-y)\right| d\xi
        < \epsilon \text{\ \ for all $x,y\in\real^d$, $|x-y|<\delta$},\\
        (2\pi)^{-d} \int_{\real^d} e^{-t_0\re\psi(\xi)} \left| e^{-(t-t_0)\re\psi(\xi)} - e^{-(u-t_0)\re\psi(\xi)}\right| d\xi
        < \epsilon \text{\ \ for all $t,u > 0$, $|t-u|<\delta$}.
    \end{gather*}
    Using Fourier inversion we get
    \begin{gather*}
        p_t(x)
        = (2\pi)^{-d}\int_{\real^d} e^{-t\psi(\xi)} e^{-ix\cdot\xi}\,d\xi,\quad t\geq t_0,\; x\in\real^d,
    \end{gather*}
    which shows the existence of a transition density. For $x,y\in\real^d$ and $t>t_0$ we get, on the one hand
    \begin{align*}
        \left|p_t(x)-p_t(y)\right|
        &\leq (2\pi)^{-d} \int_{\real^d} e^{-t\re\psi(\xi)} \left|e^{-ix\cdot\xi}-e^{-iy\cdot\xi}\right| d\xi\\
        &= 2(2\pi)^{-d} \int_{\real^d} e^{-t\re\psi(\xi)} \left|\sin \tfrac 12 (x-y)\cdot\xi\right| d\xi\\
        &\leq 2(2\pi)^{-d} \int_{\real^d} e^{-t_0\re\psi(\xi)} \left|\sin \tfrac 12 (x-y)\cdot\xi\right| d\xi;
    \end{align*}
    on the other hand, we have for $y\in\real^d$ and $t,u>t_0$
    \begin{align*}
        \left|p_t(y)-p_u(y)\right|
        \leq (2\pi)^{-d} \int_{\real^d} e^{-t_0\re\psi(\xi)} \left|e^{-(t-t_0)\re\psi(\xi)} - e^{-(u-t_0)\re\psi(\xi)}\right| d\xi.
    \end{align*}
    If $x,y\in\real^d$ and $t,u>t_0$ satisfy $|x-y|+|t-u|<\delta$, we finally get
    \begin{gather*}
        \left|p_t(x) - p_u(y)\right|
        \leq \sup_{t\geq t_0} \left|p_t(x) - p_t(y)\right| + \sup_{y\in\real^d} \left|p_t(y) - p_u(y)\right|
        < 2\epsilon,
    \end{gather*}
    proving joint continuity of $(t,x)\mapsto p_t(x)$ on $(t_0,\infty)\times\real^d$.
\end{proof}

\begin{remark}\label{rem:density}
    The assumption $e^{-t\psi(\xi)}\in L^1(d\xi)$ for $t \geq t_0$ already appears in \cite{bib:KSch}. A sufficient condition for this assumption is the Hartman-Wintner condition
    \begin{gather}
    \tag{$\mathrm{HW}_{1/t_0}$}
    \liminf_{|\xi|\to\infty} \frac{\re\psi(\xi)}{\log(1+|\xi|)} > \frac d{t_0}
    \end{gather}
    which stipulates that $\psi$ grows at infinity at least logarithmically.

    Note that the (one-sided) one-dimensional Gamma process has the exponent $\psi(\xi) = \frac 12\log(1+\xi^2) + i\arctan\xi$ and the transition density $p_t(x) = \Gamma(t)^{-1} x^{t-1} e^{-x}$, $x>0$. Clearly, $p_t(x)$ fails to be continuous on $(0,1)\times\real$, i.e.\ logarithmic growth of $\psi$ seems to be a rather sharp condition for the joint continuity of $(t,x)\mapsto p_t(x)$. Using this, one can also give an example of a symmetric L\'evy process on $\real$ satisfying our basic assumption \eqref{A1} but with a density which fails to be time-space continuous for small values of $t$. A similar picture is true for the one-dimensional symmetric Gamma process, whose transition density is given by $q_t(x) = \int_0^{\infty} \frac{1}{\sqrt{4\pi s}} e^{-|x|^2/(4s)}\, p_t(s)\, ds$. For details and further references see
    \cite[Example 2.3]{bib:KSch}.

    If we assume $\mathrm{HW}_\infty$, i.e.\ $t_0=0$, the density $p_t(x)$ is already smooth in the variable $x$, cf.\ \cite[Theorem~2.1]{bib:KSch}, indicating that a Hartman--Wintner condition cannot be optimal.
\end{remark}

We now collect a few basic properties of the Schr\"{o}dinger semigroup $\left\{U_t: t \geq 0\right\}$ and facts from potential theory for the free process $(X_t)_{t \geq 0}$. We begin with some fundamental properties of the Schr\"odinger heat kernel $u_t(x,y)$ which will be needed in the sequel.

\begin{lemma}\label{new}
    Let $H = -L+V$ be the Schr\"{o}dinger operator with confining potential $V$ such that \eqref{A1}--\eqref{A3} hold. Denote by $u_t(x,y)$ the density of the operator $U_t = e^{-tH}$.
\begin{enumerate}
\item\label{new-a}
    For every $x,y \in \real^d$ and $t>0$ we have
    \begin{gather*}
        u_t(x,y) = \lim_{s \uparrow t} \ex^x\left[e^{-\int_0^s V(X_r)\,dr} \, p_{t-s}(y-X_s)\right].
    \end{gather*}

\item\label{new-b}
    For fixed $t>0$, $u_t(\cdot,\cdot)$ is a continuous and symmetric function on $\real^d \times \real^d$.

\item\label{new-c}
    For every $x,y \in \real^d$ and $t>0$ we have
    \begin{align}\label{eq:u_by_p}
        0<u_t(x,y) \leq e^{C_{\ref{new}} t}\, p_t(y-x),
    \end{align}
    where $C_{\ref{new}}:= - (\inf_{x \in \real^d} V(x) \wedge 0)$; in particular, all semigroup operators $U_t$, $t>0$, are positivity improving.

\item\label{new-d}
    For every $t \geq t_{\bee}$ we have $\sup_{x,y \in \real^d} u_t(x,y) < \infty$. In particular, $U_t:L^2(\real^d) \to L^{\infty}(\real^d)$ is a bounded operator for all $t \geq t_{\bee}$, that is the semigroup $\left\{U_t: t \geq 0\right\}$ is ultracontractive for $t \geq t_{\bee}$.
\end{enumerate}
\end{lemma}
Lemma~\ref{new} is a standard result and we refer for its proof and further details on Feynman--Kac semigropus to the monographs~\cite{bib:DC, bib:CZ}. It shows that the kernel $u_t(x,y)$ inherits its basic regularity properties from the transition densities $p_t(y-x)$ of the free L\'evy process.

Assumption \eqref{A3} guarantees that all semigroup operators $U_t$, $t>0$, are compact operators on $L^2(\real^d)$. In particular, there is a ground state, i.e.\ $\lambda_0 := \inf \spec(H)$ is an eigenvalue with multiplicity one and there exists a unique eigenfunction $H \varphi_0 = \lambda_0 \varphi_0$ --- hence $U_t \varphi_0 = e^{-\lambda_0 t} \varphi_0$, $t>0$ --- where $\varphi_0 \in L^2(\real^d)$ and $\left\|\varphi_0\right\|_2=1$; Lemma~\ref{new}.\ref{new-d} ensures that $\varphi_0 \in L^{\infty}(\real^d)$. Moreover, it is known that $\left\{U_t: t \geq 0\right\}$ has the strong Feller property, i.e.\ $U_t (L^{\infty}(\real^d)) \subset C_b(\real^d)$ for $t>0$, which implies that $\varphi_0$ has a version in $C_b(\real^d)$ and $U_t \varphi_0(x) = e^{-\lambda_0 t} \varphi_0(x)$ has a pointwise meaning. By Lemma~\ref{new}.\ref{new-c}, we even have $\varphi_0(x) > 0$ for all $x \in \real^d$.

We denote by $p_D(t,x,y)$ the transition density of the free process $(X_t)_{t \geq 0}$ killed upon exiting a bounded, open set $D \subset \real^d$; it is given by the Dynkin-Hunt formula
\begin{align}\label{eq:HuntF}
    p_D(t,x,y)
    = p_t(y-x) - \ex^x\left[p_{t-\tau_D}(y-X_{\tau_D});\; \tau_D < t\right],
    \quad x,y \in D,
\end{align}
where
\begin{gather*}
    \tau_D = \inf\left\{t \geq 0: X_t \notin D\right\}
\end{gather*}
is the first exit time of the process $X$ from the set $D$; as usual, we set $p_D(t,x,y) = 0$ if $x \notin D$ or $y \notin D$. Hence,
\begin{align} \label{eq:killed_sem}
    \ex^x\left[f(X_t);\; t < \tau_D\right] = \int_D f(y) p_D(t,x,y)  \,dy,  \quad x \in D, \; t>0,
\end{align}
for every bounded or nonnegative Borel function $f$ on $D$.  The Green function of the process $X$ in $D$ is given by $G_D(x,y)= \int_0^{\infty} p_D(t,x,y) \,dt$. If $D=B_r(0)$, $r>0$, then we denote by $\mu_0(r)$ and $\psi_{0,r} \in L^2(B_r(0))$ the ground state eigenvalue and eigenfunction of the process killed upon leaving $B_r(0)$. It is known that $\mu_0(r)>0$ is the smallest positive number and $\psi_{0,r}$ is the unique $L^2$-function with $\left\|\psi_{0,r}\right\|_2=1$ such that
\begin{gather*}
    \int_{B_r(0)} p_{B_r(0)}(t,x,y) \psi_{0,r}(y)\,dy
    = e^{-\mu_0(r) t} \psi_{0,r}(x), \quad t>0,\; x \in B_r(0).
\end{gather*}
This equality entails that $\psi_{0,r}$ is bounded on $B_r(0)$, continuous around $0$ (see e.g.\ \cite[proof of Th.~3.4: Claim 1]{bib:SchW}) and $\psi_{0,r}(0) > 0$.

The kernel $\nu(z-X_{s-}(\omega))\,dz\,ds$ is the L\'evy system for $(X_t)_{t \geq 0}$; it is uniquely characterized by the identity
\begin{align} \label{eq:LSF}
    \ex^x \sum_{\substack{s \in (0,t]\\|\Delta X_s| \neq 0}} f(s,X_{s-},X_s)
    = \ex^x \int_0^t \int_{\real^d} f(s,X_{s-},z) \nu(z-X_{s-}) \,dz\,ds,
\end{align}
where $f:[0,\infty) \times \real^d \times \real^d \to [0,\infty)$. If we use $f(s,y,z) = \I_I(s) \I_E(y) \I_F(z)$, where $I$ is a bounded interval, and $E \subset D$, $F \subset D^c$ are Borel subsets of $\real^d$ with $\dist(E, F) >0$, the functional
\begin{gather*}
    M_t := \sum_{s \in (0,t]} f(s,X_{s-},X_s) - \int_0^t \int_{\real^d} f(s,X_{s-},z) \nu(z-X_{s-}) \,dz\,ds
\end{gather*}
is a uniformly integrable martingale, see e.g. \cite[Chapter II.1d, II.2a, II.4c]{bib:JS}. By a stopping argument we get
\begin{gather}\label{eq:IkWa}
    \pr^x(\tau_D \in dt, X_{\tau_D - } \in dy, X_{\tau_D} \in dz)
    = p_D(t,x,y) \,dt \, \I_{\{|z-y|>0\}}(y,z)\,\nu(z-y)\,dy \,dz,
\end{gather}
on $(0,\infty) \times D \times (\overline D)^c$. This is usually called \emph{Ikeda-Watanabe formula}, see~\cite[Th.~1]{bib:IW} for the original version. We will use this formula mainly in the following setting: for every $\vartheta>0$ and every bounded or non-negative Borel function $h$ on $\real^d$ such that $\dist(\supp h, D) >0$, one has
\begin{align}\label{eq:IWF}
    \ex^x\left[e^{-\vartheta \tau_D} h(X_{\tau_D})\right]
    = \int_D \int_0^{\infty} e^{-\vartheta t}\, p_D(t,x,y) \,dt \int_{D^c} h(z) \nu(z-y) \,dz\,dy,
    \quad x \in D.
\end{align}

We will also need the concept of $(X,\vartheta)$-harmonic functions, $\vartheta >0$. A Borel function $f$ on $\real^d$ is called \emph{$(X,\vartheta)$-harmonic} in an open set $D \subset \real^d$ if
\begin{align}\label{def:harm}
    f(x) = \ex^x\left[e^{-\vartheta \tau_U} f(X_{\tau_U});\; \tau_U < \infty\right], \quad x \in U,
\end{align}
for every open (possibly unbounded) set $U$ such that $\overline{U}\subset D$; $f$ is called \emph{regular $(X,\vartheta)$-harmonic} in $D$, if \eqref{def:harm} holds for $U=D$.
We will always assume that the expected value in \eqref{def:harm} is absolutely convergent. By the strong Markov property every regular $(X,\vartheta)$-harmonic function in $D$ is $(X,\vartheta)$-harmonic in $D$.

The next lemma provides a uniform estimate for $(X,\vartheta)$-harmonic functions which is often called \emph{boundary Harnack inequality}, see~\cite[Lem.~3.2(b) and Th.~3.5]{bib:BKK}.
\begin{lemma}\label{lem:bhi}
    Assume \textup{(\ref{A1}.a,b,c)} and \eqref{A2}. There exists a universal constant $C>0$ such that
    \begin{gather*}
        h(y) \leq \frac{C}{\vartheta} \int_{B_{1/2}^c(x)} h(z) \nu(x-z) \,dz, \quad y \in D \cap B_{1/4}(x)
    \end{gather*}
    holds for all $x \in \real^d$, all open sets $D \subset B_1(x)$ and any nonnegative, regular $(X,\vartheta)$-harmonic function $h$ in $D$ such that $h$ vanishes in $B_1(x) \setminus D$.
\end{lemma}

\begin{proof}
This result follows from a combination of Lemma 3.2, Theorem 3.5 and the discussion in Example 3.9 in~\cite{bib:BKK}. We only need to check the assumptions (A)--(D) in that paper. Since $X$ is a symmetric L\'evy process, (A)--(C) always hold. In order to justify (D), let $B = B_R(0)$, $0<r < R \leq 1$ and $x,y \in B$ be such that $|x-y|>r$. We have
\begin{align*}
    G_B(x,y)
    = \int_0^{\infty} p_B(t,x,y) \,dt
    &= \int_0^{t_{\bee}} p_B(t,x,y) \,dt + \int_0^{\infty} p_B(t+t_{\bee},x,y) \,dt \\
    &\leq \int_0^{t_{\bee}} p_t(y-x) \,dt + \int_0^{\infty} \int_B p_B(t_{\bee},x,z)p_B(t,z,y) \,dz \,dt \\
    &\leq t_{\bee} \sup_{t \in (0,t_{\bee}]} \sup_{r<|x| \leq 2} p_t(x) + \sup_{x \in \real^d} p_{t_{\bee}}(x) \int_0^{\infty} \pr^y(\tau_B>t) \,dt.
\end{align*}
Observe that
\begin{gather*}
    \int_0^{\infty} \pr^y(\tau_B > t) \,dt
    = \ex^y \tau_B
    \leq \ex^y \tau_{B_{2R}(y)}
    = \ex^0 \tau_{B_{2R}(0)}, \quad y \in B.
\end{gather*}
By~\cite[Rem. 4.8]{bib:Sch} the last mean exit time is finite. Thus, under \eqref{A2}, we have
\begin{gather*}
    \sup_{x, y \in B, |x-y| > r} G_{B}(x,y) < \infty,
\end{gather*}
for every $0<r<R$. This proves (D) from~\cite{bib:BKK}.
\end{proof}

Finally, we will need the following technical estimate.
\begin{lemma}\label{lem:pt_below}
    Assume \textup{(\ref{A1}.a,b,c)}. For every $t>0$ there exists a constant $C_{\ref{lem:pt_below}}>0$ such that $p_t(x) \geq C_{\ref{lem:pt_below}} \nu(x)$ for every $|x| \geq 1$.
\end{lemma}
\begin{proof}
    Fix $t>0$ and denote by $\mu_{1,t}(dx)$ and $\mu_{2,t}(dx)$ the measures determined by their characteristic functions (inverse Fourier transforms)
    \begin{align*}
    \mathcal{F}^{-1}\mu_{1,t}(\xi)
    &= \exp\left(-t \int_{\real^d \setminus \left\{0\right\}} \left(1-\cos(\xi \cdot y)\right) \nu^{\circ}_{1/2}(y)dy\right),
    \quad \xi \in \real^d,
    \intertext{and}
    \mathcal{F}^{-1}\mu_{2,t}(\xi)
    &=\exp\left(-t \int \left(1-\cos(\xi \cdot y)\right)\nu_{1/2}(y)dy\right),
    \quad \xi \in \real^d,
    \end{align*}
    with
    \begin{gather*}
        \nu^{\circ}_{1/2}(x) := \nu(x) \I_{B_{1/2}(0)}(x)
        \et
        \nu_{1/2}(x) := \nu(x) \I_{\real^d \setminus B_{1/2}(0)}(x).
    \end{gather*}
    Because of \eqref{A1} both measures are non-degenerate. Recall that
    \begin{align} \label{eq:poisson}
        \mu_{2,t}(dx)
        = \exp\left[t(\nu_{1/2}-|\nu_{1/2}|\delta_0)\right](dx)
        = e^{-t|\nu_{1/2}|}\delta_{0}(dx) + p_{2,t}(x)\,dx
    \end{align}
    with
    \begin{gather*}
        p_{2,t}(x) := e^{-t|\nu_{1/2}|} \sum_{n=1}^\infty \frac{t^n\nu_{1/2}^{n*}(x)}{n!};
    \end{gather*}
    $\nu_{1/2}^{n*}(x)$ denotes the density of the $n$-fold convolution $\nu_{1/2}^{n*}(dx)$. Moreover, due to~\cite[Th.~27.7]{bib:Sat} the measure $\mu_{1,t}(dx)$ is absolutely continuous with respect to Lebesgue measure. We denote the corresponding density by $p_{1,t}(x)$.

    Let $A\in\real^{n\times n}$, be a symmetric, positive semi-definite matrix and denote by $\gamma_t(dx)$ the Gaussian measure with characteristic function
    \begin{gather*}
        \mathcal{F}^{-1}\gamma_t(\xi)
        = \exp\left(- t \, (\xi \cdot A \xi) \right),
        \quad \xi \in \real^d.
    \end{gather*}
    Due to \eqref{eq:Lchexp} and \eqref{eq:poisson}, the density $p_t$ is of the form
    \begin{gather*}
        p_t(x) = e^{-t|\nu_{1/2}|} (p_{1,t}*\gamma_t)(x) + (p_{1,t}*p_{2,t}*\gamma_t)(x).
    \end{gather*}
    Assume that $|x| \geq 1$. By the above representation and (\ref{A1}.b,c) we get
    \begin{align*}
    p_t(x)
    &\geq \int_{\real^d} p_{2,t}(x-y) \int_{\real^d} p_{1,t}(y-z) \,\gamma_t(dz)\,dy \\
    &\geq t e^{-t|\nu_{1/2}|} \int_{\real^d} \nu_{1/2}(x-y) \int_{\real^d} p_{1,t}(y-z) \,\gamma_t(dz)\, dy \\
    &\geq c_1 \int_{|y|<1/2} \nu(x-y) \int_{\real^d} p_{1,t}(y-z) \,\gamma_t(dz)\,dy \\
    &\geq c_2 \nu(x) \int_{|y|<1/2} \int_{\real^d} p_{1,t}(y-z) \,\gamma_t(dz)\,dy.
    \end{align*}
    Since $\int_{|y|<1/2} \int_{\real^d} p_{1,t}(y-z) \,\gamma_t(dz)\,dy > 0$, the claimed bound follows.
\end{proof}

\section{Structure and estimates of large jumps of the process}\label{sec3}

\subsection{Properties of the profile function \boldmath$f$\unboldmath}\label{sec3:djp}
Sometimes it is convenient to replace the profile function $f(r)$, for large values of $r$, by its truncation
\begin{gather*}
    f_1 = f \wedge 1.
\end{gather*}
In this section we are going to show that $f_1$ still enjoys the basic assumptions (\ref{A1}.b,c,d).

If $f$ satisfies (\ref{A1}.b,c), then so does $f_1$: it is again a decreasing function and there exists a constant $\widetilde C_2 \geq C_2$ such that the following uniform growth condition holds
\begin{align} \label{eq:loc_comp_f1}
    f_1(r) \leq \widetilde C_2 f_1(r+1), \quad r >0.
\end{align}
Note that, under (\ref{A1}.a,b), the conditions (\ref{A1}.c) and \eqref{eq:loc_comp_f1} are equivalent.
\begin{lemma} \label{lem:equivalent-djp}
    Let $f$ be as in \textup{(\ref{A1})}. Condition \textup{(\ref{A1}.d)} can be replaced by the following equivalent condition: there exists a uniform constant $C_{\ref{lem:equivalent-djp}} > 0$ such that
    \begin{align} \label{eq:conv_f1}
        \int_{\real^d} f_1(|x-z|) f_1(|z|) \,dz \leq C_{\ref{lem:equivalent-djp}} \, f_1(|x|), \quad x \in \real^d.
    \end{align}
    In particular, \eqref{eq:conv_f1} implies
    \begin{align} \label{eq:conv_f1_pw}
        f_1(|x-w|) f_1(|w|) \leq \frac{(\widetilde C_2)^2 C_{\ref{lem:equivalent-djp}}}{|B_1(0)|} \, f_1(|x|), \quad x,w \in \real^d,
    \end{align}
    and there exists a constant $\widetilde C_{\ref{lem:equivalent-djp}} > 0$ such that
    \begin{align} \label{eq:prod_by_diff}
        f(|x|) f(|y|) \leq \widetilde C_{\ref{lem:equivalent-djp}} f_1(|x-y|), \quad |x|,|y| > 1.
    \end{align}
\end{lemma}
\begin{proof}
Let us first establish the equivalence of \eqref{eq:conv_f1} and (\ref{A1}.d). Assume that \eqref{eq:conv_f1} holds. Since $f$ is decreasing, we have
\begin{align} \label{eq:f-f_1}
    c f(|z|) \leq f_1(|z|) \leq f(|z|), \quad |z| \geq 1,
\end{align}
where $c:= 1/(1 \vee f(1))$. Therefore, \eqref{eq:conv_f1} implies
\begin{align*}
    c^2 \int_{\substack{|x-y|>1 \\ |y|>1}} f(|x-y|)f(|y|)\,dy
    \leq \int_{\real^d} f_1(|x-z|) f_1(|z|) \, dz
    \leq C_{\ref{lem:equivalent-djp}}  f_1(|x|) \leq f(|x|), \quad |x| \geq 1,
\end{align*}
and (\ref{A1}.d) follows.

In order to see the opposite implication, we note that (\ref{A1}.a) implies $\int_{|y|>1} f(|y|)\,dy < \infty$, hence $\int_{\real^d} f_1(|y|)\,dy < \infty$. If $|x| \leq 1$, then
\begin{gather*}
    \int_{\real^d} f_1(|x-z|) f_1(|z|) \,dz
    \leq \frac{\int_{\real^d} f_1(|y|)\,dy}{f_1(1)} f_1(1)
    \leq c_1 f_1(|x|).
\end{gather*}
Moreover, \eqref{eq:loc_comp_f1} shows for $|x| >1$
\begin{gather*}
    \int_{|z-x|\leq 1} f_1(|x-z|) f_1(|z|) \,dz
    = \int_{|z|\leq 1} f_1(|z|) f_1(|x-z|) \,dz
    \leq c_2 f_1(|x|), \quad x \in \real^d.
\end{gather*}
Finally, combining \eqref{eq:f-f_1} and (\ref{A1}.d) we see that there exists a constant $c_3>0$ such that
\begin{gather*}
    \int_{\substack{|z-x|> 1 \\ |z|>1}} f_1(|x-z|) f_1(|z|) \,dz
    \leq c_3 f_1(|x|), \quad |x|>1,
\end{gather*}
and \eqref{eq:conv_f1} follows.

The inequality \eqref{eq:conv_f1_pw} is a direct consequence of \eqref{eq:conv_f1} and \eqref{eq:loc_comp_f1}:
\begin{align*}
    \frac{|B_1(w)|}{(\widetilde C_2)^2} f_1(|x-w|) f_1(|w|)
    &\leq \int_{B_1(w)} f_1(|x-z|) f_1(|z|) \,dz \\
    &\leq \int_{\real^d} f_1(|x-z|) f_1(|z|) \,dz
    \leq C_{\ref{lem:equivalent-djp}} \, f_1(|x|), \quad x, w \in \real^d.
\end{align*}
Finally, \eqref{eq:prod_by_diff} follows from \eqref{eq:conv_f1_pw} and \eqref{eq:f-f_1}.
\end{proof}

The next lemma gives simple sufficient conditions under which the convolution condition (\ref{A1}.d) holds. Recall that a decreasing function $f$ satisfies the \emph{doubling property}, if there exists a constant $C \geq 1$ such that $f(r) \leq C f(2r)$ for all $r > 0$.\footnote{Since the doubling property is only used in connection with the direct jump property (\ref{A1}.d) it is, in fact, enough to require the doubling property for large $r$, e.g.\ $r>\frac 12$.}
\begin{lemma} \label{lem:sufficient-djp}
    Let $f:(0,\infty) \to (0,\infty)$ be a decreasing function. Under each of the following conditions, $f$ satisfies \textup{(\ref{A1}.d)}.
    \begin{enumerate}
    \item\label{sufficient-djp-a}
        \textbf{\upshape Doubling profiles:} $f$ has the doubling property and $ f|_{(1,\infty)} \in L^1((1,\infty),r^{d-1}\,dr)$.

    \item\label{sufficient-djp-b}
        \textbf{\upshape Tempered profiles:} $f(r) = e^{-mr} h(r)$ for some $m >0$ and $h:(0,\infty) \to (0,\infty)$ is a decreasing function with doubling property and $h|_{(1,\infty)} \in L^1((1,\infty),r^{d-1}\,dr)$.

    \item\label{sufficient-djp-c}
        \textbf{\upshape Log-convex profiles:} $f$ is log-convex \footnote{By log-convexity, $f'$ exists Lebesgue almost everywhere.} on $(1,\infty)$
        \begin{align} \label{eq:int-cond}
            \sup_{|x| \geq 1} \int_{|y|>1} e^{-\frac{f'(|y|)}{f(|y|)} \frac{x \cdot y }{|x|} } f(|y|)\,dy
            = \int_{|y|>1} e^{-\frac{f'(|y|)}{f(|y|)} y_1} f(|y|)\,dy
            < \infty.
        \end{align}
    \end{enumerate}
\end{lemma}
\begin{proof}
\ref{sufficient-djp-a}
Since $y$ and $x-y$ play symmetric roles in the integral in (\ref{A1}.d), we have
\begin{align*}
    \int_{\substack{|x-y|>1 \\ |y|>1}} f(|x-y|)f(|y|)\,dy
    &= 2 \int_{\substack{|x-y|>1, |y|>1 \\ |y| < |x-y|}} f(|x-y|)f(|y|)\,dy.
\intertext{From $|x-y| > |y| = |y-x+x| \geq |x| - |x-y|$, we get $|x-y| > \frac 12|x|$. The monotonicity and doubling property of $f$ show for all $|x| \geq 1$}
    \int_{\substack{|x-y|>1 \\ |y|>1}} f(|x-y|)f(|y|)\,dy
    &\leq f(\tfrac 12|x|) \int_{|y|>1} f(|y|)\,dy \leq C  f(|x|) \int_{|y|>1} f(|y|)\,dy.
\end{align*}
This gives (\ref{A1}.d).

\medskip\noindent\ref{sufficient-djp-b}
Using $|x| \leq |x-y|+|y|$, $x,y \in \real^d$, we get
\begin{align*}
    \int_{\substack{|x-y|>1 \\ |y|>1}} f(|x-y|)f(|y|)\,dy
    &= \int_{\substack{|x-y|>1 \\ |y|>1}} e^{-m(|x-y|+|y|)} h(|x-y|)h(|y|)\,dy \\
    &\leq e^{-m|x|} \int_{\substack{|x-y|>1 \\ |y|>1}} h(|x-y|)h(|y|)\,dy.
\end{align*}
Now we can use part~\ref{sufficient-djp-a} for the last integral involving $h$, and (\ref{A1}.d) follows.

\medskip\noindent\ref{sufficient-djp-c}
The condition \eqref{eq:int-cond} guarantees that $C := \int_{|y|>1} f(|y|)\,dy < \infty$. Indeed, since $f'/f \leq 0$, $f(|\cdot|)$ is integrable on the set $\left\{y: |y|>1, y_1 \geq 0\right\}$, hence on $\{|y| > 1\}$ because of rotational symmetry.

Write
\begin{align*}
    \int_{\substack{|x-y|>1 \\ |y|>1}} f(|x-y|)f(|y|)\,dy
    &\leq \int_{\substack{1<|x-y|<|x| \\ 1<|y|<|x|}} f(|x-y|)f(|y|)\,dy + 2 \int_{\substack{|x-y|>|x| \\ |y|>1}} f(|x-y|)f(|y|)\,dy \\
    &=: \mathrm{I}+\mathrm{II}.
\end{align*}
By the monotonicity of $f$, we have $\mathrm{II} \leq 2 f(|x|) \int_{|y|>1} f(|y|)\,dy \leq 2C f(|y|)$. In order to estimate $\mathrm{I}$, we consider two cases: $1\leq |x|\leq 2$ and $|x|>2$.

\medskip\noindent
\emph{Case 1: $1 \leq |x| \leq 2$}. We have
\begin{gather*}
    \mathrm{I}
    \leq f(1) \int_{|y|>1} f(|y|)\,dy
    \leq \frac{f(1)}{f(2)} f(|x|) \int_{|y|>1} f(|y|)\,dy
    \leq \frac{f(1)}{f(2)} C f(|x|).
\end{gather*}

\smallskip\noindent
\emph{Case 2: $|x| > 2$}. We can use again the symmetry of $y$ and $x-y$ in the integrand to see
\begin{align*}
    \mathrm{I}
    &= 2 f(|x|) \int_{\substack{1<|x-y|<|x|,\: 1<|y|<|x| \\ |y| < |x-y|}} \frac{f(|x-y|)}{f(|x|)}f(|y|)\,dy \\
    &= 2 f(|x|) \int_{\substack{1<|x-y|<|x|,\: 1<|y|<|x| \\ |y| < |x-y|}} e^{-(\log f(|x|) -\log f(|x-y|))} f(|y|)\,dy.
\end{align*}
Using the log-convexity of $f$ on $(1,\infty)$, we get for $|x|>2$
\begin{gather*}
    \log f(|x|) -\log f(|x-y|)
    \geq \frac{f'(|x-y|)}{f(|x-y|)}(|x|-|x-y|) ,
\end{gather*}
almost everywhere on the domain of the above integration. Since  $f'/f$ is increasing and negative, and since $0<|x|-|x-y| \leq (x \cdot y)/|x|$ holds on the domain of integration, we obtain
\begin{gather*}
    \log f(|x|) -\log f(|x-y|)
    \geq \frac{f'(|y|)}{f(|y|)} \frac{x \cdot y}{|x|}
\end{gather*}
almost everywhere. This gives the estimate
\begin{gather*}
    I
    \leq 2 f(|x|)  \int_{\substack{1<|x-y|<|x|,\: 1<|y|<|x| \\ |y| < |x-y|}} e^{-\frac{f'(|y|)}{f(|y|)} \frac{x \cdot y }{|x|} } f(|y|)\,dy
    \leq 2 f(|x|) \int_{|y|>1} e^{-\frac{f'(|y|)}{f(|y|)} \frac{x \cdot y }{|x|} } f(|y|)\,dy
\end{gather*}
for all $|x| > 2$. Thus, \eqref{eq:int-cond} implies (\ref{A1}.d).
\end{proof}

\subsection{Decomposition of the paths of the process and related estimates}\label{sec3:decomposition}

In view of the Feynman--Kac representation of the semigroup $\big\{U_t:t \geq 0\big\}$ and Lemma~\ref{new}.\ref{new-a}, we can estimate $u_t(x,y)$ if we can control the behaviour of the sample path of the process. For this, we decompose the paths using exit times from and entrance times into certain annuli. Such decompositions appeared for the first time in \cite{BK} and they were also used in \cite{bib:KS,bib:KL15}. Let $k, n_0 \in \nat$, $k \geq n_0 \geq 2$, and define
\begin{align*}
    \A{k-1}{k}
    &:=
    \begin{cases}
        \left\{x \in \real^d: k-1 <|x|\leq k \right\} &\text{if\ \ } k \geq n_0 +2,\\
        \left\{x \in \real^d: |x|\leq n_0+1 \right\} &\text{if\ \ } k = n_0+1, \\
        \left\{x \in \real^d: |x|\leq n_0 \right\} &\text{if\ \ } k = n_0,
    \end{cases}
\intertext{and}
    \A{k-2}{\infty}
    &:=
    \begin{cases}
        \left\{x \in \real^d: |x| > k -2  \right\} &\text{if\ \  } k \geq n_0 +2,\\
        \real^d &\text{if\ \ } k \in \left\{n_0, n_0+1 \right\}.
    \end{cases}
\end{align*}
The exact value of $n_0$ will be chosen later on. We need two families of stopping times
\begin{gather*}
    \sigmaAkk := \inf \left\{t \geq 0: X_t \in \A{k-1}{k}\right\}
    \quad\text{and}\quad
    \tauAkinfty := \inf \left\{t \geq 0: X_t \notin \A{k-2}{\infty}\vphantom{\mathrm{I}}\right\};
\end{gather*}
as $\A{n_0-2}{\infty} = \A{n_0-1}{\infty} = \real^d$, we have $\tauAnOinfty = \tauAnOOinfty = \infty$.

With these stopping times we can count the number of annuli which are visited by the process $X$ on its way from $\A{n-2}{\infty}$ to $\A{k-1}{k}$ for $n-2 \geq k$ moving `inward', i.e.\ when the modulus $|X_s|$ reaches a new minimum --- see Fig.~\ref{fig-22}. More precisely, for $n-2 \geq k \geq n_0$ and $t > 0$, we set
\begin{align*}
    S(\A{n-2}{\infty}, \A{k-1}{k}; 1; t)
    &:= \left\{X_{\tauAninfty} \in \A{k-1}{k}, \sigmaAkk < t \right\}, \\
    S(\A{n-2}{\infty}, \A{k-1}{k}; l; t)
    &:= \bigcup_{p=k+2}^{n-2} S(\A{n-2}{\infty}, \A{p-1}{p}; l-1; t) \cap S(\A{p-2}{\infty}, \A{k-1}{k}; 1; t), \quad l >1.
\end{align*}
The first set is the event that the process moves to the annulus $\A{k-1}{k}$ before time $t$ upon exiting $\A{n-2}{\infty}$. The second event is defined recursively: $S(\A{n-2}{\infty}, \A{p-1}{p}; l-1; t) \cap S(\A{p-2}{\infty}, \A{k-1}{k}; 1; t)$ describes those paths which move, before time $t$, from $\A{n-2}{\infty}$ to the annulus $\A{p-1}{p}$, visiting on their way exactly $l-1$ annuli in-between $\A{n-2}{\infty}$ and $\A{p-1}{p}$ (including the final destination). In the end, but still before time $t$, the process moves directly from $\A{\infty}{p-2}\supset \A{p-1}{p}$ to $\A{k-1}{k}$.

\begin{figure}\centering
    \includegraphics[width = \textwidth]{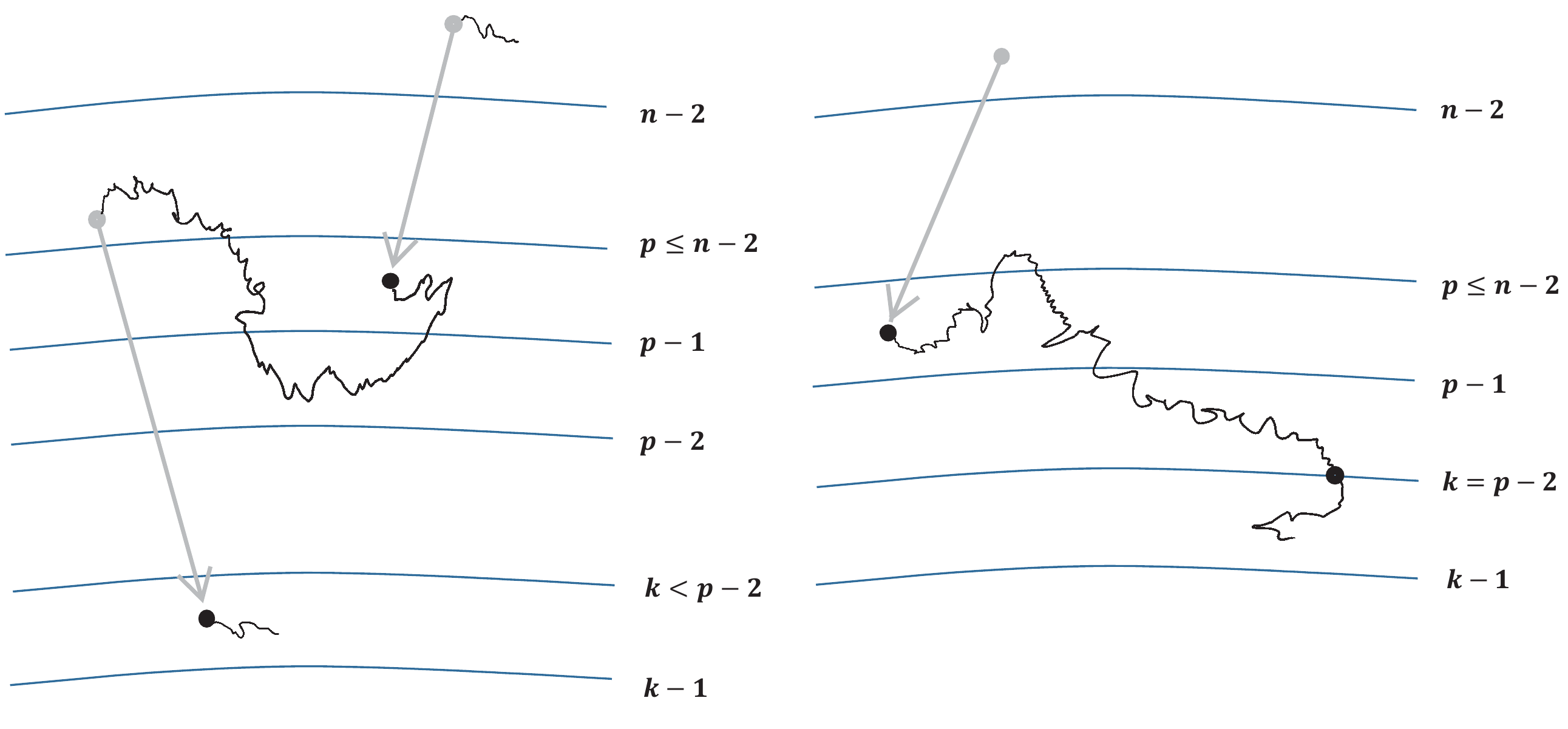}
    \caption{A typical path from the set $S(\A{n-2}{\infty}, \A{k-1}{k}; 2; t)$ with jump entry into $\A{k-1}{k}$ (picture on the left) and continuous entry (picture on the right).}\label{fig-11}
\end{figure}

\begin{figure}[ht]\centering
\includegraphics[width = .5\textwidth]{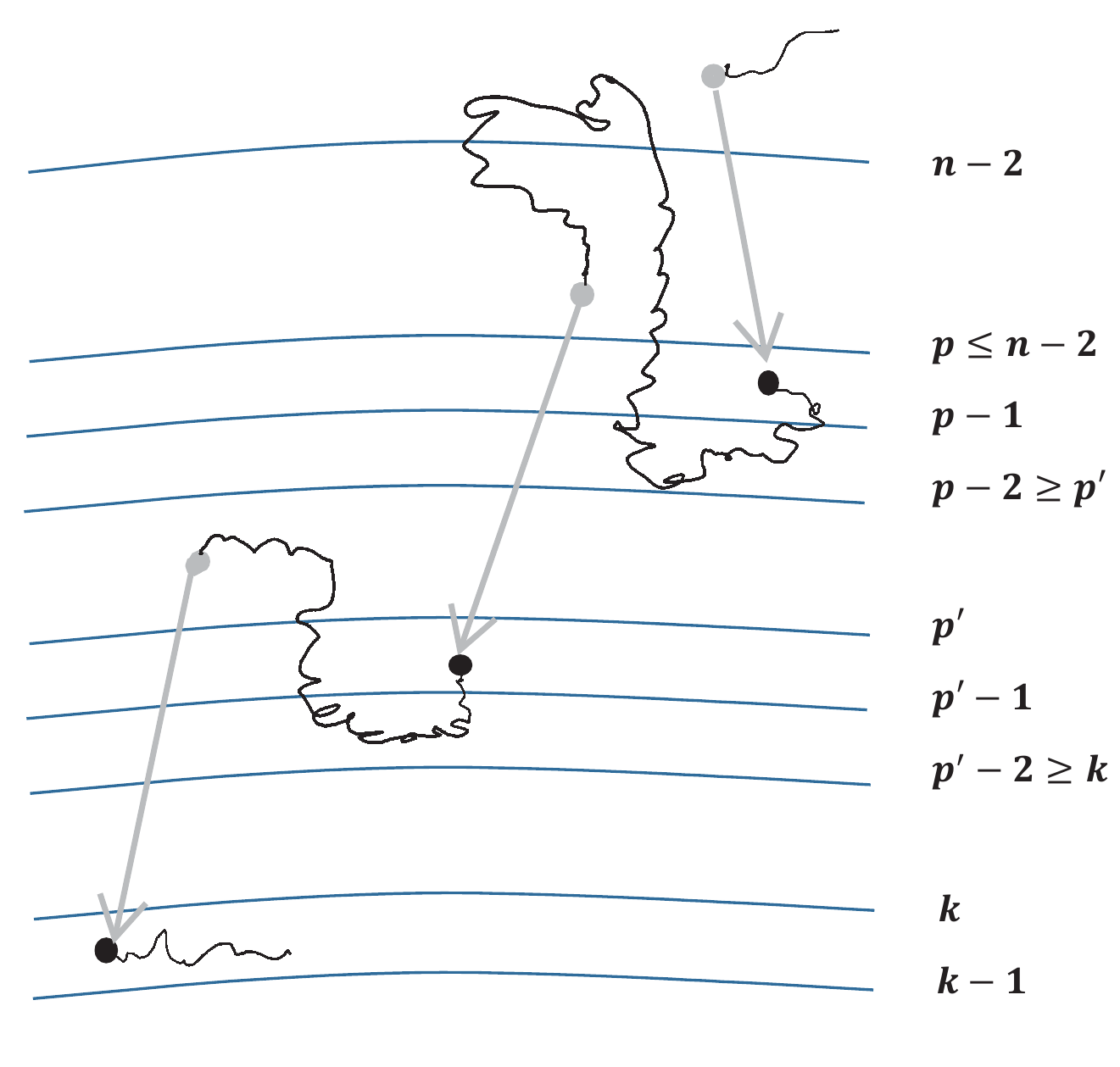}
    \caption{A typical path contained in the set $S(\A{n-2}{\infty}, \A{k-1}{k}; 3; t)$ comprising sample paths moving from $\A{n-2}{\infty}$ to $\A{k-1}{k}$ via $\A{p-1}{p}$ and $\A{p'-1}{p'}$ (for simplicity, we show only jump entries into new annuli). Notice that the number $l=3$ only counts first visits to annuli in a `downward' or `inward' movement where the modulus $|X_s|$ reaches new minima, i.e.\ only the visits $n-2\geq p > p-2 \geq p' > p'-2 \geq k$ are counted. Those annuli, which are visited (even for the first time) when going `outward', are not counted.} \label{fig-22}
\end{figure}

We will need the following class of functions which is defined with the constant $\widetilde C_2$ from \eqref{eq:loc_comp_f1}.
\begin{align} \label{eq:W_prop}
    \cC
    := \left\{W:\real^d \xrightarrow[]{\text{measurable}}\real \mid \forall x,y\in\real^d, |x-y|\leq 1 \::\: W(x) \leq \widetilde C_2 W(y) \right\}.
\end{align}
Because of the choice of $\widetilde C_2$, all functions of the form $W_z(\cdot):= f_1(|\cdot - z|)$, $z \in \real^d$, are in $\cC$.

For our heat kernel estimates we need precise estimates of the following expectations:
\begin{gather*}
    \ex^x\left[e^{-\int_0^{\sigmaAkk}V(X_s)\,ds} W(X_{\sigmaAkk}); \; S(\A{n-2}{\infty},\A{k-1}{k},l,t)\right]
\end{gather*}
for an arbitrary function $W \in \cC$. For $W \equiv 1$ such estimates have been established in~\cite[Lemmas~3.5, 3.6]{bib:KL15}; in the present paper we have to consider highly anisotropic functions $W$, which adds some complications. Since the argument from~\cite{bib:KL15} does not work in such generality, we have to find a suitable modification.
\begin{lemma}\label{lem:lem1}
    Assume \eqref{A1}--\eqref{A2} and let $W \in \cC$ and $n, k \in \nat$ be such that $n-2 \geq k \geq n_0$. There is a constant $C_{\ref{lem:lem1}} >0$ and $\vartheta_0 \geq 1$ such that for every $t > 0$, for all $x \in \A{n-1}{ \infty}$ and $\vartheta >\vartheta_0$ we have
    \begin{gather*}
        \ex^x\left[e^{-\vartheta \tauAninfty} W(X_{\tauAninfty}); \; \tauAninfty < t, X_{\tauAninfty} \in \A{k-1}{k} \right]
        \leq \frac{C_{\ref{lem:lem1}}}{\vartheta} \int_{\A{k-1}{k}} f(|x-z|) W(z)\,dz;
    \end{gather*}
    $C_{\ref{lem:lem1}}$ depends on $W$ only through the growth constant $\widetilde C_{2}$ appearing in \eqref{eq:loc_comp_f1} and \eqref{eq:W_prop}.
\end{lemma}

\begin{remark}
    On the set $\left\{\tauAninfty < t, X_{\tauAninfty} \in \A{k-1}{k}\right\}$ we have $\tauAninfty = \sigmaAkk$.
\end{remark}

\begin{proof}[Proof of Lemma~\ref{lem:lem1}.]
Define for $\vartheta >0$
\begin{gather*}
    u(y)
    := \ex^y\left[e^{-\vartheta \tauAninfty} W(X_{\tauAninfty}); \; \tauAninfty < \infty, X_{\tauAninfty} \in \A{k-1}{k} \right],
    \quad y \in \A{n-2}{\infty}.
\end{gather*}
We will consider two cases: $k=n-2$ and $k<n-2$.

\medskip\noindent\emph{Case 1:} $k=n-2$.
For every fixed $x \in \A{n-2}{\infty}$ we set
\begin{align*}
    v_x(y)
    &:= \ex^y\left[ e^{-\vartheta \tauAninfty} W(X_{\tauAninfty}); \; \tauAninfty < \infty, X_{\tauAninfty} \in \A{k-1}{k} \cap B_1(x) \right],
    \quad y \in \A{n-2}{\infty},
\intertext{and}
    h_x(y)
    &:=
    \begin{cases}
      \ex^y\left[e^{-\vartheta \tauAninfty} W(X_{\tauAninfty}); \; \tauAninfty < \infty, X_{\tauAninfty} \in \A{k-1}{k} \setminus B_1(x) \right],
        &y \in \A{n-2}{\infty}, \\[\medskipamount]
	   W(y) \I_{\A{k-1}{k}\setminus B_1(x)} (y),
        &y \notin  \A{n-2}{\infty}.
    \end{cases}
\end{align*}
Clearly, we have the following decomposition
\begin{gather*}
u(y) = v_x(y) + h_x(y), \quad y \in \A{n-2}{\infty},
\end{gather*}
and we will estimate $v_x$ and $h_x$ separately.

\begin{figure}\centering
\includegraphics[width = .6\textwidth]{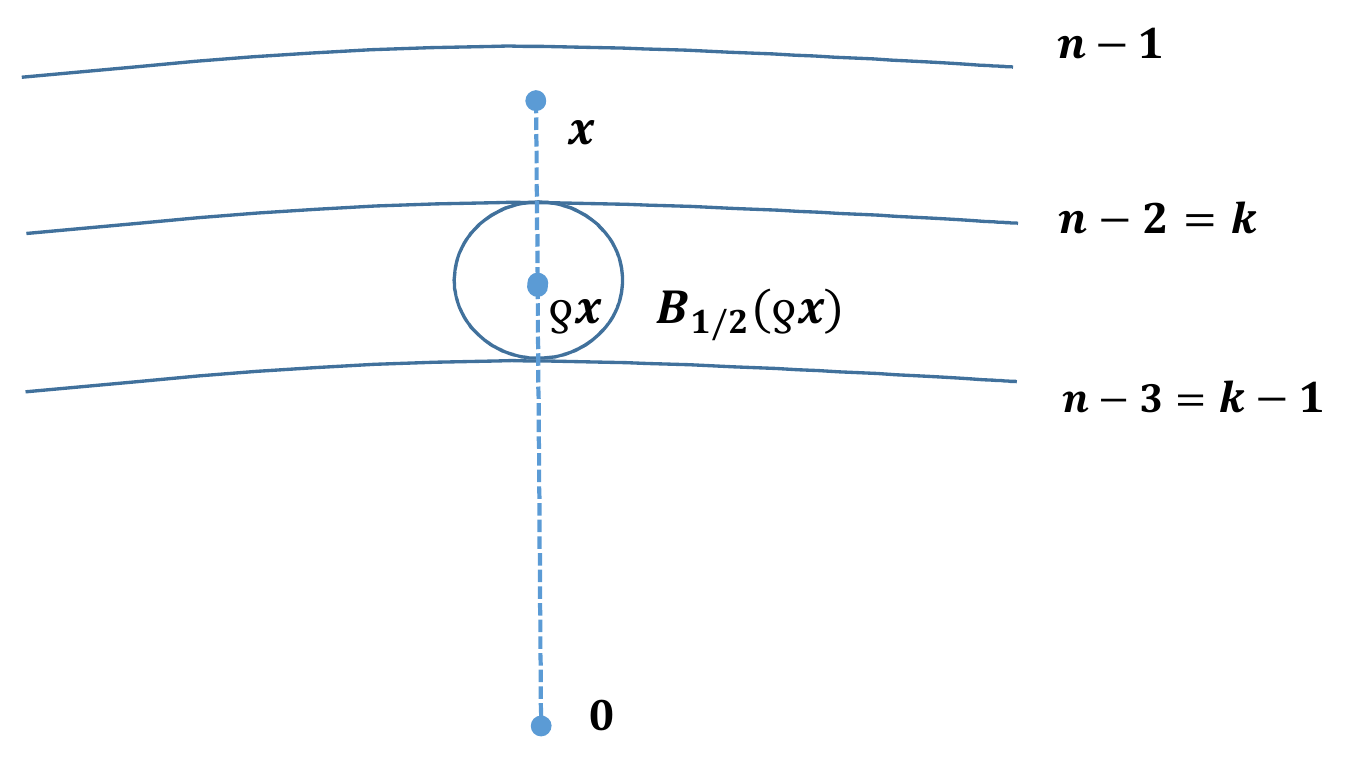}
    \caption{Construction of the ball $B_{1/2}(\varrho x)$.} \label{fig-33}
\end{figure}

\paragraph*{\itshape Estimation of $v_x$:}
    If $x \in \A{n-1}{\infty}$, then $v_x \equiv 0$ since $B_1(x) \cap \A{k-1}{k} = B_1(x) \cap  \A{n-3}{n-2} = \emptyset$. On the other hand, for $x \in \A{n-2}{\infty} \setminus \A{n-1}{\infty}$, we get from \eqref{eq:W_prop} and the definition of $v_x$ that $v_x(y) \leq \widetilde C_2 W(x)$. Pick $\varrho \in (0,1)$ such that $B_{1/2}(\varrho x) \subset {\A{k-1}{k}}$ and $\sup_{z \in B_{1/2}(\varrho x)} |x-z| \leq 2$ and let $c_1:=|B_{1/2}(0)|$, see Fig.~\ref{fig-33}. The monotonicity of $f_1$ and \eqref{eq:W_prop} give
    \begin{align*}
        W(x)
        = \frac{f_1(2) W(x) |B_{1/2}(\varrho x)|}{c_1f_1(2)}
        &= \frac{1}{c_1f_1(2)} \int_{B_{1/2}(\varrho x)} f_1(2) W(x)\, dz\\
        &\leq \frac{(\widetilde C_2)^2}{c_1f_1(2)} \int_{B_{1/2}(\varrho x)} f_1(|x-z|) W(z)\, dz.
    \end{align*}
    Therefore,
    \begin{gather*}
        v_x(y)
        \leq \frac{(\widetilde C_2)^3}{c_1f_1(2)} \int_{\A{k-1}{k}} f_1(|x-z|) W(z) \,dz,
        \quad x \in \A{n-2}{\infty} \setminus \A{n-1}{\infty}.
    \end{gather*}

\paragraph*{\itshape Estimation of $h_x$:}
    Set $D_{n,x} = \A{n-2}{\infty} \cap B_1(x)$ and observe that $\tauAninfty \geq \tau := \tau_{D_{n,x}}$.  By the strong Markov property, $h_x(y) = \ex^y\left[e^{-\vartheta \tau} h_x(X_{\tau})\right]$ for every $y \in D_{n,x}$, i.e.\ $h_x$ is regular $(X,\vartheta)$-harmonic in $D_{n,x}$. On the other hand, $h_x(y) = 0$ for every $y \in B_1(x) \setminus D_{n,x}$. By \eqref{eq:f-f_1}, Lemma~\ref{lem:bhi} and \eqref{A1}, we have
    \begin{align*}
        h_x(x)
        &\leq \frac{c_2}{\vartheta} \int_{B_{1/2}^c(x)} h_x(z) f(|x-z|)\,dz\\
        &\leq \frac{c_3}{\vartheta} \int_{B_{1/2}^c(x)} h_x(z) f_1(|x-z|)\,dz\\
        &\leq \frac{c_3}{\vartheta}\int_{\A{k-1}{k}} f_1(|x-z|) W(z)\,dz + \frac{c_3}{\vartheta} \int_{\A{n-2}{\infty}} h_x(z) f_1(|x-z|)\,dz,
    \end{align*}
    (recall that $k=n-2$) with absolute constants $c_2$ and $c_3$.

Combining the estimates from above, we get
\begin{align}
\label{eq:iter1}
    u(x)
    &= h_x(x) \leq \frac{c_3}{\vartheta}\int_{\A{k-1}{k}} f_1(|x-z|) W(z)\,dz + \frac{c_3}{\vartheta} \int_{\A{n-2}{\infty}} u(z) f_1(|x-z|)\,dz, \quad x \in \A{n-1}{\infty},
\intertext{and}
\notag
    u(x)
    &=  v_x(x) + h_x(x)\\
\label{eq:iter2}
    &\leq c_4 \int_{\A{k-1}{k}} f_1(|x-z|) W(z)\,dz + \frac{c_3}{\vartheta} \int_{\A{n-2}{\infty}} u(z) f_1(|x-z|)\,dz,
    \quad x \in \A{n-2}{\infty},
\end{align}
where $c_4 = c_3\vartheta^{-1} + (\widetilde C_2)^3 (c_1f_1(2))^{-1}$.

Using induction in $p \in \nat_0$, we will prove for $p\in\nat_0$ and $x \in \A{n-2}{\infty}$ the estimate
\begin{align} \label{eq:iter3}
    u(x)
    \leq c_4 \sum_{i=0}^{p} \left(\frac{c_3C_{\ref{lem:equivalent-djp}}}{\vartheta}\right)^i \int_{\A{k-1}{k}} f_1(|x-z|) W(z)\,dz
    + \left\|W\right\|_{\infty}\left(\frac{c_3  \left\|f_1\right\|_1}{\vartheta}\right)^{p+1}.
\end{align}
For $p=0$ the inequality \eqref{eq:iter3} follows from \eqref{eq:iter2} and the definition of the function $u$. Suppose now that \eqref{eq:iter3} holds true for some $p \geq 0$. Inserting it into \eqref{eq:iter2}, we get
\begin{align*} 
    u(x)
    &\leq c_4 \int_{\A{k-1}{k}} f_1(|x-z|) W(z)\,dz \\
    &\qquad\mbox{}+ \frac{c_3c_4}{\vartheta}  \sum_{i=0}^{p} \left(\frac{c_3C_{\ref{lem:equivalent-djp}}}{\vartheta}\right)^i   \underbrace{\int_{\A{n-2}{\infty}} \left(\int_{\A{k-1}{k}} f_1(|z-w|) W(w)\,dw\right) f_1(|x-z|)\,dz}_{=:\mathrm{J}(x)} \\
    &\qquad\mbox{}+ \left\|W\right\|_{\infty}\left(\frac{c_3  \left\|f_1\right\|_1}{\vartheta}\right)^{p+2},
\end{align*}
for every $x \in \A{n-2}{\infty}$. By Tonelli and \eqref{eq:conv_f1} we have
\begin{gather*}
    \mathrm{J}(x)
    = \int_{\A{k-1}{k}} \left(\int_{\A{n-2}{\infty}} f_1(|x-z|)f_1(|z-w|)\,dz \right) W(w)\,dw
    \leq C_{\ref{lem:equivalent-djp}} \int_{\A{k-1}{k}} f_1(|x-w|) W(w)\,dw.
\end{gather*}
Returning to the previous estimate, we get
\begin{align*}
    u(x)
    &\leq c_4 \left[1+\frac{c_3C_{\ref{lem:equivalent-djp}}}{\vartheta} \sum_{i=0}^{p}
     \left(\frac{c_3C_{\ref{lem:equivalent-djp}}}{\vartheta}\right)^i\right]
     \int_{\A{k-1}{k}} f_1(|x-z|) W(z)\,dz + \left\|W\right\|_{\infty} \left(\frac{c_3  \left\|f_1\right\|_1}{\vartheta}\right)^{p+2} \\
    &= c_4 \sum_{i=0}^{p+1} \left(\frac{c_3C_{\ref{lem:equivalent-djp}}}{\vartheta}\right)^i \; \int_{\A{k-1}{k}} f_1(|x-z|) W(z)\,dz + \left\|W\right\|_{\infty}\left(\frac{c_3  \left\|f_1\right\|_1}{\vartheta}\right)^{p+2}.
\end{align*}
This is exactly \eqref{eq:iter3} with $p \rightsquigarrow p+1$.

Taking $\vartheta > \max\left\{2c_3C_{\ref{lem:equivalent-djp}},\, c_3 \left\|f_1\right\|_1\right\}$ and letting $p \to \infty$, we obtain
\begin{align}\label{eq:iter4}
    u(x)
    \leq \frac{c_4 \vartheta}{\vartheta - c_3C_{\ref{lem:equivalent-djp}}} \int_{\A{k-1}{k}} f_1(|x-z|) W(z)\,dz
    \leq 2c_4 \int_{\A{k-1}{k}} f_1(|x-z|) W(z)\,dz,
    \quad x \in \A{n-2}{\infty}.
\end{align}
Inserting \eqref{eq:iter4} into \eqref{eq:iter1} and a further application of Tonelli's theorem and \eqref{eq:conv_f1} (as above) yields
\begin{gather*}
    u(x)
    \leq \frac{C_{\ref{lem:lem1}}}{\vartheta} \int_{\A{k-1}{k}} f_1(|x-z|) W(z)\,dz
    \leq \frac{C_{\ref{lem:lem1}}}{\vartheta} \int_{\A{k-1}{k}} f(|x-z|) W(z)\,dz ,
    \quad x \in \A{n-1}{\infty},
\end{gather*}
for every $\vartheta > 2c_3C_{\ref{lem:equivalent-djp}} + c_3 \left\|f_1\right\|_1$ with the constant $C_{\ref{lem:lem1}}= c_3 (1+2c_4 C_{\ref{lem:equivalent-djp}})$. This completes the proof in the case $k=n-2$.

\medskip\noindent
\emph{Case 2:} $k<n-2$. From the Ikeda-Watanabe formula \eqref{eq:IWF} we get
\begin{gather*}
    u(x)
    = \int_{\A{n-2}{\infty}} \int_0^{\infty} e^{-\vartheta t}\, p_{\A{n-2}{\infty}}(t,x,y) \,dt \int_{\A{k-1}{k}} W(z) \nu(z-y) \,dz\,dy,
    \quad x \in \A{n-2}{\infty}.
\end{gather*}
It follows from Lemma~\ref{lem:pt_below} that
\begin{align} \label{eq:low_dens}
    C_{\ref{lem:pt_below}} \nu(w)
    \leq p_{t_{\bee}}(w), \quad |w| \geq 1.
\end{align}
Observe that $p_{\A{n-2}{\infty}}(t,x,y)\leq p_t(x-y)$.
Since $\dist(\A{k-1}{k}, \A{n-2}{\infty}) \geq 1$, the estimate \eqref{eq:low_dens}, Tonelli's theorem and the Chapman--Kolmogorov equation give
\begin{align*}
    u(x)
    &\leq C_{\ref{lem:pt_below}}^{-1}\int_{\A{k-1}{k}}\int_0^{\infty} e^{-\vartheta t}\left(\int_{\A{n-2}{\infty}} p_t(x-y)p_{t_{\bee}}(y-z)\,dy \right) dt\,W(z)\,dz\\
    &\leq C_{\ref{lem:pt_below}}^{-1}\int_{\A{k-1}{k}}\int_{0}^{\infty} e^{-\vartheta t}\, p_{t+t_{\bee}}(x-z) \,dt \, W(z) \,dz.
\end{align*}
Finally, using (\ref{A2}.a) and \eqref{A1} we get
\begin{gather*}
    u(x)
    \leq c_5 \int_{0}^{\infty} e^{-(\vartheta - C_5) t} \,dt \; \int_{\A{k-1}{k}} f(|x-z|) W(z) \,dz,
    \quad x \in \A{n-1}{\infty}.
\end{gather*}
Setting $\vartheta_0: = 2c_3C_{\ref{lem:equivalent-djp}} + c_3 \left\|f_1\right\|_1 + 2C_5$ and $C_{\ref{lem:lem1}} := c_3 (1+2c_4 C_{\ref{lem:equivalent-djp}}) + 2c_5$ finishes the proof.
\end{proof}

\begin{lemma}\label{lem:lem2}
    Assume \eqref{A1}, \eqref{A2} and \textup{(\ref{A3}.a,b)} with radius $R_0>0$, and let $W \in \cC$, cf.\ \eqref{eq:W_prop}. Suppose that $n_0 \in \nat$ is so large that
    \begin{gather*}
        n_0 \geq R_0 + 2
    \et
        g(n_0-2) \geq 2C_6 \left[\vartheta_0 + 2C_3C_{\ref{lem:lem1}}(1+f(1))\right],
    \end{gather*}
    where $\vartheta_0$ and $C_{\ref{lem:lem1}}$ are from Lemma~\ref{lem:lem1}. For $n,k,l \in \nat$, $n-1 < |x| \leq n$, $n_0 \leq k \leq n-2$, and $t >0$, the following estimate holds with $C_{\ref{lem:lem2}} := 2 C_6 C_{\ref{lem:lem1}}$
    \begin{gather*}
        \ex^x\left[e^{-\int_0^{\sigmaAkk}V(X_s)\,ds} \, W(X_{\sigmaAkk}); \; S(\A{n-2}{\infty},\A{k-1}{k},l,t)\right]
        \leq \frac{C_{\ref{lem:lem2}}}{2^l g(n-2)} \int_{\A{k-1}{k}} f(|x-z|) W(z)\,dz.
    \end{gather*}
\end{lemma}

\begin{proof}
We use induction in $l \in \nat$.

Let $l=1$. By definition, $S(\A{n-2}{\infty},\A{k-1}{k},1,t) = \left\{X_{\tauAninfty} \in \A{k-1}{k}, \sigmaAkk < t \right\}$ and $\sigmaAkk = \tauAninfty$ on this set. From (\ref{A3}.b) and Lemma~\ref{lem:lem1} with $\vartheta = g(n-2)/C_6$ we get
\begin{align*}
    &\ex^x \left[e^{-\int_0^{\sigmaAkk}V(X_s)\,ds} \, W(X_{\sigmaAkk}); \; S(\A{n-2}{\infty}, \A{k-1}{k},1,t) \right] \\
    &\qquad \leq \ex^x\left[ e^{- (g(n-2)/C_6) \tauAninfty} \, W(X_{\tauAninfty}); \; X_{\tauAninfty} \in \A{k-1}{k}, \sigmaAkk < t \right]\\
    &\qquad\leq \frac{2C_6C_{\ref{lem:lem1}}}{2 g(n-2)} \int_{\A{k-1}{k}} f(|x-z|) W(z)\,dz.
\end{align*}
This means that the claimed bound holds for $l=1$, all $n,k,x,t$ as detailed in the statement of the lemma and all functions $W$ from the class $\cC$ (recall that $\cC$ includes all functions of the form $W_y(x)=f_1(|x-y|)$, $y \in \real^d$).

Now assume that the induction assumption holds for $1,2,...,l-1$ with $l \geq 2$. Using the decomposition of paths introduced at the beginning of Section~\ref{sec3:decomposition}, we may write
\begin{align*}
    \ex^x&\left[e^{-\int_0^{\sigmaAkk}V(X_s)\,ds} W(X_{\sigmaAkk}); \; S(\A{n-2}{\infty},\A{k-1}{k},l,t) \right] \\
    &\leq \sum_{p=k+2}^{n-2} \ex^x\left[e^{-\int_0^{\sigmaAkk}V(X_s)\,ds} W(X_{\sigmaAkk}); \; S(\A{n-2}{\infty},\A{p-1}{p},l-1,t) \cap S(\A{p-2}{\infty},\A{k-1}{k},1,t) \right].
\end{align*}
Since the process visits first $\A{p-1}{p}$ and then $\A{k-1}{k}$, we have $\sigmaAkk > \sigmaApp$. By the strong Markov property, the above expression becomes
\begin{align*}
    &\sum_{p=k+2}^{n-2} \ex^x\Bigg[ e^{-\int_0^{\sigmaApp}V(X_s)\,ds}\times\\
    &\times\ex^{X_{ \sigmaApp }}\!\! \left[e^{-\int_0^{\sigmaAkk}V(X_s)\,ds} W(X_{\sigmaAkk}); \; S(\A{p-2}{\infty},\A{k-1}{k},1,t-r) \right] \Bigg|_{r=\sigmaApp}\!\!\!; \; S(\A{n-2}{\infty},\A{p-1}{p},l-1,t)\Bigg].
\end{align*}
Using the induction hypothesis with $l=1$ for the inner expectation, we see that the above sum is less than
\begin{gather*}
    \sum_{p=k+2}^{n-2} \ex^x\bigg[e^{-\int_0^{\sigmaApp}V(X_s)\,ds} \frac{C_{\ref{lem:lem2}}}{2 g(p-2)} \int_{\A{k-1}{k}} f(|X_{\sigmaApp}-z|) W(z)\,dz; \, S(\A{n-2}{\infty},\A{p-1}{p},l-1,t)\bigg],
\intertext{which is, by Fubini's theorem, equal to}
    \int_{\A{k-1}{k}} \sum_{p=k+2}^{n-2} \frac{C_{\ref{lem:lem2}}}{2 g(p-2)} \ex^x\bigg[e^{-\int_0^{\sigmaApp}V(X_s)\,ds} f(|X_{\sigmaApp}-z|); \; S(\A{n-2}{\infty},\A{p-1}{p},l-1,t)\bigg]W(z)\,dz.
\end{gather*}
Now we estimate the expectation under the sum. Since $\dist(\A{p-1}{p},\A{k-1}{k}) \geq 1$, we have $f(|X_{\sigmaApp}-z|) \leq (1+f(1)) f_1(|X_{\sigmaApp}-z|)$, cf.~\eqref{eq:f-f_1}. Using the induction hypothesis for the functions $W_z(w) = f_1(|w-z|)$, $z \in \A{k-1}{k}$, we get
\begin{align*}
    \ex^x\bigg[e^{-\int_0^{\sigmaApp}V(X_s)\,ds} &f(|X_{\sigmaApp}-z|); \; S(\A{n-2}{\infty},\A{p-1}{p},l-1,t) \bigg] \\
    &\leq \frac{C_{\ref{lem:lem2}}(1+f(1))}{2^{l-1} g(n-2)} \int_{\A{p-1}{p}} f(|x-w|) f_1(|w-z|)\,dw.
\end{align*}
Plugging this estimate into the above expression, we finally get that the initial expectation is bounded by
\begin{align*}
    \frac{C_{\ref{lem:lem2}}^2(1+f(1))}{2^l g(n_0-2)g(n-2)} \int_{\A{k-1}{k}} \left(\int_{k+2<|w|<n-2} f(|x-w|) f_1(|w-z|)\,dw\right) W(z)\,dz.
\end{align*}
Using that $f_1\leq f$, the convolution condition \textup{(\ref{A1}.d)} and the fact that $2C_3C_6C_{\ref{lem:lem1}} (1+f(1)) \leq g(n_0-2)$, the last expression is bounded by
\begin{gather*}
    \frac{C_{\ref{lem:lem2}}}{2^l g(n-2)} \int_{\A{k-1}{k}} f(|x-z|) W(z)\,dz,
\end{gather*}
and we are done.
\end{proof}

In the sequel we will often use the following estimate which is based on the decomposition of paths introduced above. For every $x,y \in \real^d$ and $t > t_0 >0$ we have
\begin{gather}\label{eq:est_decomp}
    u_t(x,y)
    = \ex^x\left[e^{-\int_0^{t-t_0} V(X_s)\,ds}\, u_{t_0}(X_{t-t_0},y)\right]
    \leq \mathrm{I} + \sum_{k=n_0}^{n-2} \sum_{l=1}^{\infty} \mathrm{I}_{k,l},
\intertext{where}
\notag
    \mathrm{I} := \ex^x\left[e^{-\int_0^{t-t_0} V(X_s)\,ds}\, u_{t_0}(X_{t-t_0},y); \; t-t_0 < \tauAninfty\right],\\
\notag
    \mathrm{I}_{k,l}:=\ex^x\left[e^{-\int_0^{t-t_0} V(X_s)\,ds}\, u_{t_0}(X_{t-t_0},y); \; S(\A{n-2}{\infty}, \A{k-1}{k}, l,t-t_0), t-t_0 < \tau_{\A{k-2}{\infty}}\right]
\end{gather}
The estimates from Lemma~\ref{lem:lem2} will be essential for proving sharp bounds for the terms $\mathrm{I}_{k,l}$.

\section{General estimates of the Schr\"odinger heat kernel}\label{sec4}

Recall that the Schr\"{o}dinger semigroup is given through the Feynman-Kac formula
\begin{gather*}
    e^{-tH} f(x)
    = U_t f(x)
    = \ex^x\left[e^{-\int_0^t V(X_s)\,ds} f(X_t) \right]
    = \int_{\real^d} u_t(x,y) f(y)\,dy,
    \quad f \in L^2(\real^d),\; t>0.
\end{gather*}
In the next two sections we prove upper and lower estimates for $u_t(x,y)$ under rather general assumptions on the potential. More precisely, we will only assume that the potential $V$ satisfies the assumption \textup{(\ref{A3}.a,b)} and we do not require the growth property \textup{(\ref{A3}.c)}.

\subsection{The upper bound}\label{sec4:upper}

\begin{lemma}\label{lem:lem3}
    Assume \eqref{A1}, \eqref{A2} and \textup{(\ref{A3}.a,b)} with $t_{\bee}>0$ and $R_0>0$. Let $n_0 \in \nat$ be as in Lemma~\ref{lem:lem2}. There exists a constant $C>0$ such that for every $|x| > n_0+3$, $y \in \real^d$ and $t > t_{\bee}$ we have
    \begin{align*}
        u_t(x,y)
        &\leq \frac{Ce^{(C_5+C_{\ref{new}})t}}{g(|x|-2)} \bigg(e^{-\frac{t-t_{\bee}}{C_6} g(|x|-2)}g(|x|-2) f_1(|x-y|)  \\
        &\qquad\qquad\mbox{}+\int_{n_0+2 < |z| < |x|-1} f(|x-z|) f_1(|z-y|)  e^{-\frac{t-t_{\bee}}{2C_6} g(|z|-2)}\,dz
        + f(|x|) f_1(|y|) \bigg).
    \end{align*}
\end{lemma}
\begin{proof}
Let $|x| > n_0+3$, $y \in \real^d$ and $t \geq t_{\bee}$. Pick $n \in \nat$, $n\geq n_0+4$ such that $n-1 < |x| \leq n$. Because of \eqref{eq:est_decomp} we have to estimate $\mathrm{I}$ and $\mathrm{I}_{k,l}$.

By \textup{(\ref{A3}.b)}, \eqref{eq:u_by_p}, the Chapman--Kolmogorov equation for $p_t$ and (\ref{A2}.a), we get
\begin{align*}
    \mathrm{I}
    &\leq e^{-\frac{t-t_{\bee}}{C_6} g(n-2)} e^{C_{\ref{new}}t_{\bee}} \ex^x[p_{t_{\bee}}(X_{t-t_{\bee}}-y)] \\
    &\leq e^{C_{\ref{new}}t_{\bee}} e^{-\frac{t-t_{\bee}}{C_6} g(|x|-2)}\, p_{t}(x-y) \\
    &\leq C_4 e^{C_{\ref{new}}t_{\bee}} e^{C_5 t} e^{-\frac{t-t_{\bee}}{C_6} g(|x|-2)} f_1(|x-y|).
\end{align*}

Now we turn to $\mathrm{I}_{k,l}$. First, assume that $k \geq n_0+2$. To keep notation simple, we set $r:=t-t_{\bee}$. By the fact that $\sigmaAkk < r$ on $S(\A{n-2}{\infty}, \A{k-1}{k}, l,r)$, and \textup{(\ref{A3}.b)}, \eqref{eq:u_by_p} we have
\begin{align*}
    \mathrm{I}_{k,l}
    &\leq \ex^x\left[e^{-\int_0^r V(X_s) \,ds}\, u_{t_{\bee}}(X_{r},y); \; S(\A{n-2}{\infty}, \A{k-1}{k}, l,r), r < \tau_{\A{k-2}{\infty}}\right]
    \\
    &= \ex^x\left[ e^{-\frac 12\int_0^r V(X_s) \,ds} e^{-\frac 12\int_0^r V(X_s)\,ds}\, u_{t_{\bee}}(X_{r},y); \; S(\A{n-2}{\infty}, \A{k-1}{k}, l,r), r < \tau_{\A{k-2}{\infty}}\right]
    \\
    &\leq e^{C_{\ref{new}}t_{\bee}} e^{- \frac{r}{2C_6} g(k-2)} \ex^x\left[e^{-\int_0^{\sigmaAkk} (V(X_s)/2)\,ds }\, p_{t_{\bee}}(X_r-y); \; S(\A{n-2}{\infty}, \A{k-1}{k}, l,r)\right].
\end{align*}
By the strong Markov property, the Chapman--Kolmogorov equation for $p_t$, \eqref{A2} and \eqref{A3},
\begin{align*}
    &\mathrm{I}_{k,l}\\
    &\leq e^{C_{\ref{new}}t_{\bee}} e^{-\frac{r}{2C_6} g(k-2)} \ex^x\left[e^{-\frac 12\int_0^{\sigmaAkk} V(X_s)\,ds} \, \ex^{X_{\sigmaAkk}}[p_{t_{\bee}}(X_{r-v}-y)]_{v = \sigmaAkk};
    \; S(\A{n-2}{\infty}, \A{k-1}{k}, l,r)\right] \\
    &= e^{C_{\ref{new}}t_{\bee}} e^{-\frac{r}{2C_6} g(k-2)} \ex^x\left[e^{-\frac 12\int_0^{\sigmaAkk} V(X_s)\,ds }\, p_{t-\sigmaAkk}(X_{\sigmaAkk}-y);
    \; S(\A{n-2}{\infty}, \A{k-1}{k}, l,r)\right] \\
    &\leq C_4 e^{C_{\ref{new}}t_{\bee}} e^{C_5 t}e^{-\frac{r}{2C_6} g(k-2)} \ex^x\left[ e^{-\frac 12\int_0^{\sigmaAkk} V(X_s)\,ds }f_1(|X_{\sigmaAkk}-y|); \; S(\A{n-2}{\infty}, \A{k-1}{k}, l,r)\right] \\
    &\leq \frac{C_4 C_{\ref{lem:lem2}} e^{C_{\ref{new}}t_{\bee}} e^{C_5 t}}{2^l g(n-2)} \int_{\A{k-1}{k}} f(|x-z|) f_1(|z-y|) e^{-\frac{r}{2C_6} g(|z|-2)}\,dz.
\end{align*}
In the last estimate we use Lemma~\ref{lem:lem2} with the function $W = f_1$ (which is in the class $\cC$ defined in \eqref{eq:W_prop}) as well as the monotonicity (\ref{A3}.b) of $g$.

We still have to estimate $\mathrm{I}_{k,l}$ for $k=n_0, n_0+1$. Recall that $\A{n_0-1}{\infty} = \A{n_0-2}{\infty} = \real^d$ and $\tau_{\A{n_0-1}{\infty}} = \tau_{\A{n_0-2}{\infty}} = \infty$. This is the most critical situation and we cannot use the above arguments. By the strong Markov property and (the analogue of) the Chapman--Kolmogorov equations for the kernel $u_t$, we have
\begin{align}
    \mathrm{I}_{k,l}
    &\notag= \ex^x\left[e^{-\int_0^r V(X_s) \,ds}\, u_{t_{\bee}}(X_{r},y); \; S(\A{n-2}{\infty}, \A{k-1}{k}, l,r)\right] \\
    &\notag= \ex^x\left[e^{-\int_0^{\sigmaAkk} V(X_s) \,ds} \,\ex^{X_{\sigmaAkk}}\left[e^{-\int_0^{r-v} V(X_s) \,ds}\, u_{t_{\bee}}(X_{r-v},y)\right]_{v=\sigmaAkk}; \; S(\A{n-2}{\infty}, \A{k-1}{k}, l,r)\right] \\
    &\label{eq:lem3-aux}= \ex^x\left[e^{-\int_0^{\sigmaAkk} V(X_s) \,ds}\, u_{t-\sigmaAkk}(X_{\sigmaAkk},y); \; S(\A{n-2}{\infty}, \A{k-1}{k}, l,r)\right].
\end{align}
Applying \eqref{eq:u_by_p} and \eqref{A2} as before, we get that the above expectation is not greater than
\begin{align*}
    C_4 e^{C_{\ref{new}}t} \ex^x\left[e^{-\int_0^{\sigmaAkk} V(X_s)\,ds }\left(\left[e^{C_5 t} f(|X_{\sigmaAkk}-y|)\right] \wedge 1\right); \; S(\A{n-2}{\infty}, \A{k-1}{k}, l,r)\right].
\end{align*}
Observe that
\begin{gather*}
   \left[e^{C_5 t} f(|w|)\right] \wedge 1
    \leq e^{C_5 t} f_1(|w|),
    \quad w \in \real^d,
\end{gather*}
and recall that for $y \in \real^d$ the function  $W_{y}(w) = f_1(|w-y|)$ is also in the class $\cC$ defined in \eqref{eq:W_prop}. From Lemma~\ref{lem:lem2} we see
\begin{align*}
    \mathrm{I}_{k,l}
    &\leq C_4 e^{(C_5+C_{\ref{new}})t} \ex^x\left[e^{-\int_0^{\sigmaAkk} V(X_s)\,ds } W_{y}(X_{\sigmaAkk}); \; S(\A{n-2}{\infty}, \A{k-1}{k}, l,r)\right] \\
    &\leq \frac{C_4 C_{\ref{lem:lem2}}e^{(C_5+C_{\ref{new}})t}}{2^l g(n-2)} \int_{\A{k-1}{k}} f(|x-z|) f_1(|z-y|)\, dz
\end{align*}
and with (\ref{A1}.b) and (\ref{A1}.c) we conclude that
\begin{gather*}
    \mathrm{I}_{k,l} \leq c_3 \frac{e^{(C_5+C_{\ref{new}})t}}{2^l g(n-2)} f(|x|) f_1(|y|),
\end{gather*}
for $k = n_0, n_0+1$. Combining all inequalities, we finally obtain
\begin{align*}
    u_t(x,y)
    &\leq \frac{e^{(C_5+C_{\ref{new}})t}}{g(|x|-2)} \bigg(c_1 e^{-\frac{t-t_{\bee}}{C_6} g(|x|-2)}g(|x|-2) f_1(|x-y|)  \\
    &\qquad\mbox{}+c_2 \int_{n_0+2 < |z| < |x|-1} f(|x-z|) f_1(|z-y|)  e^{-\frac{t-t_{\bee}}{2C_6} g(|z|-2)}\,dz
    +c_3 f(|x|)f_1(|y|)\bigg).
\qedhere
\end{align*}
\end{proof}

\begin{lemma}\label{lem:lem4}
Assume \eqref{A1}, \eqref{A2} and \textup{(\ref{A3}.a,b)} with $t_{\bee}>0$ and $R_0>0$. For $n_0 \in \nat$ as in Lemma~\ref{lem:lem3} \textup{(}and Lemma~\ref{lem:lem2}\textup{)} there exists a constant $C>0$ such that for $t > 4t_{\bee}$ the following assertions hold.
\begin{enumerate}
\item\label{lem:lem4-a}
    If $|x|, |y| \leq n_0+3$, then
    \begin{align*}
        u_t(x,y) \leq C e^{(C_5+C_{\ref{new}}) t}.
    \end{align*}

\item\label{lem:lem4-b}
    If $|x| > n_0+3$ and $|y| \leq n_0+3$, then
    \begin{align*}
        u_t(x,y) \leq C e^{(C_5+C_{\ref{new}}) t} \frac{f(|x|)}{g(|x|-2)};
    \end{align*}
    if $|x| \leq n_0+3$ and $|y| > n_0+3$, then, by symmetry,
    \begin{align*}
        u_t(x,y) \leq C e^{(C_5+C_{\ref{new}}) t} \frac{f(|y|)}{g(|y|-2)}.
    \end{align*}

\item\label{lem:lem4-c}
    If $|x|, |y| > n_0+3$, then
    \begin{align*}
        u_t(x,y)
        &\leq C e^{(C_5 + C_{\ref{new}})t} \Bigg(e^{- \frac{t-3t_{\bee}}{C_6} g(|x|-2)} \, \frac{f_1(|x-y|)}{g(|y|-2)}\\
        &\quad\mbox{}+\frac{1 }{g(|x|-2) g(|y|-2)} \int_{n_0+2 < |z| < |x|-1} f(|x-z|) f_1(|z-y|)  e^{-\frac{t-3t_{\bee}}{2C_6} g(|z|-2)}\,dz \\
        &\quad\mbox{}+\frac{1}{g(|x|-2)g(|y|-2)} f(|x|) f(|y|)\Bigg) .
    \end{align*}
    In particular, the following symmetrized estimate holds
    \begin{align*}
        u_t(x,y)
        &\leq C e^{(C_5 + C_{\ref{new}})t} \Bigg[ \Bigg(\frac{e^{- \frac{t-3t_{\bee}}{2C_6} g(|x|-3) }}{g(|y|-2)}
        \wedge \frac{e^{- \frac{t-3t_{\bee}}{2C_6} g(|y|-3)}}{g(|x|-2)} \Bigg) f_1(|x-y|) \\
        &\quad\mbox{}+\frac{1}{g(|x|-2) g(|y|-2)} \!\!\int\limits_{n_0+2 \leq |z| \leq (|x|-1) \vee (|y|-1)}\!\!\!\!\!\!\!\!\!\!\!\!\!\!\! f_1(|x-z|)f_1(|z-y|) e^{-\frac{t-3t_{\bee}}{2C_6} g(|z|-2)}\,dz \\
        &\quad\mbox{}+\frac{1}{g(|x|-2)g(|y|-2)} f(|x|) f(|y|) \Bigg].
    \end{align*}
\end{enumerate}
\end{lemma}

\begin{proof}
\ref{lem:lem4-a}
This follows directly from \eqref{eq:u_by_p} and \eqref{A2}.

\medskip\noindent
\ref{lem:lem4-b}
By symmetry it is enough to consider the case $|x| > n_0+3$ and $|y| \leq n_0+3$. Since $|y| \leq n_0+3$, the estimate in Lemma~\ref{lem:lem3} and (\ref{A3}.a,b) show for every $t > 4t_{\bee}$
\begin{align*}
    u_t(x,y)
    &\leq \frac{e^{(C_5+C_{\ref{new}})t}}{g(|x|-2)}\Bigg(c_1 f(|x|)+ c_2 \int\limits_{n_0+2 < |z| < |x|-1}\!\!\!\! f(|x-z|) f(|z|) \,dz + c_3  f(|x|)\Bigg)
\end{align*}
and, by \textup{(\ref{A1}.d)}, we conclude that
\begin{gather*}
    u_t(x,y)  \leq \frac{c_4 e^{(C_5+C_{\ref{new}})t}}{g(|x|-2)} f(|x|).
\end{gather*}

\medskip\noindent
\ref{lem:lem4-c}
Let $|x|, |y| > n_0+3$ and $t > 4t_{\bee}$. In view of \textup{(\ref{A1}.d)}, \textup{(\ref{A3}.a,b)}, \eqref{eq:prod_by_diff} and \eqref{eq:f-f_1}, we get with Lemma~\ref{lem:lem3} that
\begin{align} \label{eq:bd_u_t0}
    u_{2t_{\bee}}(z,y)
    = u_{2t_{\bee}}(y,z)
    \leq c_5 \frac{f_1(|z-y|)}{g(|y|-2)}, \quad z \in \real^d.
\end{align}
The symmetry of the kernel $u_{2t_{\bee}}(z,y)$, \eqref{eq:bd_u_t0} and a further application of the estimate from Lemma~\ref{lem:lem3} to $u_{t-2t_{\bee}}(x,z)$ (with $t-2t_{\bee} \geq 2t_{\bee} > t_{\bee}$), yield
\begin{align*}
    &u_t(x,y)  = \int_{\real^d} u_{t-2t_{\bee}}(x,z) u_{2t_{\bee}}(y,z)\,dz \\
    &\leq c_7 e^{(C_5+C_{\ref{new}}) t} e^{-\frac{t-3t_{\bee}}{C_6} g(|x|-2)} \frac{1}{g(|y|-2)} \int_{\real^d}f_1(|x-z|) f_1(|z-y|)\,dz  \\
    &\quad\mbox{}+ \frac{c_7 e^{(C_5+C_{\ref{new}}) t}}{g(|x|-2)g(|y|-2)} \int\limits_{\real^d}
    \int\limits_{n_0+2 < |w| < |x|-1} f(|x-w|) f_1(|w-z|)  e^{-\frac{t-3t_{\bee}}{2C_6} g(|w|-2)}\,dw \: f_1(|z-y|) \,dz\\
    &\quad\mbox{}+ \frac{c_7e^{(C_5+C_{\ref{new}})t}}{g(|x|-2)g(|y|-2)} f(|x|) \int_{\real^d} f_1(|z|)f_1(|z-y|) \,dz .
\end{align*}
By Fubini's theorem and \eqref{eq:conv_f1}, we finally arrive at
\begin{align*}
    u_t(x,y)
    &\leq c_8 e^{(C_5 + C_{\ref{new}})t} \Bigg(e^{-\frac{t-3t_{\bee}}{C_6} g(|x|-2) } \frac{f_1(|x-y|)}{g(|y|-2)}\\
    &\quad\mbox{}+\frac{1 }{g(|x|-2) g(|y|-2)} \int_{n_0+2 < |z| < |x|-1} f(|x-z|) f_1(|z-y|)  e^{-\frac{t-3t_{\bee}}{2C_6} g(|z|-2)}\,dz \\
    &\quad\mbox{}+\frac{1}{g(|x|-2)g(|y|-2)} f(|x|) f(|y|) \Bigg).
    \end{align*}
This is the first claimed bound. The second one easily follows from the first by symmetry.
\end{proof}

\subsection{The lower bound}\label{sec4:lower}

We begin with an auxiliary result. Recall that $\mu_0(r) > 0$ and $\psi_{0,r}$ are the ground state eigenvalue and eigenfunction for the process killed on leaving a ball $B_r(0)$, $r >0$.

\begin{lemma} \label{lem:low1}
Assume \textup{(\ref{A1}.a,b,c)}, \eqref{A2} and \textup{(\ref{A3}.a,b)} with $t_{\bee} >0$ and $R_0 >0$. For every $n_0 \geq R_0$ there exist the constants $C, \widetilde C >0$ such that we have for $t > 2t_{\bee}$
\begin{gather*}
    u_t(x,y) \geq C e^{- t \mu_0(\frac 12)} f_1(|x-y|)
    \times
    \begin{cases}
        \frac{e^{- C_6 t g(|x|+2)}}{g(|y|+1/2)} \vee \frac{e^{-C_6 t g(|y|+2)}}{g(|x|+1/2)} &\text{if\ \ }|x| > n_0+2, |y| > n_0+2,
        \\[\medskipamount]
    	\frac{e^{- \widetilde Ct}}{g(|x|+1/2)}  &\text{if\ \ } |x| > n_0+2, |y| \leq n_0+2,
        \\[\medskipamount]
    	\frac{e^{- \widetilde Ct}}{g(|y|+1/2)}  &\text{if\ \ } |x| \leq n_0+2, |y| > n_0+2,
        \\[\medskipamount]
    	e^{- \widetilde Ct}  &\text{if\ \ } |x| \leq n_0+2, |y| \leq n_0+2.
    \end{cases}
\end{gather*}
\end{lemma}
\begin{proof}
Set $T_{\bee} := \frac 32 t_{\bee}$. We distinguish between two cases: $|x-y|<\frac 54$ and $|x-y|\geq \frac 54$.

\noindent
\emph{Case 1:} Assume that $t>T_{\bee}$ and $x, y \in \real^d$ satisfy $|x-y|<\frac 54$. By (the analogue of) the Chapman--Kolmogorov equations for $u_t$,
\begin{align*}
    u_t(x,y)
    &= \ex^x\left[e^{-\int_0^{t-T_{\bee}} V(X_s)\,ds}\, u_{T_{\bee}}(X_{t-T_{\bee}},y)\right] \\
    &\geq \ex^x\left[e^{-\int_0^{t-T_{\bee}} V(X_s)\,ds}\, u_{T_{\bee}}(X_{t-T_{\bee}},y); \; t-T_{\bee} < \tau_{B_{1/2}(0)}\right] \\
    &\geq e^{-(t-T_{\bee}) \sup_{z \in B_{1/2}(x)} V(z)} \pr^0\left(t-T_{\bee} < \tau_{B_{1/2}(0)}\right) \inf_{z \in B_{1/2}(x)} u_{T_{\bee}}(z,y).
\end{align*}
Moreover,
\begin{gather*}
    e^{-(t-T_{\bee})\mu_0(\frac 12)} \psi_{0,\frac12}(0)
    = \int_{B_{\frac 12}(0)} p_{B_{\frac 12}(0)}(t-T_{\bee},0,y) \psi_{0,\frac12}(y)\,dy
    \leq \big\|\psi_{0,\frac12}\big\|_{\infty}\pr^0\left(t-T_{\bee} < \tau_{B_{\frac 12}(0)}\right)
\end{gather*}
and, since $\psi_{0,1/2}(0)>0$, we have for some $c_1>0$
\begin{align} \label{eq:killed_exp}
    \pr^0\left(t-T_{\bee} < \tau_{B_{\frac 12}(0)}\right)
    \geq c_1 e^{-(t-T_{\bee}) \mu_0(\frac 12)}, \quad t > T_{\bee}.
\end{align}

Suppose first that $|x|,|y| > n_0+2$. By \eqref{eq:killed_exp}, \textup{(\ref{A3}.a,b)} and our assumption $|x-y| < \frac 54$, we get
\begin{gather} \label{eq:u_t_lower_est}
    u_t(x,y)
    \geq c_1 e^{-t \mu_0(\frac 12)} e^{-C_6(t-T_{\bee}) g(|y|+2)} \inf_{z \in B_{\frac 12}(x)} u_{T_{\bee}}(z,y).
\end{gather}
So it is enough to estimate the infimum. Let $z$ be such that $|x-z| < \frac 12$. Since $|y-z| < \frac 74$, we derive from \eqref{eq:killed_sem} that for all $s \in (0,T_{\bee})$
\begin{align} \label{eq:killed_aux_est}
    \ex^z\left[p_{B_2(y)}(T_{\bee}-s,X_s,y);\; s <  \tau_{B_2(y)}\right]
    &= \int_{B_2(y)} p_{B_2(y)}(s,z,w) p_{B_2(y)}(T_{\bee}-s,w,y) \, dw \\
    &\notag= p_{B_2(y)}(T_{\bee},z,y).
\end{align}
If we take $s < \frac 12 t_{\bee}$, we have $T_{\bee}-s > t_{\bee}$ and $p_{B_2(y)}(T_{\bee}-s,x,y) \leq p(T_{\bee}-s,x,y)$ is bounded because of (\ref{A2}.a). Using \cite[Proof of Th.\ 3.4, Claim 1]{bib:SchW} we get from the above equality that the left-hand side, hence $z\mapsto p_{B_2(y)}(T_{\bee},z,y)$ is continuous.

Moreover, by \eqref{eq:killed_sem}, the symmetry and the spatial homogeneity of the process $(X_t)_{t \geq 0}$, we get for every nonnegative function $h$ supported in $B_{7/4}(y)$
\begin{align*}
    \int_{B_2(y)} h(z)p_{B_2(y)}(T_{\bee},z,y)\,dz
    &= \int_{B_2(y)} h(z)p_{B_2(y)}(T_{\bee},y,z)\, dz \\
    &= \ex^y\left[h(X_{T_{\bee}});\; T_{\bee} <  \tau_{B_2(y)}\right]\\
	&= \ex^0\left[h(X_{T_{\bee}}-y);\; T_{\bee} <  \tau_{B_2(0)}\right] \\
	&= \int_{B_2(0)} h(z-y) p_{B_2(0)}(T_{\bee},0,z)\,dz\\
    &\geq \inf_{w \in B_{7/4}(0)} p_{B_2(0)}(T_{\bee},0,w) \int_{B_2(y)} h(z)\,dz .
\end{align*}
Inserting for $h$ a sequence of type delta centered at $y$, and using the continuity (in the variable $z$) of the kernel of the killed process, we get
\begin{align} \label{eq:killed_aux_est_bis}
    p_{B_2(y)}(T_{\bee},z,y) \geq \inf_{w \in B_{7/4}(0)} p_{B_2(0)}(T_{\bee},0,w).
\end{align}
By Lemma~\ref{new}.\ref{new-a}, \eqref{eq:killed_aux_est}, \textup{(\ref{A3}.a,b)} and \eqref{eq:killed_aux_est_bis}, we have
\begin{align*}
    u_{T_{\bee}}(z,y)
    &= \lim_{s \uparrow T_{\bee}} \ex^z\left[e^{-\int_0^s V(X_u)\,du}\, p_{T_{\bee}-s}(y-X_s)\right] \\
		&\geq \liminf_{s \uparrow T_{\bee}} \ex^z\left[e^{-\int_0^s V(X_u)\,du}\, p_{T_{\bee}-s}(y-X_s);\; s <  \tau_{B_2(y)}\right]\\
    &\geq e^{-T_{\bee} \sup_{z \in B_{2}(y)} V(z)} \liminf_{s \uparrow T_{\bee}}\ex^z\left[p_{B_2(y)}(T_{\bee}-s,X_s,y);\; s <  \tau_{B_2(y)}\right] \\
	&\geq e^{-C_6 T_{\bee} g(|y|+2)}\, p_{B_2(y)}(T_{\bee},z,y)\\
    &\geq e^{- C_6 T_{\bee} g(|y|+2)} \inf_{w \in B_{7/4}(0)} p_{B_2(0)}(T_{\bee},0,w),
\end{align*}
and returning to \eqref{eq:u_t_lower_est}, we conclude that
\begin{gather*}
    u_t(x,y)
    \geq c_2 e^{-t \mu_0(\frac 12)} e^{-C_6 t g(|y|+2)}.
\end{gather*}
Set $\widetilde C:= \sup_{z \in B(0,n_0+6)} V_{+}(z)$ and observe that if $|x| \leq n_0+2$ or $|y| \leq n_0+2$ --- we still assume $|x-y| < \frac 54$ --, we get with a similar argument
\begin{gather*}
    u_t(x,y)
    \geq c_1 e^{-t \mu_0(\frac 12) -(t-T_{\bee}) \widetilde C} \inf_{z \in B_{1/2}(x)} u_{T_{\bee}}(z,y)
\intertext{and}
    \inf_{z \in B_{1/2}(x)} u_{T_{\bee}}(z,y) \geq c_2 e^{-\widetilde CT_{\bee}}.
\end{gather*}
Together with the symmetry of the kernel $u_t(x,y)$, this gives
\begin{gather}\label{eq:low1}
    u_t(x,y)
    \geq c_3 e^{- t \mu_0(\frac 12)}
    \begin{cases}
        e^{-C_6 t g(|x| \wedge|y|+2)}, &\text{if $|x| > n_0+2$ and $|y| > n_0+2$},\\[\medskipamount]
		e^{- \widetilde C t},          &\text{if $|x| \leq n_0+2$ or $|y| \leq n_0+2$},
    \end{cases}
\end{gather}
as long as $|x-y|<\frac 54$ and $t > T_{\bee}$.

\medskip\noindent
\emph{Case 2}: Assume that $t>2t_{\bee}$ $x,y \in \real^d$ satisfy $|x-y| \geq \frac 54$. A further application of (the analogue of) the Chapman--Kolmogorov equations for the kernel $u_t$ and the strong Markov property yields
\begin{align*}
    u_t(x,y)
    &\geq \ex^x\left[e^{-\int_0^{t-T_{\bee}} V(X_s)\,ds}\, u_{T_{\bee}}(X_{t-T_{\bee}},y); \; t-T_{\bee} > \tau_{B_{1/2}(x)}\right] \\
    &=\ex^x\left[e^{-\int_0^{\tau_{B_{1/2}(x)}} V(X_s)\,ds}\, u_{t-\tau_{B_{1/2}(x)}}(X_{\tau_{B_{1/2}(x)}},y); \; t-T_{\bee} > \tau_{B_{1/2}(x)}\right].
\end{align*}
By the Ikeda-Watanabe formula \eqref{eq:IkWa}, the last expectation is greater than or equal to
\begin{gather*}
    \int_0^{t-T_{\bee}} \int_{B_{1/2}(x)} e^{- s \sup_{\xi \in B_{1/2}(x)} V(\xi)} \, p_{B_{1/2}(x)}(s,x,z)
    \Bigg( \int_{\substack{|y-w|<1/2 \\ |w| \geq |y|}} \nu(z-w) u_{t-s}(w,y) \,dw \Bigg)\, dz \,ds
\end{gather*}
and, because of \eqref{A1} and $|x-y| \geq \frac 54$, this expression can be estimated from below by
\begin{gather*}
    c_4 f(|x-y|) \int_0^{t-T_{\bee}} e^{- s \sup_{\xi \in B_{1/2}(x)} V(\xi)} \, \pr^0\left(s < \tau_{B_{1/2}(0)}\right)
    \Bigg( \int_{\substack{|y-w|<1/2 \\ |w| \geq|y|}} u_{t-s}(w,y) \,dw \Bigg)\, ds.
\end{gather*}

Suppose first that $|x|, |y| > n_0+2$. The restriction $|w| \geq|y|$ in the domain of integration above guarantees that $|w| > n_0+2$. By \textup{(\ref{A3}.a,b)}, \eqref{eq:killed_exp}, \eqref{eq:low1} and the fact that $t -T_{\bee} \geq \frac 12t_{\bee}$,  we finally get
\begin{align*}
    u_t(x,y)
    &\geq c_5 e^{-t\mu_0(\frac 12)} f(|x-y|) \left(\int_0^{t-T_{\bee}}  e^{- C_6 s g(|x|+1/2) - s\mu_0(\frac 12)} \,ds \right) e^{-C_6 t g(|y|+2)} \\
    &\geq c_6 e^{-t\mu_0(\frac 12)} \frac{  e^{-C_6 t g(|y|+2)}}{g(|x|+1/2)} f_1(|x-y|).
\end{align*}
For the proof of the first inequality above we use the fact that $|\left\{w: |y-w|<1/2, |w| \geq|y|\right\}| \geq \frac 12|B_{1/2}(y)|= \frac 12|B_{1/2}(0)|$.

If $|x| > n_0+2$ and $|y| \leq n_0+2$, a similar reasoning shows
\begin{align*}
    u_t(x,y)
    &\geq c_5  f(|x-y|) \left(\int_0^{t-T_{\bee}}  e^{- C_6 s g(|x|+1/2) - s\mu_0(\frac 12)} \,ds \right) e^{-t\mu_0(\frac 12)}e^{-\widetilde C t}  \\
    &\geq c_6 e^{-t\mu_0(\frac 12)} \frac{  e^{-\widetilde C t}}{g(|x|+1/2)} f_1(|x-y|).
\end{align*}
Also for $|x|, |y| \leq n_0+2$, we obtain
\begin{align*}
    u_t(x,y)
    &\geq c_5  f(|x-y|) \left(\int_0^{t-T_{\bee}}  e^{- \widetilde C s - s\mu_0(\frac 12)} \,ds \right) e^{-t\mu_0(\frac 12)}e^{-\widetilde C t}  \\
    &\geq c_7 e^{-t(\mu_0(\frac 12) + \widetilde C)} f_1(|x-y|).
\end{align*}
Because of the symmetry of $u_t(x,y)$, this gives the required bound for $|x-y| \geq \frac 54$.
\end{proof}

\begin{lemma}\label{lem:low2}
    Assume \textup{(\ref{A1}.a,b,c)}, \eqref{A2} and \textup{(\ref{A3}.a,b)} with $t_{\bee} >0$ and $R_0 >0$. For every $n_0 \geq R_0$ there exist constants $C, \widetilde C >0$ such that for any $t > 4t_{\bee}$ the following estimates hold.
\begin{enumerate}
\item\label{lem:low2-a}
    If $|x|, |y| \leq n_0+2$, then
    \begin{align*}
        u_t(x,y) \geq Ce^{-t(\mu_0(\frac 12)+\widetilde C)}.
    \end{align*}

\item\label{lem:low2-b}
    If $|x| > n_0+2$ and $|y| \leq n_0+2$ , then
    \begin{align*}
        u_t(x,y) \geq C e^{-t(\mu_0(\frac 12)+\widetilde C)} \frac{f(|x|)}{g(|x|+\frac 12)};
    \end{align*}
    if $|x| \leq n_0+2$ and $|y| > n_0+2$, then, by symmetry,
    \begin{align*}
        u_t(x,y) \geq C e^{-t(\mu_0(\frac 12)+\widetilde C)} \frac{f(|y|)}{g(|y|+\frac 12)}.
    \end{align*}

\item\label{lem:low2-c}
    If $|x|, |y| > n_0+2$, then
    \begin{align*}
        u_t(x,y)
        \geq C\frac{e^{-t\mu_0(\frac 12)}}{g(|y|+\frac 12)g(|x|+\frac 12)}
        \int_{\substack{|x-z|>1\\|y-z| > 1\\|z| \geq n_0+2}} f(|x-z|) f(|z-y|)e^{-C_6 tg(|z|+2)}\,dz.
    \end{align*}
\end{enumerate}
\end{lemma}

\begin{proof}
Set $\widetilde C:= \sup_{z \in B(0,n_0+6)} V_{+}(z)$ and let $t >4t_{\bee}$. The estimates in~\ref{lem:low2-a} and \ref{lem:low2-b} have already been established in Lemma~\ref{lem:low1}. The remaining assertion~\ref{lem:low2-c} can be shown by the same method which was used in the second part of the proof of Lemma~\ref{lem:low1}.
By (the analogue of) the Chapman--Kolmogorov equations, the strong Markov properties, and the Ikeda-Watanabe formula we get
\begin{align*}
    &u_t(x,y)
    \geq \ex^x\left[e^{-\int_0^{t-2t_{\bee}} V(X_s)\,ds}\, u_{2t_{\bee}}(X_{t-2t_{\bee}},y);\; t-2t_{\bee} > \tau_{B_{1/2}(x)}\right] \\
    &\quad=\ex^x\left[e^{-\int_0^{\tau_{B_{1/2}(x)}} V(X_s)\,ds}\, u_{t-\tau_{B_{1/2}(x)}}(X_{\tau_{B_{1/2}(x)}},y);\; t-2t_{\bee} > \tau_{B_{1/2}(x)}\right]\\
    &\quad\geq \int_0^{t-2t_{\bee}} \int_{B_{1/2}(x)} e^{- s \sup_{\xi \in B_{1/2}(x)} V(\xi)} \, p_{B_{1/2}(x)}(s,x,z)
    \left( \int_{\substack{|x-w| > 1\\|y-w| > 1} } \nu(z-w) u_{t-s}(w,y) \,dw \right) dz \,ds.
\end{align*}
Since $z \in B_{1/2}(x)$ and $|x-w| > 1$, we can use (\ref{A1}.a,b) to see $\nu(z-w)\leq cf(|x-w|)$; this means that we can estimate the previous expression by
\begin{gather*}
    c_1 \int_0^{t-2t_{\bee}} e^{- s \sup_{\xi \in B_{1/2}(x)} V(\xi)} \, \pr^0(s < \tau_{B_{1/2}(0)}) \int_{\substack{|x-w| > 1\\|y-w| > 1}} f(|x-w|) u_{t-s}(w,y) \,dw  \,ds.
\end{gather*}
If $|x|, |y| > n_0+2$, we use \textup{(\ref{A3}.a,b)} to see that the last expression is greater than or equal than
\begin{gather*}
    c_1 \int_0^{t-2t_{\bee}} e^{- C_6 s g(|x|+1/2)} \, \pr^0(s < \tau_{B_{1/2}(0)}) \int_{\substack{|x-w| > 1\\|y-w| > 1\\|w| \geq n_0+2}} f(|x-w|) u_{t-s}(w,y) \,dw  \,ds.
\end{gather*}
From Lemma~\ref{lem:low1} we know
\begin{gather*}
    u_{t-s}(w,y)
    \geq c_2 e^{-(t-s)\mu_0(\frac 12)} \frac{e^{- C_6(t-s) g(|w|+2)}}{g(|y|+\frac 12)} f(|y-w|),
    \quad |y-w| > 1, \; t - s > 2t_{\bee}.
\end{gather*}
Together with \eqref{eq:killed_exp} this finally gives
\begin{align*}
    &u_t(x,y)\\
    &\geq c_3\frac{e^{-t\mu_0(\frac 12)}}{g(|y|+\frac 12)} \left(\int_0^{t-2t_{\bee}} e^{- C_6 s g(|x|+\frac 12))} \,ds \right)
    \int_{\substack{|x-w| > 1\\|y-w| > 1\\|w| \geq n_0+2}} f(|x-w|) f(|w-y|) e^{- C_6 t  g(|w|+2)} \,dw \\
    &\geq c_4 \frac{e^{-t\mu_0(\frac 12)}}{g(|x|+\frac 12)g(|y|+\frac 12)}
    \int_{\substack{|x-w| > 1 \\ |y-w| > 1 \\ |w| \geq n_0+2}} f(|x-w|) f(|w-y|) e^{- C_6 t  g(|w|+2)} \,dw.
\end{align*}
This completes the proof of~\ref{lem:low2-c}.
\end{proof}

\subsection{Sharp general two-sided estimates}\label{sec4:sharp}

We are now going to show that the estimates from the two previous sections are sharp in the spatial variable \emph{if we assume \textup{(\ref{A3}.c)}} in addition to \textup{(\ref{A3}.a,b)}. These estimates lead to a considerable improvement in $t$, too. Recall that $\lambda_0 :=\inf \spec(H)$ and $\varphi_0 \in L^2(\real^d)$ denote the ground state eigenvalue and eigenfunction. Under \eqref{A3} it follows from~\cite[Cor.~2.2]{bib:KL15} that for every $R>0$ there exist constants $C, C' >0$ such that
\begin{align}\label{eq:eig_1}
    C \frac{f(|z|)}{g(|z|)}
    \leq \varphi_0(z)
    \leq C' \frac{f(|z|)}{g(|z|)},
    \quad |z| \geq R.
\end{align}
Since we always deal with a strictly positive and continuous version of $\varphi_0$, there are (possibly different) constants $C,C' >0$ such that
\begin{align}\label{eq:eig_2}
    C
    \leq \varphi_0(z)
    \leq C',
    \quad |z| \leq R.
\end{align}

The following lemma will be used to improve the last term of the estimate in Lemma \ref{lem:lem4}.\ref{lem:lem4-c}.
\begin{lemma} \label{lem:sharp1}
    Assume \eqref{A1}--\eqref{A3} with $t_{\bee}>0$ and $R_0>0$. Let $n_0 \in \nat$ be as in Lemma~\ref{lem:lem3}. If $\lambda_0>0$, then we require, additionally, that $n_0$ is so large that
    \begin{align} \label{eq:n_0_cond_stronger}
        g(n_0-2) \geq 2C_6\lambda_0.
    \end{align}
    If there is a constant $C >0$ such that
    \begin{align} \label{eq:aux_upper_bound}
        u_t(x,y)
        \leq C e^{-\lambda_0 t} f_1(|y|),
        \quad t > 5t_{\bee}, \; |x| < n_0+2, \; y \in \real^d,
    \end{align}
    then there exists a constant $\widetilde C>0$ such that for every $t > 10t_{\bee}$ and all $|x|, |y| > n_0+3$ we have
    \begin{align*}
    u_t(x,y)
    &\leq \frac{\widetilde C e^{(C_5 + C_{\ref{new}})t}}{g(|x|)g(|y|)}e^{-\frac{t-10t_{\bee}}{C_6C_7^2} g(|x| \vee |y|)} f_1(|x-y|) \\
    &\quad\mbox{}+ \frac{\widetilde C e^{(C_5 + C_{\ref{new}}) t}}{g(|x|)g(|y|)} \int_{n_0+2 < |z| < (|x|-1) \vee (|y|-1)} f_1(|x-z|) f_1(|z-y|)  e^{-\frac{t-10t_{\bee}}{2C_6C_7^2} g(|z|)}\,dz \\
    &\quad\mbox{}+ \widetilde Ce^{-\lambda_0 t}\frac{f(|x|)f(|y|)}{g(|x|)g(|y|)} .
    \end{align*}
\end{lemma}

\begin{proof}
First we prove that there exists a constant $\widetilde C>0$ such that for every $|x| > n_0+3$, $y \in \real^d$ and $t > 5t_{\bee}$ we have
\begin{align}
    u_t(x,y)
    &\leq \widetilde C e^{(C_5+C_{\ref{new}}) t}e^{-\frac{t-5t_{\bee}}{C_6} g(|x|-2)} f_1(|x-y|)  \label{eq:aux_symm}\\
    &\notag\quad\mbox{}+ \frac{\widetilde Ce^{(C_5+C_{\ref{new}}) t}}{g(|x|-2)} \int_{n_0+2 < |z| < |x|-1} f(|x-z|)f_1(|z-y|)  e^{-\frac{t-5t_{\bee}}{2C_6} g(|z|-2)}\,dz\\
    &\notag\quad\mbox{}+ \widetilde C\frac{e^{-\lambda_0t}}{g(|x|-2)} f(|x|) f_1(|y|).
\end{align}
This can be shown with the argument used in Lemma~\ref{lem:lem3}. In fact, only the estimates of the last two terms $\mathrm{I}_{k,l}$ for $k=n_0, n_0+1$ in the proof of that lemma require a modification. From now on, let $k=n_0$ or $k=n_0+1$. Recall from \eqref{eq:lem3-aux} that
\begin{align*}
    \mathrm{I}_{k,l} = \ex^x\left[e^{-\int_0^{\sigmaAkk} V(X_s) \,ds}\, u_{t-\sigmaAkk}(X_{\sigmaAkk},y);\; S(\A{n-2}{\infty}, \A{k-1}{k}, l,t-5t_{\bee})\right].
\end{align*}
Since $\sigmaAkk < t-5t_{\bee}$ on the set $S(\A{n-2}{\infty}, \A{k-1}{k}, l,t-5t_{\bee})$ and $X_{\sigmaAkk} \in B_{n_0+2}(0)$, we have by \eqref{eq:aux_upper_bound} that
\begin{gather*}
    u_{t-\sigmaAkk}(X_{\sigmaAkk},y)
    \leq C e^{ \lambda_0 \sigmaAkk} e^{-\lambda_0 t} f_1(|y|);
\end{gather*}
Consequently,
\begin{align*}
    \mathrm{I}_{k,l}
    \leq C e^{-\lambda_0 t} f_1(|y|) \, \ex^x\left[e^{-\int_0^{\sigmaAkk} (V(X_s) - \lambda_0) \,ds};\;
    S(\A{n-2}{\infty}, \A{k-1}{k}, l,t-5t_{\bee})\right].
\end{align*}
The condition \eqref{eq:n_0_cond_stronger} ensures that the shifted potential $\widetilde V:= V - \lambda_0$ appearing above also satisfies \textup{(\ref{A3}.a,b)} with the radius $\widetilde R_0 = n_0$, the profile $\widetilde g$ such that $\widetilde g|_{[0,\widetilde R_0)} \equiv 1$, $\widetilde g|_{[\widetilde R_0,\infty)} = g$, and  the constant $\widetilde C_6 := 2C_6$. Applying Lemma~\ref{lem:lem2} with $W \equiv 1$ to the latter expectation, finally gives
\begin{gather*}
    \mathrm{I}_{k,l} \leq c_1 \frac{e^{-\lambda_0 t}}{2^l g(|x|-2)} f_1(|y|) \int_{|z| \leq n_0+1} f(|x-z|)\,dz.
\end{gather*}
With (\ref{A1}.c) we conclude that
\begin{gather*}
    \mathrm{I}_{k,l} \leq c_2 \frac{e^{-\lambda_0 t}}{2^l g(|x|-2)} f(|x|) f_1(|y|)
\end{gather*}
for $k=n_0, n_0+1$ and any $l \in \nat$. This gives the required last term of the estimate \eqref{eq:aux_symm}.

In order to get the bound in a symmetric form for every $|x|, |y| > n_0+3$ and $t > 10t_{\bee}$, it is now enough to write
$u_t(x,y)  = \int_{\real^d} u_{t-5t_{\bee}}(x,z) u_{5t_{\bee}}(y,z)\,dz$ and to repeat the symmetrization argument from the proof of Lemma~\ref{lem:lem4}.\ref{lem:lem4-c}.  Here we use \eqref{eq:aux_symm} to estimate $u_{t-5t_{\bee}}(x,z)$. The uniform growth condition \textup{(\ref{A3}.c)} is used to replace $1/(g(|x|-2) g(|y|-2))$ with $1/(g(|x|) g(|y|))$.
\end{proof}

We are now in a position to prove the main result of this section. Recall that $\mu_0(r) > 0$ and $\psi_{0,r}$ are the ground state eigenvalue and eigenfunction for the process killed on leaving a ball $B_r(0)$, $r >0$.
\begin{theorem}[Sharp two-sided bounds]\label{th:th1}
    Let $H = -L+V$ be the Schr\"{o}dinger operator with confining potential $V$ such that \eqref{A1}--\eqref{A3} hold with $t_{\bee}>0$ and $R_0>0$. Denote by $\lambda_0$ and $\varphi_0$ the ground-state eigenvalue and eigenfunction, and by $u_t(x,y)$ the density of the semigroup $U_t = e^{-tH}$. Let $n_0 \in \nat$ be as in Lemma \ref{lem:lem3} and \ref{lem:sharp1} and so large that
    \begin{align} \label{eq:n_0_cond_strongest}
        g(n_0-2) \geq 12(C_5+C_{\ref{new}})C_6C_7^2 + \frac{ \mu_0(\frac 12)}{C_6C_7^2}.
    \end{align}
    There exists a constant $C \geq 1$ such that for every $t > 30t_{\bee}$ we have the following estimates.
\begin{enumerate}
\item\label{th:th1-a}
    If $|x|, |y| \leq n_0+3$, then
    \begin{align*}
         \frac{1}{C} e^{-\lambda_0 t}
         \leq u_t(x,y)
         \leq C e^{-\lambda_0 t}.
    \end{align*}
\item\label{th:th1-b}
    If $|x| > n_0+3$ and $|y| \leq n_0+3$, then
    \begin{align*}
        \frac{1}{C} e^{-\lambda_0 t} \frac{f(|x|)}{g(|x|)}
        \leq u_t(x,y)
        \leq C e^{-\lambda_0 t} \frac{f(|x|)}{g(|x|)}.
    \end{align*}
    By symmetry, if $|x| \leq n_0+3$ and $|y| > n_0+3$,
    \begin{align*}
        \frac{1}{C} e^{-\lambda_0 t} \frac{f(|y|)}{g(|y|)}
        \leq u_t(x,y)
        \leq C e^{-\lambda_0 t} \frac{f(|y|)}{g(|y|)}.
    \end{align*}

\item\label{th:th1-c}
    If $|x|, |y| > n_0+3$, then
    \begin{gather*}
        \frac{1}{C}\frac{F(\mathsf{K}t,x,y) \vee e^{-\lambda_0 t} f(|x|) f(|y|)}{g(|x|) g(|y|)} \leq u_t(x,y)
        \leq C\frac{F\left(\frac{t}{\mathsf{K}},x,y\right)\vee e^{-\lambda_0 t} f(|x|) f(|y|)}{g(|x|) g(|y|)},
    \end{gather*}
    where $\mathsf{K} = 4 C_6 C_7^2$ and
    \begin{align} \label{def:def_F}
        F(\tau,x,y) := \int_{n_0+2 < |z| < |x| \vee |y|} f_1(|x-z|) f_1(|z-y|) e^{- \tau g(|z|)}\,dz.
    \end{align}
\end{enumerate}
\end{theorem}

\begin{remark}\label{rem:th1}
Since $x\mapsto f(x)/g(x)$ is bounded from above and below (away from $0$) on compact intervals, it
is possible to combine \ref{th:th1-a} and \ref{th:th1-b} in a single estimate:
\begin{gather*}
    u_t(x,y)
    \asymp  e^{-\lambda_0 t} \left(1 \wedge \frac{f(|x|)}{g(|x|)}\right)\left(1 \wedge \frac{f(|y|)}{g(|y|)}\right)
    \quad\text{for all}\quad |x| \wedge |y| \leq n_0+3.
\end{gather*}
Moreover, due to \eqref{eq:eig_1}, \eqref{eq:eig_2}, the above two-sided estimates are equivalent to
\begin{gather*}
    u_t(x,y)
    \asymp e^{-\lambda_0 t} \varphi_0(x)\varphi_0(y),
    \quad\text{for all}\quad |x| \wedge |y| \leq n_0+3, \; t > 30t_{\bee}.
\end{gather*}
\end{remark}

\begin{proof}[Proof of Theorem \ref{th:th1}]
\ref{th:th1-a}\ \ Fix $t > 5t_{\bee}$ and let $|x|, |y| \leq n_0+3$. It follows from \eqref{eq:eig_1}, \eqref{eq:eig_2}, \textup{(\ref{A3}.c)} and the estimates in Lemma~\ref{lem:lem4}.\ref{lem:lem4-a},\ref{lem:lem4-b} applied to $u_{5t_{\bee}}(y,z)$ that
\begin{align*}
    &u_t(x,y)
    = \int_{|z| < n_0+3} u_{t-5t_{\bee}}(x,z) u_{5t_{\bee}}(z,y)\,dz + \int_{|z| \geq n_0+3} u_{t-5t_{\bee}}(x,z) u_{5t_{\bee}}(z,y)\,dz \\
    &\quad\leq \frac{\sup_{z,w \in B(0,n_0+3)} u_{5t_{\bee}}(z,w)}{\inf_{w \in B(0,n_0+3)} \varphi_0(w)}\int_{|z| < n_0+3} u_{t-5t_{\bee}}(x,z) \varphi_0(z)\,dz + c_1 \int_{|z| \geq n_0+3} u_{t-5t_{\bee}}(x,z) \varphi_0(z)\,dz \\
    &\quad\leq c_3 \int_{\real^d} u_{t-5t_{\bee}}(x,z) \varphi_0(z)\,dz\\
    &\quad \leq c_4 e^{-\lambda_0 t} \varphi_0(x).
\end{align*}
Since $|x| \leq n_0+3$, a further application of \eqref{eq:eig_2} gives the upper estimate. The lower estimate follows from similar arguments based on the lower estimates in Lemma~\ref{lem:low2}.\ref{lem:low2-a},\ref{lem:low2-b} and the two-sided bounds \eqref{eq:eig_1}, \eqref{eq:eig_2}.

\medskip
\ref{th:th1-b}\ \ Fix $t > 5t_{\bee}$. By symmetry, it is enough to assume that $|x| > n_0+3$ and $|y| \leq n_0+3$. Exactly the same argument as in part \ref{th:th1-a} shows
\begin{align*}
    c_5 e^{-\lambda_0 t} \varphi_0(x)
    \leq u_t(x,y)
    \leq c_6 e^{-\lambda_0 t} \varphi_0(x).
\end{align*}
Since now $|x| > n_0+3$, the claimed bound follows immediately from \eqref{eq:eig_1}, \eqref{eq:eig_2}.

\medskip
\ref{th:th1-c}\ \ Let $|x|, |y| > n_0+3$ and $t > 30t_{\bee}$. We begin with  the upper bound. From the already established parts~\ref{th:th1-a} and \ref{th:th1-b}, and (\ref{A3}.b) we have
\begin{gather*}
    u_t(x,y) \leq c_7 e^{-\lambda_0 t} f_1(|y|), \quad |x| < n_0+3,\; y \in \real^d,\; t > 5t_{\bee},
\end{gather*}
which is exactly \eqref{eq:aux_upper_bound}. Thus, we can use Lemma~\ref{lem:sharp1} and get
\begin{align*}
    u_t(x,y)
    &\leq \frac{c_8 e^{(C_5+C_{\ref{new}}) t}}{g(|x|)g(|y|)}e^{-\frac{t-10t_{\bee}}{C_6C_7^2} g(|x| \vee |y|)} f_1(|x-y|)  \\
    &\quad\mbox{}+ \frac{c_8 e^{(C_5+C_{\ref{new}}) t}}{g(|x|)g(|y|)} \int_{n_0+2 < |z| < |x| \vee |y|} f_1(|x-z|) f_1(|z-y|) e^{-\frac{t-10t_{\bee}}{2C_6C_7^2} g(|z|)}\,dz \\
	&\quad\mbox{}+ c_8 e^{-\lambda_0 t}\frac{f(|x|)f(|y|)}{g(|x|)g(|y|)}.
\end{align*}
Without loss of generality we may assume that $|y| \leq |x|$; the case $|y| > |x|$ follows from the fact that $x$ and $y$ play symmetric roles; set $y_0:= (|y| -1/2)|y|^{-1} y$. We have
\begin{align*}
    &\int_{n_0+2 < |z| < |x| \vee |y|} f_1(|x-z|)f_1(|z-y|) e^{-\frac{t-10t_{\bee}}{2C_6C_7^2} g(|z|)}\,dz \\
    &\quad\geq \int_{B_{1/2}(y_0)} f_1(|x-z|)f_1(|z-y|)  e^{-\frac{t-10t_{\bee}}{2C_6C_7^2} g(|z|)}\,dz \\
    &\quad\geq  c_9 e^{-\frac{t-10t_{\bee}}{2C_6C_7^2} g(|y|)} f_1(|x-y|)
\end{align*}
which gives the estimate
\begin{align*}
    \int_{n_0+2 < |z| < |x| \vee |y|} f_1(|x-z|)  f_1(|z-y|)  e^{-\frac{t-10t_{\bee}}{2C_6C_7^2} g(|z|)}\,dz
    &\geq  c_9 e^{-\frac{t-10t_{\bee}}{2C_6C_7^2} g(|x| \wedge |y|)} f_1(|x-y|)\\
    &\geq  c_9 e^{-\frac{t-10t_{\bee}}{C_6C_7^2} g(|x| \vee |y|)} f_1(|x-y|).
\end{align*}
From this we get at once
\begin{align*}
    u_t(x,y)
    &\leq \frac{c_{10} e^{(C_5+C_{\ref{new}}) t}}{g(|x|)g(|y|)} \int_{n_0+2 < |z| < |x| \vee |y|} f_1(|x-z|)f_1(|z-y|)  e^{-\frac{t-10t_{\bee}}{2C_6C_7^2} g(|z|)}\,dz \\
    &\quad\mbox{}+ c_8 e^{-\lambda_0 t}\frac{f(|x|)f(|y|)}{g(|x|)g(|y|)}.
\end{align*}
In order to complete the proof of the upper bound, it suffices to observe that for every $|z| \geq n_0+2$ and $t > 30t_{\bee}$ we have
\begin{align*}
    (C_5+C_{\ref{new}}) t - \frac{t-10t_{\bee}}{2C_6C_7^2} g(|z|)
    &\leq (C_5+C_{\ref{new}}) t - \frac{t}{3C_6C_7^2} g(|z|) + \frac{t}{4C_6C_7^2} g(|z|) -\frac{t}{4C_6C_7^2} g(|z|) \\
	&= \left(C_5+C_{\ref{new}}  - \frac{1}{12C_6C_7^2} g(|z|) \right) t  -\frac{t}{4C_6C_7^2} g(|z|)\\
    &\leq -\frac{t}{4C_6C_7^2} g(|z|);
\end{align*}
The last inequality requires \eqref{eq:n_0_cond_strongest}. This completes the proof of the upper estimate.

\medskip
Now we turn to the lower estimate. Let $|x|, |y| > n_0+3$ and $t > 30t_{\bee}$. Recall that we have by Lemma~\ref{lem:low2}.\ref{lem:low2-c} and assumption \textup{(\ref{A3}.c)}
\begin{gather*}
    u_t(x,y)
    \geq \frac{c_{11}}{g(|x|)g(|y|)} \int_{\substack{|x-z|>1 \\ |y-z| > 1 \\ |z| \geq n_0}} f(|x-z|) f(|z-y|)e^{-2C_6C_7^2 t g(|z|)} e^{\left(C_6C_7^2 g(|z|)-\mu_0(\frac 12)\right)t} \,dz.
\end{gather*}
From \eqref{eq:n_0_cond_strongest} we see that $C_6C_7^2 g(|z|)-\mu_0(\frac 12) \geq C_6C_7^2g(n_0)-\mu_0(\frac 12) \geq 0$ for $|z| \geq n_0$, and so
\begin{align} \label{eq:lower_init}
    u_t(x,y)
    \geq \frac{c_{11}}{g(|x|)g(|y|)} \int_{\substack{|x-z|>1 \\ |y-z| > 1 \\ |z| \geq n_0}} f(|x-z|) f(|z-y|)e^{-2C_6C_7^2 t g(|z|)}\,dz.
\end{align}
As before, we may assume that $|x| \leq |y|$; \eqref{eq:lower_init}, \textup{(\ref{A3}.b)}, \eqref{eq:f-f_1} and \eqref{eq:loc_comp_f1} yield
\begin{align} \label{eq:lower_member1}
    u_t(x,y)
    &\notag\geq \frac{c_{11}}{g(|x|)g(|y|)} \int_{\substack{n_0 \leq |z| \leq |x|-1 \\ |x-z|>1,\: |y-z| > 1}} f(|x-z|) f(|z-y|)e^{-2C_6C_7^2 t g(|z|)}\,dz\\
    &\notag\geq \frac{c_{11}}{g(|x|)g(|y|)} e^{-2C_6C_7^2 t g(|x|-1)} \int_{\substack{n_0 \leq |z| \leq |x|-1 \\ 2 \geq |x-z| > 1, \: |y-z| > 1}} f_1(|x-z|) f_1(|z-y|) \,dz\\
	&\geq \frac{c_{12}c_{13}}{g(|x|)g(|y|)} e^{-2C_6C_7^2 t g(|x|-1)} f_1(|x-y|),
\end{align}
where $c_{13}:=\inf_{|x|,|y| \geq n_0+3} \int_{\substack{n_0 \leq |z| \leq |x|-1 \\ 2 \geq |x-z| > 1, |y-z| > 1}} f_1(|x-z|) \,dz > 0$. Since, by \eqref{eq:conv_f1},
\begin{gather*}
    f_1(|x-y|)
    \geq c_{14} \int_{\substack{n_0+2 \leq |z| \leq |x| \vee |y| \\ |w-z| \leq 1}} f_1(|x-z|)f_1(|y-z|) \,dz,
    \quad w \in \real^d,
\end{gather*}
we finally get from \eqref{eq:lower_member1} and the monotonicity of $g$ that for $w=x$ or $w=y$
\begin{gather*}
    u_t(x,y)
    \geq \frac{c_{15}}{g(|x|)g(|y|)} \int_{\substack{n_0+2 \leq |z| \leq |x| \vee |y| \\ |w-z| \leq 1}} f_1(|x-z|)f_1(|y-z|) e^{-2C_6C_7^2 t g(|z|)} \,dz.
\end{gather*}
Together with the estimate
\begin{gather*}
    u_t(x,y)
    \geq \frac{c_{11}}{g(|x|)g(|y|)} \int_{\substack{n_0+2 \leq |z| \leq |x| \vee |y| \\ |x-z|>1, |y-z| > 1 }} f_1(|x-z|)f_1(|y-z|) e^{-2C_6C_7^2 t g(|z|)} \,dz
\end{gather*}
which comes from \eqref{eq:lower_init}, we obtain the bound
\begin{gather} \label{eq:lower-better}
    u_t(x,y) \geq \frac{c_{16}}{g(|x|)g(|y|)} F(\tfrac 12\mathsf{K}t,x,y) \geq \frac{c_{16}}{g(|x|)g(|y|)} F(\mathsf{K}t,x,y).
\end{gather}
It is now enough to show that
\begin{align} \label{eq:lower_member3}
    u_t(x,y)
    &\geq c_{17} e^{-\lambda_0 t} \frac{f(|x|)f(|y|)}{g(|x|)g(|y|)}.
\end{align}
From Lemma~\ref{lem:low2}.\ref{lem:low2-c}, \textup{(\ref{A1}.c)} and \textup{(\ref{A3}.c)}, we get for $x_0:=(R_0+3,0,...,0)$
\begin{gather} \label{eq:lower_6tb}
    u_{5t_{\bee}}(w,y)
    \geq  \frac{c_{18}}{g(|w|)g(|y|)} \int_{|x_0-z|<1} f(|w-z|) f(|z-y|) \,dz
    \geq c_{19} \frac{f(|w|)f(|y|)}{g(|w|)g(|y|)},
    \quad |w| > n_0+3.
\end{gather}
In the following calculation we estimate the first half of the integral using the lower estimate in Lemma~\ref{lem:low2}.\ref{lem:low2-b} --- in this Lemma $n_0$ was arbitrary, so we may increase it to $n_0+1$ --- \textup{(\ref{A3}.c)} and \eqref{eq:eig_2}. For the second half of the integral, we use \eqref{eq:lower_6tb} in combination with \eqref{eq:eig_1}:
\begin{align*}
    u_t(x,y)
    &\geq \int_{|w| \leq n_0+ 3} u_{t-5t_{\bee}}(x,w) u_{5t_{\bee}}(w,y)\,dw
        + \int_{|w| > n_0+3} u_{t-5t_{\bee}}(x,w) u_{5t_{\bee}}(w,y)\,dw \\
    &\geq c_{20} \frac{f(|y|)}{g(|y|)} \int_{|w| \leq n_0+3} u_{t-5t_{\bee}}(x,w) \varphi_0(w)\,dw
        + c_{21} \frac{f(|y|)}{g(|y|)} \int_{|w| > n_0+3} u_{t-5t_{\bee}}(x,w) \varphi_0(w)\,dw \\
    &\geq c_{22} \frac{f(|y|)}{g(|y|)} \int_{\real^d} u_{t-5t_{\bee}}(x,w) \varphi_0(w)\,dw \\
    &= c_{22} \frac{f(|y|)}{g(|y|)} \int_{\real^d} U_{t-5t_{\bee}}\varphi_0(x)\\
    &= c_{23} \frac{f(|y|)}{g(|y|)} e^{-\lambda_0 t} \varphi_0(x)\\
    &\geq c_{24} e^{-\lambda_0 t} \frac{f(|x|)f(|y|)}{g(|x|)g(|y|)} .
\end{align*}
In the last inequality we use \eqref{eq:eig_1} once again. This completes the proof of the lower bound in~\ref{th:th1-c}.
\end{proof}

\subsection{Sharp two-sided estimates of \boldmath$U_t\I(x)$\unboldmath}\label{sec4:Ut}

In this section we apply Theorem~\ref{th:th1} to obtain two-sided large-time estimates for the functions $U_t\I(x)$. Recall that $\left\{U_t:t \geq 0\right\}$ is the Schr\"{o}dinger semigroup with kernel $u_t(x,y)$.

\begin{theorem}\label{th:th2}
    Let $H = -L+V$ be the Schr\"{o}dinger operator with confining potential $V$ such that \eqref{A1}--\eqref{A3} hold with $t_{\bee}>0$ and $R_0>0$. Denote by $\lambda_0$ and $\varphi_0$ the ground-state eigenvalue and eigenfunction, and by $u_t(x,y)$ the density of the operator $U_t = e^{-tH}$. For $n_0 \in \nat$ large enough \textup{(}as in Theorem~\ref{th:th1}\textup{)} there exists a constant $C \geq 1$ such that for every $t > 30 t_{\bee}$ we have the following estimates.
\begin{enumerate}
\item\label{th:th2-a}
    If $|x| \leq n_0+3$, then
    \begin{align*}
         \frac{1}{C} e^{-\lambda_0 t} \leq U_t \I(x) \leq C e^{-\lambda_0 t}.
    \end{align*}

\item\label{th:th2-b}
    If $|x| > n_0+3$, then
    \begin{gather*}
        \frac{1}{C} \frac{G(\mathsf{K}t,x) \vee e^{-\lambda_0 t} f(|x|)}{g(|x|)}
        \leq U_t \I(x)
        \leq C \frac{G\left(\mathsf{K}^{-1}t,x\right) \vee e^{-\lambda_0 t} f(|x|)}{g(|x|)},
    \end{gather*}
    where $\mathsf{K} = 4 C_6 C_7^2$ and
    \begin{align} \label{def:def_G}
        G(\tau,x) := \int_{n_0+2 < |z| \leq |x| } f_1(|x-z|) e^{- \tau g(|z|)}\,dz.
    \end{align}
\end{enumerate}
\end{theorem}
\begin{proof}
    Since $U_t \I(x) = \int_{\real^d} u_t(x,y)\,dy$, $x \in \real^d$, $t >0$, all estimates follow from the estimates of the kernel $u_t(x,y)$, cf.\ Theorem~\ref{th:th1}. Recall that the lower bound of Theorem~\ref{th:th1}.\ref{th:th1-c} actually holds with $F(\frac 12 \mathsf{K}t,x,y)$, see \eqref{eq:lower-better}. We will use this fact in the following calculation. If $|x| > n_0+3$, the key step is to observe that by Tonelli's theorem
\begin{align*}
    \int_{|y| > n_0+3} \frac{F(\frac 12 \mathsf{K}t,x,y)}{g(|y|)} \,dy
    &= \int_{|y| > n_0+3 }\int_{n_0+2 < |z| < |x| \vee |y|}f_1(|x-z|)\frac{f_1(|z-y|)}{g(|y|)}  e^{- \frac 12 \mathsf{K}t g(|z|)}\,dz\,dy \\
    &\geq \int_{|z| > n_0+2} \int_{|y| > (n_0+3) \vee |z| } f_1(|x-z|)\frac{f_1(|z-y|)}{g(|y|)}  e^{- \frac 12 \mathsf{K}t g(|z|)}\,dy\,dz \\
    &\geq c_1 \int_{|z| > n_0+2} \int_{\substack{|y| > (n_0+3) \vee |z| \\ |y-z| <2}} f_1(|x-z|)\frac{e^{- \frac 12 \mathsf{K}t g(|z|)}}{g(|z|)}  \,dy\,dz
    \\
    &\geq c_2\int_{|z| > n_0+2} f_1(|x-z|) e^{- \mathsf{K} t g(|z|)}  \,dz
    \\
    &\geq c_2 G(\mathsf{K} t,x).
\end{align*}
Similarly, \eqref{def:def_F}, the monotonicity of $g$ and one more use of Tonelli's theorem, imply
\begin{gather*}
    \int_{|y| > n_0+3 } \frac{F(\mathsf{K}^{-1}t,x,y)}{g(|y|)} \,dy
    \leq \frac{\left\|f_1\right\|_1}{g(n_0+3)} \int_{|z| > n_0+ 2} f_1(|x-z|) e^{- \frac{t}{\mathsf{K}} g(|z|)}  \,dz
    \leq c_3 G(\mathsf{K}^{-1}t,x).
\end{gather*}
The last estimate is a consequence of the fact that for $|x|>n_0+3$ we have
\begin{gather*}
    \int_{|z| > n_0+2} f_1(|x-z|) e^{- \frac{t}{\mathsf{K}} g(|z|)}\,dz
    = G(\mathsf{K}^{-1}t,x) + \int_{|z| > |x|} f_1(|x-z|) e^{- \frac{t}{\mathsf{K}} g(|z|)}\,dz
\end{gather*}
and, by the monotonicity of $g$,
\begin{gather*}
    \int_{|z| > |x|} f_1(|x-z|) e^{- \frac{t}{\mathsf{K}} g(|z|)}  \,dz
    \leq c_4 e^{-\frac{t}{\mathsf{K}} g(|x|)}
    \leq c_5 \int_{n_0+2 < |z| < |x|} f_1(|x-z|) e^{- \frac{t}{\mathsf{K}} g(|z|)} \,dz.
\qedhere
\end{gather*}
\end{proof}

\subsection{Applications to asymptotic intrinsic ultracontractivity}\label{sec4:aIUC}

Under the assumption \eqref{A3} some of the aIUC results of~\cite[Corollary 3.3]{bib:KKL2018} (see also~\cite[Corollary 2.3 (2)]{bib:KL15}) can be recovered from our two-sided estimates of the kernels $u_t(x,y)$. We continue to use the functions $F(\tau,x,y)$ and $G(\tau,x)$ introduced in  \eqref{def:def_F} and \eqref{def:def_G}.

\begin{lemma}\label{lem:aiuc_est}
   Let $f:(0,\infty) \to (0,\infty)$ and $g:[0, \infty) \to (0,\infty)$ be profile functions as in \textup{(\ref{A1})} and \textup{(\ref{A3})}, and let $n_0 \geq R_0+2$. Suppose that there exist $C>0$ and $R_1 \geq R_0$ such that $C g(r) \geq |\log f(r)|$, $r \geq R_1$. We have the following estimates.
\begin{enumerate}
\item\label{lem:aiuc_est-a}
    There is a constant $\widetilde C \geq 1$ such that for every $\tau \geq 3C $ and $ |x|,|y| \geq n_0+3$ we have
    \begin{gather*}
        \frac{1}{\widetilde C} e^{-\tau g(n_0+3)} f(|x|)f(|y|) \leq F(\tau,x,y) \leq \widetilde C e^{-\frac{1}{3}\tau g(n_0+2)} f(|x|)f(|y|).
    \end{gather*}
\item\label{lem:aiuc_est-b}
     There is a constant $\widetilde C \geq 1$ such that for every $\tau \geq 2C $ and $ |x| \geq n_0+3$ we have
    \begin{gather*}
        \frac{1}{\widetilde C} e^{-\tau g(n_0+3)} f(|x|) \leq G(\tau,x) \leq \widetilde C e^{-\frac{1}{2}\tau g(n_0+2)} f(|x|).
    \end{gather*}
\end{enumerate}
\end{lemma}

\begin{proof}
The lower bound follows from
\begin{align*}
    F(\tau,x,y)
    &= \int_{n_0+2 < |z| < |x| \vee |y|} f_1(|x-z|) f_1(|z-y|)  e^{- \tau g(|z|)}\,dz \\
	&\geq \int_{n_0+2 < |z| < n_0+3} f_1(|x-z|) f_1(|z-y|)  e^{- \tau g(|z|)}\,dz \\
	&\geq c_1 e^{- \tau g(n_0+3)} f(|x|) f(|y|).
\end{align*}
In a similar way we get
\begin{align*}
    G(\tau,x)
    = \int_{n_0+2 < |z| < |x|} f_1(|x-z|) e^{- \tau g(|z|)}\,dz
	\geq c_2 e^{- \tau g(n_0+3)} f(|x|).
\end{align*}

Let us establish the upper bounds. We give only details for  $F(\tau,x,y)$, since $G(\tau,x)$ can be dealt with in a similar fashion. We set
\begin{gather*}
    F(\tau,x,y)
    = \left(\int_{\substack{n_0+2 < |z| < |x| \vee |y| \\ |z| < R_1}} + \int_{\substack{n_0+2 < |z| < |x| \vee |y| \\ |z| \geq R_1}} \right)
    f_1(|x-z|) f_1(|z-y|)  e^{- \tau g(|z|)}\,dz
\end{gather*}
and denote the two integrals by $\mathrm{I}$ and $\mathrm{II}$. Clearly, $\mathrm{I} \leq c_3 e^{-\tau g(n_0+2)} f(|x|) f(|y|)$.
By assumption, $e^{-\frac{1}{3}\tau g(r)} \leq f_1(r) $, for every $\tau \geq 3C$ and $r \geq R_1$. Hence,
\begin{gather*}
    \mathrm{II}
    \leq c_4 e^{-\frac{1}{3}\tau g(n_0+2)} \int_{\real^d}  f_1(|x-z|) f_1(|z|)
    f_1(|z-y|) f_1(|z|) \,dz.
\end{gather*}
From \eqref{eq:conv_f1_pw} and \eqref{eq:conv_f1} we easily get $\mathrm{II} \leq c_5 e^{-\frac{1}{3}\tau g(n_0+2)} f(|x|) f(|y|)$. This completes the proof.
\end{proof}

The next corollary contains equivalent conditions for the aIUC property of the semigroup $\left\{U_t:t \geq 0\right\}$. Due to \cite[Corollary 3.3]{bib:KKL2018} these are in fact also equivalent conditions for the intrinsic hypercontractivity. That means, in particular, that aIUC and intrinsic hypercontractivity coincide in this setting. Recall that every aIUC semigroup is automatically pIUC with the threshold function $r \equiv \infty$, cf.\ Definition~\ref{def:piuc}.
\begin{corollary}\label{cor:aIUC}
Assume \eqref{A1}--\eqref{A3} with $t_{\bee}>0$ and $R_0>0$. The following statements are equivalent.
\begin{enumerate}
\item\label{cor:aIUC-a}
    There exist $C > 0$ and $R_1>0$ such that $V(x) \geq C |\log\nu(x)|$ for $|x| \geq R_1$.
\item\label{cor:aIUC-b}
    There exist $\widetilde C > 0$ and $\widetilde R_1>0$ such that $g(r) \geq \widetilde C |\log f(r)|$ for $r \geq \widetilde R_1$.
\item\label{cor:aIUC-c}
    There exists some $t_0 > 0$ such that
   \begin{gather*}
        u_t(x,y)
        \asymp  e^{-\lambda_0 t} \left(1 \wedge \frac{f(|x|)}{g(|x|)}\right)\left(1 \wedge \frac{f(|y|)}{g(|y|)}\right),
        \quad x,y \in \real^d, \; t \geq t_0,
	\end{gather*}
	or, equivalently,
	\begin{gather*}
        u_t(x,y) \asymp e^{-\lambda_0 t} \varphi_0(x) \varphi_0(y),
        \quad x,y \in \real^d, \; t \geq t_0,
    \end{gather*}
    i.e.\ the semigroup $\left\{U_t: t \geq 0\right\}$ is asymptotically intrinsic ultracontractive \textup{(}aIUC\textup{)}.
\item\label{cor:aIUC-d}
	There exists some $t_0 > 0$ such that
   \begin{gather*}
        U_t \I(x)
        \asymp e^{-\lambda_0 t} \left(1 \wedge \frac{f(|x|)}{g(|x|)}\right),
        \quad x \in \real^d, \; t \geq t_0.
    \end{gather*}
    or, equivalently,
   \begin{gather*}
        U_t \I(x)
        \asymp e^{-\lambda_0 t} \varphi_0(x),
        \quad x \in \real^d, \; t \geq t_0.
    \end{gather*}
    i.e., the semigroup $\left\{U_t: t \geq 0\right\}$ is asymptotically ground state dominated.
\end{enumerate}
More precisely, if \ref{cor:aIUC-b} is true for some $\widetilde C>0$, then \ref{cor:aIUC-c} and \ref{cor:aIUC-d} hold with $t_0 = 30t_{\bee} + 3 \widetilde C\mathsf{K}$.
\end{corollary}
\begin{proof}
    The statements \ref{cor:aIUC-a} and \ref{cor:aIUC-b} are equivalent because of (\ref{A1}.a) and (\ref{A3}.a);
    \ref{cor:aIUC-b} implies \ref{cor:aIUC-c} because of the estimates in Theorem \ref{th:th1} (see also Remark~\ref{rem:th1}) and Lemma \ref{lem:aiuc_est}.\ref{lem:aiuc_est-a};
    \ref{cor:aIUC-c} implies \ref{cor:aIUC-d} by integration;
    \ref{cor:aIUC-b} follows from \ref{cor:aIUC-d} with the estimates from Theorem \ref{th:th2}. Indeed, with the upper bound in \ref{cor:aIUC-d}, the lower bound in Theorem \ref{th:th2}.\ref{th:th2-b}, the definition \eqref{def:def_G} of $G$ and (\ref{A3}.b), we get for $|x|$ large enough
    \begin{gather*}
        \frac{c_1 e^{-\mathsf{K}t_0 g(|x|)}}{ g(|x|)}
        \leq \frac{c_2}{ g(|x|)} G(\mathsf{K}t_0,x)
        \leq U_{t_0} \I(x)
        \leq c_3 e^{-\lambda_0 t_0} \frac{f(|x|)}{g(|x|)},
    \end{gather*}
    which implies \ref{cor:aIUC-b}. Alternatively, we can use the argument from the proof of \cite[Theorem 2.6 (2)]{bib:KL15} to see that \ref{cor:aIUC-d} gives \ref{cor:aIUC-b}.
\end{proof}

\subsection{Applications to spectral functions}\label{sec4:spectre}
A further application of our results is the study of spectral regularity of compact semigroups, e.g.\ the heat trace or the heat content. We are not aware of such results in the literature. Recall that $U_t$ is said to be a/have a
\begin{align}
\tag{TC}\label{tc}
    &\text{\emph{trace class operator} if\ \ } \int_{\real^d} u_t(x,x) \,dx < \infty;\\
\tag{HS}\label{hs}
    &\text{\emph{Hilbert-Schmidt operator} if\ \ } \int_{\real^d}\int_{\real^d} u_t^2(x,y) \,dx \,dy < \infty;\\
\tag{fHC}\label{fhc}
    &\text{\emph{finite heat content} if\ \ } \int_{\real^d}\int_{\real^d} u_t(x,y) \,dx \,dy < \infty.
\end{align}
By (the analogue of) the Chapman--Kolmogorov equations, if the integrals in \eqref{tc}, \eqref{hs} and \eqref{fhc} are finite for \emph{some} $t_0>0$, then they are finite for \emph{all} $t \geq t_0$.

Using our bounds on $u_t(x,y)$, we can give a necessary and sufficient condition for the spectral properties \eqref{tc}, \eqref{hs} and \eqref{fhc}. Recall that $\mathsf{K} = 4 C_6 C_7^2$.

\begin{corollary} \label{cor:spec_reg}
    Assume \eqref{A1}--\eqref{A3} with $t_{\bee}>0$ and $R_0>0$. For large times the properties \eqref{tc}, \eqref{hs} and \eqref{fhc} coincide and they are equivalent to the condition
    \begin{align} \label{eq:exp_int}
        \text{there exists $s > 0$ such that\ \ } \int_{|x|>R_0} e^{-sV(x)} \,dx < \infty.
    \end{align}
    More precisely, the following assertions hold.
    \begin{enumerate}
    \item\label{cor:spec_reg-a}
        If \eqref{eq:exp_int} is true with some $s > 0$, then the integrals in \eqref{tc}, \eqref{hs} and \eqref{fhc} are finite for all $t \geq  30t_{\bee}+\mathsf{K} s$.

    \item\label{cor:spec_reg-b}
        If the integral in \eqref{tc} or \eqref{fhc} is finite for some $t>0$, then \eqref{eq:exp_int} holds for all $s \geq  \mathsf{K}(30t_{\bee} + t)$.

    \item\label{cor:spec_reg-c}
        If the integral in \eqref{hs} is finite for some $t>0$, then \eqref{eq:exp_int} holds for all $s \geq 2\mathsf{K}(30t_{\bee}+t)$.
    \end{enumerate}
\end{corollary}
It is somewhat surprising that the validity of \eqref{tc}, \eqref{hs} and \eqref{fhc} only depends on the potential, but not on the free process. The free process seems only to determine the threshold $t_{\bee}$.

\section{Applications to specific classes of nonlocal Schr\"odinger operators}\label{sec5}

\subsection{Potentials as functions of \boldmath $|\log \nu|$. \unboldmath aIUC- and non-aIUC-regime}\label{sec5:log}
Up to now we have studied Schr\"{o}dinger operators whose L\'evy density $\nu$ and potential $V$ are controlled by profiles $f$ and $g$ which satisfy the assumptions \eqref{A1} and \eqref{A3} with some $R_0>0$. From now on we assume, in addition, \eqref{A4} which says that $f$ satisfies $f(R_0)<1$ and $f$ and $g$ are connected via
\begin{gather*}
    g(r) = h\left(|\log f(r)|\right), \quad r \geq R_0;
\end{gather*}
$h:[|\log f(R_0)|,\infty) \to (0,\infty)$ is an increasing function such that $h(s)/s$ is monotone (increasing or decreasing). We will always assume that monotone functions are right-continuous.

\begin{remark} \label{rem:regimes}
If \eqref{A4} holds, the Schr\"{o}dinger operators with profiles $f$ and $g$ are divided into two distinct classes:
\begin{enumerate}
\item \label{rem:regimes-a}
    \textbf{(aIUC-regime)}: \emph{there exist $C$ and $R \geq R_0$ such that $C g(r) \geq |\log f(r)|$, $r \geq R$.}

    This is equivalent to the property that $h(s)/s$ is, outside some compact set, bounded below by a strictly positive constant.
    Since $h(s)/s$ is monotone this is the same as to say that $\lim_{s\to\infty} h(s)/s \in (0,\infty]$. If necessary, we may increase the constant $C$ to ensure that $C g(r) \geq |\log f(r)|$, $r \geq R_0$; this is equivalent to $e^{-\tau g(r)} \leq f(r)$, for every $r \geq R_0$ and $\tau \geq \tau_0:=C$.

\item \label{rem:regimes-b}
    \textbf{(non-aIUC-regime)}: \emph{$\lim_{r\to\infty} g(r)/|\log f(r)| = 0$.}

    This is equivalent to $\lim_{s\to\infty} h(s)/s = 0$. Since $h(s)/s$ is monotone, the limit is actually an infimum. This class will be discussed in Lemma~\ref{lem:non-aIUC} below.
\end{enumerate}
\end{remark}

We will use Remark~\ref{rem:regimes} as definition of the aIUC- and non-aIUC-regimes. Since $h(s)/s$ is a monotone function, the two cases in Remark~\ref{rem:regimes} are indeed complementary and exhaustive classes.

The following fact will be crucial for our further investigations. It explains the relation between the profile functions $f$ and $g$ in the non-aIUC-regime.

\begin{lemma}\label{lem:non-aIUC}
    Let $f:(0,\infty) \to (0,\infty)$ be a decreasing function such that $\lim_{r\to \infty}f(r) = 0$, pick $R_0>0$ such that $f(R_0) < 1$ and let $g(r) = h\left(|\log f(r)|\right)$, $r \geq R_0$, with an increasing function $h:[|\log f(R_0)|,\infty) \to (0,\infty)$ such that $h(s)/s$ is decreasing and $\lim_{s\to\infty} h(s)/s = 0$. Define
    \begin{gather}\label{def:initial_time}
        \ttau{r} := \frac{|\log f(r)|}{ h(|\log f(r)|)}, \quad r \geq R_0,
    \intertext{and its right-continuous generalized inverse}\label{def:moving_boundary}
        \ittau(\tau) := \Lambda^{-1}(\tau) := \inf \left\{ r \geq R_0: \ttau{r} > \tau\right\}, \quad \tau \geq \ttau{R_0}.
    \end{gather}
    The function $\ittau(\cdot)$ is increasing, satisfies $\lim_{\tau\to\infty} \ittau(\tau) = \infty$, and for every $\tau > \ttau{R_0}$
    \begin{gather*}
        e^{-\tau g(r) } \leq  f(r), \quad r \in [R_0,\ittau(\tau)),
    \intertext{resp.,}
        e^{-\tau g(r) } \geq  f(r), \quad r \in [\ittau(\tau),\infty).
    \end{gather*}
    Moreover, the function $r \mapsto {e^{-\tau g(r) }}/{f(r)}$ is increasing on $[\ittau(\tau),\infty)$.
\end{lemma}

\begin{proof}
Since $h(s)/s$ decreases to zero as $s \to \infty$ and $|\log f(r)|$ increases to $\infty$ as $r \to \infty$, we see that $\ittau(\tau)$ also increases to $\infty$ as $\tau \to \infty$. Moreover, we have
\begin{gather*}
    \frac{e^{-\tau g(r) }}{f(r)}
    = e^{|\log f(r)| - \tau h(|\log f(r)|)}
    = e^{|\log f(r)|\left(1 - \frac{ h(|\log f(r)|)}{|\log f(r)|} \tau\right)},
    \quad r \geq R_0, \; \tau > \ttau{R_0}.
\end{gather*}

From $h(|\log f(r)|)/|\log f(r)| \geq 1/\tau$ for $r \in [R_0,\ittau(\tau))$, we get  ${e^{-\tau g(r) }}/{f(r)}\leq 1$, and the first inequality holds.

Similarly, $h(|\log f(r)|)/|\log f(r)| \leq 1/\tau$ for $r \geq \ittau(\tau)$, which gives the second estimate.

The last assertion follows from the fact that $|\log f(r)|$ is an increasing function on $[R_0,\infty)$ and that the function $s \mapsto s\left(1-\tau (h(s)/s)\right)$ is positive and increasing on $[|\log f(\ittau(\tau))|, \infty)$.
\end{proof}

From now on we will choose $R_0$ and $n_0$, depending on $g$ and $f$ in such a way that all results from Section~\ref{sec3} and Section~\ref{sec4} become available. This means, in particular, that $R_0$ and $n_0$ are so large that
\begin{gather}\label{con-n0}
\left\{\begin{aligned}
    f(R_0) &< 1 && \text{($R_0$ from \eqref{A3}, \eqref{A4})},\\
    n_0 &\geq R_0+2 && \text{(cf.\ Lemma~\ref{lem:lem2})},\\
    g(n_0-2) &\geq 2C_6 \left[\vartheta_0 + 2C_3C_{\ref{lem:lem1}}(1+f(1))\right] && \text{(cf.\ Lemma~\ref{lem:lem2})},\\
    g(n_0-2) &\geq 2C_6\lambda_0 && \text{(cf.\ Lemma~\ref{lem:sharp1})},\\
    g(n_0-2) &\geq 12(C_5+C_{\ref{new}})C_6C_7^2 + \frac{ \mu_0(\frac 12)}{C_6C_7^2} && \text{(cf.\ Theorem~\ref{th:th1})}.
\end{aligned}\right.
\end{gather}
Throughout, we will also use the function $\ttau\cdot$ and its generalized inverse $\ittau(\cdot)$ introduced in Lemma~\ref{lem:non-aIUC}. Note that $\ttau{\ittau(\tau)}\geq\tau$ and $\ittau(\ttau r)\geq r$. This means, in particular, that
\begin{gather*}
    \tau  \geq 2\ittau(n_0+4)
    \implies \ttau{\tfrac 12\tau} \geq \ttau{\ittau(n_0+4)} \geq n_0+4 > n_0+3.
\end{gather*}

\subsection{Progressive intrinsic ultracontractivity}\label{sec5:piuc}
Our results indicate that it makes sense to consider a new type of intrinsic contractivity property for Schr\"odinger semigroups which is essentially weaker than aIUC. In this section we identify and discuss the concept of progressive intrinsic ultracontractivity (pIUC), cf.~Definition~\ref{def:piuc}.

Recall that our bounds for $u_t(x,y)$ and $U_t \I(x)$ are given in terms of the functions $F(\tau,x,y)$ and $G(\tau,x)$ defined in \eqref{def:def_F} and \eqref{def:def_G}.

\begin{lemma}\label{lem:progress_est}
    Let $f:(0,\infty) \to (0,\infty)$ be a profile function satisfying \textup{(\ref{A1}.b)} and \textup{(\ref{A1}.d)}. Moreover, let $g(r) = h\left(|\log f(r)|\right)$, $r \geq R_0$, with an increasing function $h:[|\log f(R_0)|,\infty) \to (0,\infty)$ such that $h(s)/s$ is decreasing and $\lim_{s\to\infty}h(s)/s = 0$.
    \begin{enumerate}
    \item\label{lem:progress_est-a}
        There are constants $\Ci,\Cii >0$ such that for every $\tau \geq 2\ttau{n_0+4}$ and $n_0+3 < |x| < \ittau(\tau/2)$
        \begin{gather*}
            \Ci  e^{-\tau g(n_0+3)} f(|x|) \leq G(\tau,x) \leq \Cii  e^{-\frac{1}{2}\tau g(n_0+2)} f(|x|).
        \end{gather*}
    \item\label{lem:progress_est-b}
        There are constants $\Ci, \Cii  >0$ such that for every $\tau \geq 3\ttau{n_0+4}$ and $|x|,|y| > n_0+3$, $|x| \wedge |y| < \ittau(\tau/3)$
        \begin{gather*}
            \Ci  e^{-\tau g(n_0+3)} f(|x|)f(|y|) \leq F(\tau,x,y) \leq \Cii  e^{-\frac{1}{3}\tau g(n_0+2)} f(|x|)f(|y|).
        \end{gather*}
    \end{enumerate}
\end{lemma}

\begin{proof}
\ref{lem:progress_est-a}\ \ We have $\ttau{n_0+4} \geq \ttau{R_0}$ and $\ittau(\frac 12\tau) > n_0+3$ for $\tau \geq 2\ttau{n_0+4}$. If $|x| < \ittau(\frac 12\tau)$, then by Lemma~\ref{lem:non-aIUC} we have $e^{- \frac{1}{2}\tau  g(|z|)} \leq f(|z|)$, for $n_0 +2 \leq |z| \leq |x|$ and $\tau \geq 2\ttau{n_0+4}$. From the definition of $G(\tau,x)$ we get
\begin{align*}
    G(\tau,x)
    &\leq e^{- \frac{1}{2}\tau g(n_0+2)} \int_{n_0+2 < |z| < |x| } f_1(|x-z|) e^{- \frac{1}{2}\tau  g(|z|)}\,dz \\
    &\leq e^{- \frac{1}{2}\tau g(n_0+2)} \int_{n_0+2 < |z| < |x| } f_1(|x-z|) f(|z|) \,dz,
\end{align*}
and from \textup{(\ref{A1}.d)}
\begin{gather*}
    G(\tau,x) \leq c_1 e^{- \frac{1}{2}\tau g(n_0+2)}  f(|x|),
\end{gather*}
for all $n_0+3 < |x| < \ittau(\frac 12\tau)$ and $\tau \geq 2\ttau{n_0+4}$. The lower bound is easier:
\begin{align*}
    G(\tau,x)
    &= \int_{n_0+2 < |z| < |x| } f_1(|x-z|) e^{- \tau g(|z|)}\,dz \\
    &\geq \int_{n_0+2 < |z| < n_0+3 } f_1(|x-z|) e^{- \tau g(|z|)}\,dz \\
    &\geq c_2 e^{- \tau g(n_0+3)} f(|x|).
\end{align*}

\smallskip\noindent
\ref{lem:progress_est-b}\ \ Without loss of generality we may assume that $n_0+3 < |y| < \ittau(\frac 13\tau)$, $\tau \geq 3\ttau{n_0+4}$ and $ |y| \leq |x|$. From the definition of $F(\tau,x,y)$ we get
\begin{gather*}
    F(\tau,x,y)
    = \left(\int_{n_0+2 < |z| < |y|}  +  \int_{|y| \leq |z| < |x|}  \right)f_1(|x-z|) f_1(|y-z|) e^{- \tau g(|z|)}\,dz
    =: \mathrm{I} + \mathrm{II}.
\end{gather*}
The monotonicity of $g$, \eqref{eq:conv_f1}, the inequality $e^{- \frac{1}{3}\tau  g(|y|)} \leq f(|y|)$ and \eqref{eq:conv_f1_pw} imply
\begin{align*}
    \mathrm{II}
    \leq c_3 e^{- \tau g(|y|)} f_1(|x-y|)
    &\leq c_3 e^{- \frac{1}{3}\tau  g(n_0+2)} f(|y|) f(|y|) f_1(|x-y|)\\
    &\leq c_4 e^{- \frac{1}{3}\tau g(n_0+2)} f(|x|) f(|y|).
\end{align*}
We still need to estimate $\mathrm{I}$. Since $e^{- \frac{1}{3}\tau  g(|z|)} \leq f(|z|)$, for $|z| < |y| < \ittau(\frac 13\tau)$, we get with \eqref{eq:conv_f1} and \eqref{eq:conv_f1_pw},
\begin{align*}
    \mathrm{I}
    &\leq e^{- \frac{1}{3}\tau g(n_0+2)} \int_{n_0+2 < |z| < |y| } f_1(|x-z|) f_1(|y-z|) e^{- \frac{2}{3}\tau  g(|z|)}\,dz \\
    &\leq e^{- \frac{1}{3}\tau g(n_0+2)} \int_{n_0+2 < |z| < |y| } f_1(|x-z|) f(|z|) f_1(|y-z|) f(|z|) \,dz \\
    &\leq c_5 e^{- \frac{1}{3}\tau g(n_0+2)} f(|x|) \int_{n_0+2 < |z| < |y| } f_1(|y-z|) f(|z|) \,dz \\
    &\leq c_6 e^{- \frac{1}{3}\tau g(n_0+2)} f(|x|) f(|y|),
\end{align*}
which yields the upper bound in~\ref{lem:progress_est-b}. The lower bound is again simpler:
\begin{align*}
    F(\tau,x,y)
    &= \int_{n_0+2 < |z| < |x| } f_1(|x-z|) f_1(|y-z|) e^{- \tau g(|z|)}\,dz \\
    &\geq \int_{n_0+2 < |z| < n_0+3 } f_1(|x-z|) f_1(|y-z|) e^{- \tau g(|z|)}\,dz\\
    &\geq c_7 e^{- \tau g(n_0+3)} f(|x|)f(|y|).
\qedhere
\end{align*}
\end{proof}

It is clear from Definition~\ref{def:piuc} that every aIUC-semigroup $\left\{U_t:t \geq 0\right\}$ is also pIUC for the threshold function $r \equiv \infty$. We will now show that under the assumptions \eqref{A1}--\eqref{A4} every non-aIUC semigroup $\left\{U_t:t \geq 0\right\}$ is still pIUC. This is the main result of this section. Recall that $\lambda_0$ and $\varphi_0$ denote the ground state eigenvalue and eigenfunction.
\begin{corollary}\label{cor:cor_prog}
    Assume \textup{\eqref{A1}--\eqref{A4}} with $t_{\bee}>0$. Set $\mathsf{K_1} := 2 \mathsf{K} = 8 C_6 C_7^2$ and $\mathsf{K_2} := 3 \mathsf{K} = 12 C_6 C_7^2$. Let $\lim_{r\to\infty} g(r)/|\log f(r)| = 0$, i.e.\ we are in the non-aIUC-regime.
    \begin{enumerate}
    \item\label{cor:cor_prog-a}
    For every $t > \max\{30t_{\bee}, \, \mathsf{K_1}\ttau{n_0+4}\}$
    \begin{gather*}
        U_t \I(x)  \asymp  e^{-\lambda_0 t} \left(1 \wedge \frac{f(|x|)}{g(|x|)}\right),
        \quad |x| < \ittau(t/\mathsf{K_1}).
    \end{gather*}
    Equivalently, for  $t >\max\{30t_{\bee},\, \mathsf{K_1}\ttau{n_0+4}\}$
    \begin{gather*}
        U_t \I(x) \asymp e^{-\lambda_0 t} \varphi_0(x),
        \quad |x| < \ittau(t/\mathsf{K_1}).
    \end{gather*}

    \item\label{cor:cor_prog-b}
    \textup{\textbf{(pIUC)}} For every $t > \max\{30t_{\bee},\, \mathsf{K_2}\ttau{n_0+4}\}$
        \begin{gather*}
            u_t(x,y) \asymp e^{-\lambda_0 t} \left(1 \wedge \frac{f(|x|)}{g(|x|)}\right) \left(1 \wedge \frac{f(|x|)}{g(|x|)}\right),
            \quad |x| \wedge |y| < \ittau(t/\mathsf{K_2}) .
        \end{gather*}
        Equivalently, for $t > \max\{30t_{\bee},\, \mathsf{K_2} \ttau{n_0+4}\}$
        \begin{gather*}
            u_t(x,y) \asymp  e^{-\lambda_0 t} \varphi_0(x) \varphi_0(y),
            \quad  |x| \wedge |y| < \ittau(t/\mathsf{K_2}).
        \end{gather*}
    \end{enumerate}
\end{corollary}

\begin{proof}
    The estimates for the functions $G(\tau,x)$ and $F(\tau,x,y)$ from Lemma~\ref{lem:progress_est} allow us to simplify the bounds in Theorems~\ref{th:th1}.\ref{th:th1-c} and~\ref{th:th2}.\ref{th:th2-b}. Since $\lim_{r\to\infty} g(r) = \infty$, there is some large $n_0$ such that $e^{-\mathsf{K_1}tg(n_0+3)} \vee e^{-(t/\mathsf{K_1})g(n_0+2)} \leq e^{-\lambda_0 t}$ and $e^{-\mathsf{K_2}tg(n_0+3)} \vee e^{-(t/\mathsf{K_2})g(n_0+2)} \leq e^{-\lambda_0 t}$, for every $t > 0$.

    The lower bounds are obtained directly by taking $\tau= \mathsf{K_i}t$, while the upper bounds follow by taking $\tau = t/\mathsf{K_i}$.

    The alternative equivalent statements are a consequence of the two-sided bound $\varphi_0(x) \asymp 1 \wedge \frac{f(|x|)}{g(|x|)}$ which is valid for all $x \in \real^d$.
\end{proof}

\subsection{Explicit estimates of \boldmath$U_t\I(x)$\unboldmath}\label{sec5:Ut}
Under the assumption \eqref{A4} we can find explicit two-sided estimates for the function $G(\tau,x)$ (defined in \eqref{def:def_G}) for the full range of $(\tau,x)$. Throughout we use $\ttau\cdot$ and $\ittau(\cdot)$ from Lemma~\ref{lem:non-aIUC} and choose $R_0$ and $n_0$ according to \eqref{con-n0}.

\begin{lemma}\label{lem:expl_est_G}
    Let $f:(0,\infty) \to (0,\infty)$ be a profile function satisfying \textup{(\ref{A1}.b)} and \textup{(\ref{A1}.d)}. Assume that $g(r) = h\left(|\log f(r)|\right)$, $r \geq R_0$, with an increasing function $h:[|\log f(R_0)|,\infty) \to (0,\infty)$ such that $h(s)/s$ is monotone.
    \begin{enumerate}
    \item\label{lem:expl_est_G-a}
        If $C>0$ is such that $Cg(r) \geq |\log f(r)|$, $r \geq R_0$, then there are constants $\Ci, \Cii  >0$ such that for every $ |x| > n_0+3$ and $\tau \geq 2 \tau_0 := C$
        \begin{gather*}
            \Ci  e^{-\tau g(n_0+3)} f(|x|)
            \leq G(\tau,x)
            \leq \Cii  e^{-\frac{1}{2}\tau g(n_0+2)} f(|x|).
        \end{gather*}

    \item\label{lem:expl_est_G-b}
        If $\lim_{r\to\infty} g(r)/|\log f(r)| = 0$, then there are constants $\Ci, \Cii  >0$ such that for every $\tau \geq 2 \ttau{n_0+4}$

        \smallskip

        \begin{enumerate}
        \item\label{lem:expl_est_G-b1}
        $\displaystyle\Ci  e^{-\tau g(n_0+3)} f(|x|) \leq G(\tau,x) \leq \Cii  e^{-\frac{1}{2}\tau g(n_0+2)} f(|x|)$
        for $ n_0+3 < |x| < \ittau(\tau/2)$.

        \medskip

        \item\label{lem:expl_est_G-b2}
        $\displaystyle\Ci  e^{-\tau g(|x|)} \leq G(\tau,x) \leq \Cii  e^{- \frac{1}{2}\tau g(n_0+2)} e^{-\frac{1}{2}\tau g(|x|)}$
        for $|x| \geq \ittau(\tau/2)$.
        \end{enumerate}
    \end{enumerate}
\end{lemma}

\begin{proof}
Part~\ref{lem:expl_est_G-a} follows directly from Lemma~\ref{lem:aiuc_est}.\ref{lem:aiuc_est-b} and part~\ref{lem:expl_est_G-b1} is exactly Lemma~\ref{lem:progress_est}.\ref{lem:progress_est-a}. We only need to show~\ref{lem:expl_est_G-b2}. From the definition \eqref{def:def_G} of $G$ we get
\begin{gather*}
    G(\tau,x)
    = \int_{n_0+2 < |z| < \ittau(\frac 12 \tau)} f_1(|x-z|) e^{- \tau g(|z|)}\,dz +  \int_{\ittau(\frac 12 \tau) \leq |z| \leq |x|} f_1(|x-z|) e^{- \tau g(|z|)}\,dz
    =: \mathrm{I} + \mathrm{II}.
\end{gather*}
Exactly the same argument as in Lemma~\ref{lem:progress_est}.\ref{lem:progress_est-a} yields
\begin{gather*}
    \mathrm{I}
    \leq e^{- \frac{1}{2}\tau g(n_0+2)} \int_{n_0+2 < |z| < \ittau(\frac 12 \tau)} f_1(|x-z|) f(|z|)\,dz
    \leq c_2  e^{- \frac{1}{2}\tau g(n_0+2)} f(|x|).
\end{gather*}
Since $|x| \geq \ittau(\frac 12 \tau)$,the second inequality in Lemma~\ref{lem:non-aIUC} shows $\mathrm{I} \leq c_2 e^{- \frac{1}{2}\tau g(n_0+2)} e^{- \frac{1}{2}\tau g(|x|)}$. In order to estimate $\mathrm{II}$ we write
\begin{gather*}
    \mathrm{II}
    \leq  e^{- \frac{1}{2}\tau g(n_0+2)} \int_{\ittau(\frac 12 \tau) < |z|
    \leq |x|} f_1(|x-z|) f(|z|) \, \frac{e^{- \frac{1}{2}\tau g(|z|)}}{f(|z|)} \,dz
\end{gather*}
which is less than or equal to
\begin{gather*}
    e^{- \frac{1}{2}\tau g(n_0+2)} \frac{e^{- \frac{1}{2}\tau g(|x|)}}{f(|x|)} \int_{n_0+2 < |z| \leq |x|} f_1(|x-z|) f(|z|)  \,dz
    \leq c_3 e^{- \frac{1}{2}\tau g(n_0+2)} e^{- \frac{1}{2}\tau g(|x|)},
\end{gather*}
by the last (monotonicity) assertion in Lemma~\ref{lem:non-aIUC} and \eqref{eq:conv_f1}. This gives the upper bound in~\ref{lem:expl_est_G-b2}. The corresponding lower estimate follows directly:
\begin{gather*}
    G(\tau,x)
    \geq e^{- \tau g(|x|)} \int_{n_0+2 < |z| \leq |x|} f_1(|x-z|) \,dz
    \geq c_4 e^{- \tau g(|x|)}.
\qedhere
\end{gather*}
\end{proof}

We are now ready to state the following corollary to Lemma~\ref{lem:expl_est_G} which simplifies the estimates of the function $U_t\I(x)$ in Theorem~\ref{th:th2}.\ref{th:th2-b}.

\begin{corollary}\label{cor:cor3}
    Assume \eqref{A1}--\eqref{A4} with $t_{\bee}>0$ and set $\mathsf{K_1} := 2 \mathsf{K} = 8 C_6 C_7^2$.
    \begin{enumerate}
    \item\label{cor:cor3-a}
        \emph{\textbf{(aIUC-regime)}} If $Cg(r) \geq |\log f(r)|$, $r \geq R_0$, for some $C>0$, then for every $t > \max\{30t_{\bee},\,  \mathsf{K_1}\tau_0\}$ with $\tau_0:= C$
        \begin{gather*}
            U_t \I(x) \asymp  e^{-\lambda_0 t} \left(1 \wedge \frac{f(|x|)}{g(|x|)}\right)
            \quad\text{for all\ \ } x \in \real^d.
        \end{gather*}

    \item\label{cor:cor3-b}
        \emph{\textbf{(non-aIUC-regime)}} If $\lim_{r\to\infty} g(r)/|\log f(r)| = 0$, then there are constants $\Ci, \Cii  >0$ such that for every $t > \max\{t_{\bee},\, \mathsf{K_1}\ttau{n_0+4}\}$
        \begin{align*}
        \Ci  e^{-\lambda_0 t} \left(1 \wedge \frac{f(|x|)}{g(|x|)}\right)
        &\leq U_t \I(x) \leq \Cii  e^{-\lambda_0 t} \left(1 \wedge \frac{f(|x|)}{g(|x|)}\right)
        &&\text{for\ \ } |x| < \ittau(t/\mathsf{K_1}),
        \\
        \Ci \frac{e^{- \mathsf{K_1} t g(|x|)}}{g(|x|)}
        &\leq U_t \I(x) \leq \Cii \frac{e^{-\frac{t}{\mathsf{K_1}} g(|x|)}}{g(|x|)}
        &&\text{for\ \ } |x| \geq \ittau(t/\mathsf{K_1}).
        \end{align*}
    \end{enumerate}
\end{corollary}

\begin{proof}
Part~\ref{cor:cor3-a} and the first formula in~\ref{cor:cor3-b} are already stated in Corollaries~\ref{cor:aIUC} and~\ref{cor:cor_prog}; the second set of estimates in~\ref{cor:cor3-b} follows by arguments similar to those in Lemma~\ref{lem:expl_est_G}.\ref{lem:expl_est_G-b2}.
\end{proof}

\subsection{Doubling L\'evy measures} \label{sec5:doubling}

In this section we discuss profiles $f$ which enjoy the \emph{doubling property}. This means that $f$ is a decreasing function such that
\begin{align} \label{eq:f_doubl}
    \text{there exists $C \geq 1$ such that $f(r) \leq C f(2r)$ for every $r >0$.}
\end{align}
Throughout, we choose $n_0$ according to \eqref{con-n0}.

\begin{example}[Fractional and layered fractional Schr\"odinger operators]\label{ex:stable}
Let
\begin{gather*}
    \nu(x) = \sigma\left(\frac{x}{|x|}\right) f(|x|)
    \quad\text{with}\quad f(r) = r^{-d-\alpha}(e \vee r)^{-\gamma},
\end{gather*}
where $\alpha \in (0,2)$ and $\gamma \geq 0$, and $\sigma: \dsphere \to (0,\infty)$  is a function on the unit sphere $\dsphere \subset \real^d$ such that $\sigma(-\theta) = \sigma(\theta)$ and $0 < \inf_{\theta \in \dsphere} \sigma(\theta) \leq \sup_{\theta \in \dsphere} \sigma(\theta) < \infty$. Moreover, we assume that there is no Gaussian part in the L\'evy--Khintchine formula \eqref{eq:Lchexp}, i.e.\ $A \equiv 0$. Recall that this setup covers the following two important classes of L\'evy processes.
\begin{enumerate}
\item\label{ex:stable-a}
    \emph{Symmetric $\alpha$-stable processes} (if $\gamma = 0$), see~\cite{bib:Sat};
\item\label{ex:stable-b}
    \emph{Layered symmetric $\alpha$-stable processes} (if $\gamma > 2 -\alpha$), see~\cite{CK}.
\end{enumerate}
The assumptions \textup{(\ref{A1}.a,b)} are clearly satisfied and (\ref{A1}.c,d) follow from the doubling property \eqref{eq:f_doubl} ((\ref{A1}.d) is checked in Lemma \ref{lem:sufficient-djp}.\ref{sufficient-djp-a}. Moreover, for both classes of processes~\ref{ex:stable-a} and~\ref{ex:stable-b} the conditions in the assumption \textup{(\ref{A2})} follows directly from the available estimates of the corresponding transition densities, see e.g.~\cite[Theorem~2]{bib:KS2013}. In fact, they hold true for every fixed $t_{\bee}>0$ with appropriate constants $C_4, C_5$, depending on $t_{\bee}$.

Let
\begin{gather} \label{eq:doubl_pot}
    V(x) = \left(1 \vee \log|x|\right)^{\beta},
\end{gather}
for some $\beta >0$. Let us check \eqref{A3} and \eqref{A4}. We take $g(r) = (1 \vee \log r)^{\beta}$ and $R_0=e$ in \eqref{A3}; obviously, $C_6 = 1$ and $C_7=\log(1+e)^{\beta}$. Since $|\log f(r)| = (d+\alpha+\gamma)\log r$ for $r \geq e$, \eqref{A4} is satisfied with $h(r) = \left(r/(d+\alpha+\gamma)\right)^{\beta}$. We then have $\ttau{r} = (d+\alpha+\gamma) (\log r)^{1-\beta}$ and $\ittau(\tau)= \exp\left(\left(\frac{\tau}{d+\alpha+\gamma}\right)^{\frac{1}{1-\beta}}\right)$, for $\beta \in (0,1)$. Set $\mathsf{K_2} := 3 \mathsf{K} = 12 C_6 C_7^2 = 12 \left(\log(1+e)\right)^{2\beta}$ and $\mathsf{K_3} := 4 \mathsf{K} = 16 C_6 C_7^2 = 16 \left(\log(1+e)\right)^{2\beta}$.

We have the following large time estimates.
\begin{enumerate}
\item
    If $\beta \geq 1$, then we are in the \emph{aIUC-regime}, and we have for $t > 30t_{\bee} + (d+\alpha+\gamma)\mathsf{K}_2$ and all $x,y \in \real^d$
    \begin{gather*}
        u_t(x,y) \asymp \frac{e^{-\lambda_0 t}}{(1+|x|)^{d+\alpha+\gamma}(1 \vee \log|x|)^{\beta}(1+|y|)^{d+\alpha+\gamma} (1 \vee \log|y|)^{\beta}}.
    \end{gather*}

\item
    If $\beta \in (0,1)$, then we are in the \emph{non-aIUC-regime}, and there exists some $C\geq 1$ such that
    \begin{enumerate}
    \item
        for all $t> \max\{30t_\bee,\, \mathsf{K}_2\ttau{n_0+4}\}$ and $|x| \wedge |y| <\exp\left[\left(\frac{t}{\mathsf{K_2}(d+\alpha+\gamma)} \right)^{\frac{1}{1-\beta}}\right]$, one has
        \begin{gather*}
            u_t(x,y) \casymp{C} \frac{e^{-\lambda_0 t}}{(1+|x|)^{d+\alpha+\gamma}(1 \vee \log|x|)^{\beta}(1+|y|)^{d+\alpha+\gamma} (1 \vee \log|y|)^{\beta}}.
    \end{gather*}
		    In particular, the semigroup $\left\{U_t:t \geq 0\right\}$ is \emph{pIUC}.
  \item
    for all $t> \max \left\{ 30t_\bee,\, \mathsf{K}_2\ttau{n_0+4},\: (d+\alpha+\gamma)\mathsf{K}_2^{1/\beta} (4 |\lambda_0|)^{(1-\beta)/\beta} \right\}$ and $|x|, |y| \geq \exp\left[\left(\frac{t}{\mathsf{K_2}(d+\alpha+\gamma)} \right)^{\frac{1}{1-\beta}}\right]$, one has
    \begin{align*}
      &\frac{ C^{-1} e^{-\mathsf{K_3} t \left(\log (|x|\wedge|y|)\right)^{\beta}}} {(1+|x-y|)^{d+\alpha+\gamma}(1 \vee \log|x|)^{\beta}(1 \vee \log|y|)^{\beta}}  \leq u_t(x,y) \\
      &\phantom{(1+|x-y|)^{d+\alpha+\gamma}}
        \leq \frac{ C e^{-\frac{t}{\mathsf{K_3}} \left(\log (|x|\wedge|y|)\right)^{\beta}}}{(1+|x-y|)^{d+\alpha+\gamma}(1 \vee \log|x|)^{\beta}(1 \vee \log|y|)^{\beta}}.
        \end{align*}
    \end{enumerate}
\end{enumerate}
The estimates in~a) and~b1) follow directly from Corollaries~\ref{cor:aIUC} and~\ref{cor:cor_prog}; part~b2) is a consequence of the Corollary \ref{cor:cor4} stated below.

Clearly, if the growth order of the potential $V$ at infinity is slower than that in \eqref{eq:doubl_pot}, e.g.\ like $(\log \log|x|)^{\beta}$, then the corresponding Schr\"odinger heat kernels enjoy two-sided estimates as in part~b) with appropriate threshold functions $\ittau(t/\mathsf{K_2})$.
\end{example}

The next lemma is needed in the proof of Corollary \ref{cor:cor4}. Recall that $\ttau{r}$ and $\ittau(\tau)$ are defined in \eqref{def:initial_time} and \eqref{def:moving_boundary}.
\begin{lemma}\label{lem:expl_est_F}
    Let $f:(0,\infty) \to (0,\infty)$ be a decreasing profile with the doubling property \eqref{eq:f_doubl} and $\lim_{r\to\infty}f(r) = 0$. Assume that $g(r) = h\left(|\log f(r)|\right)$, $r \geq R_0$, with an increasing function $h:[|\log f(R_0)|,\infty) \to (0,\infty)$ such that $h(s)/s$ decreases to $0$ as $r \to \infty$.  There are constants $\Ci, \Cii  >0$ such that for every $\tau \geq 3\ttau{n_0+4}$ and $|x|, |y|  \geq \ittau(\frac 13 \tau)$
    \begin{align*}
        \Ci  e^{-\tau g(|x|\wedge|y|)} f_1(|x-y|)
        \leq F(\tau,x,y)
        \leq \Cii  e^{- \frac{1}{3}\tau g(n_0+2)} e^{-\frac{1}{3}\tau g(|x|\wedge|y|)} f_1(|x-y|).
    \end{align*}
\end{lemma}

\begin{proof}
The lower bound follows easily:
\begin{align*}
    F(\tau,x,y)
    \geq \int_{\substack{n_0+2 < |z| < |y| \\ |z-y| < 1}} f_1(|x-z|) f_1(|y-z|) e^{- \tau g(|z|)}\,dz  \geq c_1 e^{- \tau g(|y|)} f_1(|x-y|).
\end{align*}
For the proof of the upper bound we assume, without loss of generality, that $|y| \leq |x|$. Let
\begin{gather*}
    F(\tau,x,y)
    = \left(\int_{n_0+2 < |z| < |y|}  +  \int_{|y| \leq |z| \leq |x|}  \right)f_1(|x-z|) f_1(|y-z|) e^{- \tau g(|z|)}\,dz
    =: \mathrm{I} + \mathrm{II},
\end{gather*}
and observe that, by the monotonicity of $g$ and \eqref{eq:conv_f1}, $\mathrm{II} \leq c_2 e^{- \tau g(|y|)} f_1(|x-y|)$. We only need to estimate $\mathrm{I}$. Write
\begin{gather*}
     \mathrm{I}
     = \left(\int_{n_0+2 < |z| <  \ittau(\frac 13 \tau)}  +  \int_{ \ittau(\frac 13 \tau) \leq |z| \leq |y|}  \right)f_1(|x-z|) f_1(|y-z|) e^{- \tau g(|z|)}\,dz
     =: \mathrm{I}_1+\mathrm{I}_2.
\end{gather*}
With the argument in the proof of Lemma~\ref{lem:progress_est}.\ref{lem:progress_est-b}, we get $\mathrm{I}_1 \leq  c_3 e^{- \frac{1}{3}\tau g(n_0+2)} f(|x|) f(|y|)$. Since $|x-y| \leq 2|x|$ and $|y| \geq \ittau(\frac 13 \tau)$, \eqref{eq:f_doubl} together with the inequality $f(|y|) \leq e^{- \frac{1}{3}\tau g(|y|)}$ yields $f(|x|) f(|y|) \leq c_4 f_1(|x-y|) e^{- \frac{1}{3}\tau g(|y|)}$. This proves $ \mathrm{I}_1 \leq  c_5 e^{- \frac{1}{3}\tau g(n_0+2)} f_1(|x-y|) e^{- \frac{1}{3}\tau g(|y|)}$.

Now we estimate $\mathrm{I}_2$. The monotonicity of $g$ and $f_1$ combined with $|z| \geq \ittau(\frac 13 \tau) \geq n_0+3$ gives
\begin{align*}
    \mathrm{I}_2
    &\leq \left(\int_{\substack{\ittau(\frac 13 \tau) \leq |z| \leq |y| \\ 2|x-z| \geq |x-y|}}
    +  \int_{\substack{\ittau(\frac 13 \tau) \leq |z| \leq |y| \\ 2|y-z| \geq |x-y|}}  \right) f_1(|x-z|) f_1(|y-z|) e^{- \tau g(|z|)}\,dz \\
    &\leq f_1(\tfrac 12|x-y|) e^{- \frac{1}{3}\tau  g(n_0+2)} \int_{\ittau(\frac 13 \tau) \leq |z| \leq |y|}
    \left( f_1(|x-z|) + f_1(|y-z|) \right)\frac{e^{- \frac{1}{3}\tau g(|z|)}}{f(|z|)} f(|z|)\,dz.
\end{align*}
Note that $f_1$ inherits the doubling property \eqref{eq:f_doubl} from $f$. Together with the last (monotonicity) assertion in Lemma~\ref{lem:non-aIUC} and \textup{(\ref{A1}.b,c,d)} we get
\begin{align*}
    \mathrm{I}_2
    &\leq c_3 f_1(|x-y|) e^{- \frac{1}{3}\tau  g(n_0+2)} \frac{e^{- \frac{1}{3}\tau g(|y|)}}{f(|y|)}\int_{\ittau(\frac 13 \tau) \leq |z| \leq |y|} \left( f_1(|x-z|) + f_1(|y-z|) \right) f(|z|)\,dz \\
    &\leq c_4  e^{- \frac{1}{3}\tau  g(n_0+2)} f_1(|x-y|) e^{- \frac{1}{3}\tau g(|y|)}.
\end{align*}
This completes the proof.
\end{proof}

We have the following corollary.

\begin{corollary}\label{cor:cor4}
    Assume the doubling condition \eqref{eq:f_doubl}, \eqref{A1}--\eqref{A4} with $R_0 >0$, $t_{\bee}>0$ and $\lim_{r\to\infty} g(r)/|\log f(r)| = 0$. Define $\mathsf K_2 = 12C_6C_7^2$ and $\mathsf K_3 = \frac 43\mathsf K_2$. There are constants $\Ci, \Cii >0$ such that for every $t > \max\{30t_{\bee},\, \mathsf{K_2}\ttau{n_0+4}\}$ and $|x|, |y| \geq  \ittau(t/\mathsf{K_2})$ satisfying $g(\ittau(t/\mathsf{K_2})) \geq 4\mathsf{K_2}|\lambda_0|$
        \begin{align*}
            \Ci  \frac{e^{-\mathsf{K_3} t g(|x|\wedge|y|)}f_1(|x-y|)}{g(|x|)g(|y|)}
            \leq u_t(x,y)
            \leq \Cii  \frac{e^{-\frac{t}{\mathsf{K_3}} g(|x|\wedge|y|)} f_1(|x-y|)}{g(|x|)g(|y|)}.
        \end{align*}
\end{corollary}

Note that for all $t \geq \mathsf{K_2}\ttau{g^{-1}(4\mathsf{K_2}|\lambda_0|)}$ --- here $g^{-1}$ stands for the generalized right-inverse of the increasing function $g$ --- the inequality $g(\ittau(t/\mathsf{K_2})) \geq 4\mathsf{K_2}|\lambda_0|$ is satisfied. This follows immediately from the fact that $g(g^{-1})(s)\geq s$ and $\ttau{\ittau(s)}\geq s$. In particular, we have a threshold value for $t$ in Corollary~\ref{cor:cor4}

\begin{proof}[Proof of Corollary~\ref{cor:cor4}]
The lower bound follows immediately if we insert the lower estimate from Lemma \ref{lem:expl_est_F} into Theorem \ref{th:th1}.\ref{th:th1-c}.

Similarly, we get for the upper estimate
\begin{align*}
    u_t(x,y)
    &\leq c_1 \frac{e^{-\frac{t}{\mathsf{K_2}} g(|x|\wedge|y|)} f_1(|x-y|) \vee e^{-\lambda_0 t} f(|x|) f(|y|)}{g(|x|)g(|y|)}\\
    &= c_1 \frac{e^{-\frac{t}{\mathsf{K_2}} g(|x|\wedge|y|)} f_1(|x-y|) \vee e^{-\lambda_0 t} f(|x|\wedge|y|) f(|x|\vee|y|)}{g(|x|)g(|y|)}.
\end{align*}
Since $f_1$ is doubling and decreasing,
\begin{gather*}
    f(|x| \vee |y|)
    \leq c_2 f_1(|x| \vee |y|)
    \leq c_3 f_1(2(|x| \vee |y|))
    \leq c_3 f_1(|x-y|).
\end{gather*}
Moreover, by Lemma \ref{lem:non-aIUC},
\begin{gather*}
    f(|x| \wedge |y|) \leq e^{-\frac{t}{\mathsf{K_2}} g(|x| \wedge |y|)},
    \quad
    |x|, |y| \geq  \ittau(t/\mathsf{K_2}).
\end{gather*}
Thus, for all $|x|, |y| \geq  \ittau(t/\mathsf{K_2})$
\begin{align*}
    u_t(x,y)
    &\leq c_4 \frac{ e^{-\lambda_0 t}  e^{-\frac{t}{\mathsf{K_2}} g(|x|\wedge|y|)} f_1(|x-y|)}{g(|x|)g(|y|)} \\
    &\leq c_4  \frac{e^{\big(|\lambda_0|-\frac 14\frac{1}{\mathsf{K_2}} g(\ittau(t/\mathsf{K_2})) \big) t } e^{-\frac 34 \frac{t}{\mathsf{K_2}} g(|x|\wedge|y|)} f_1(|x-y|)}{g(|x|)g(|y|)},
\end{align*}
and the claimed bound follows for $t>\max\{30t_{\bee},\, \mathsf{K_2}\ttau{n_0+4}\}$ and $|\lambda_0| \leq g(\ittau(t/\mathsf{K_2}))/(4\mathsf{K_2})$ hold.
\end{proof}

\subsection{Exponential L\'evy measures} \label{sec5:exp}
Throughout this section we assume that the profile $f$ of the L\'evy density has \emph{exponential decay at infinity}, i.e.\ $f$ is a decreasing function such that
\begin{align} \label{eq:exp_decay}
    f(r) = e^{-\kappa r} r^{-\gamma}, \quad r \geq 1,
\end{align}
for some $\kappa > 0$ and $\gamma \geq 0$. This setting covers several important examples.

Let us show that such a profile satisfies the direct jump property (\ref{A1}.d) if $\gamma > \frac 12(d+1)$. We check the criteria from Lemma~\ref{lem:sufficient-djp}: If $\gamma > d$, we can use Lemma \ref{lem:sufficient-djp}.\ref{sufficient-djp-b}; note that $d=1$ implies $\frac 12(d+1)=d$. If $d>1$ and $d\geq \gamma > \frac 12 (d+1)$, we use Lemma~\ref{lem:sufficient-djp}.\ref{sufficient-djp-c}. For this, we have to check the integrability condition \eqref{eq:int-cond}. Since $f'(|y|)/f(|y|) = -\kappa - \gamma/|y|$, we have for $|y| > 1$
\begin{gather*}
    \int_{|y|>1} e^{-\frac{f'(|y|)}{f(|y|)} y_1} f(|y|)\,dy
    = \int_{\substack{y_1 \leq 1 \\ |y|>1}} e^{\kappa y_1 + \gamma \frac{y_1}{|y|}} f(|y|)\,dy
        + \int_{y_1 > 1} e^{\kappa y_1 + \gamma \frac{y_1}{|y|}} f(|y|)\,dy.
\end{gather*}
It is easy to see that the first integral on the right-hand side is finite. The second integral is bounded by $e^\gamma \, \mathrm{I}$, where
\begin{gather*}
    \mathrm{I}
    := \int_{y_1 > 1} e^{\kappa y_1} f(|y|)\,dy
    = \int_{y_1 > 1} e^{\kappa (y_1-|y|)} |y|^{-\gamma} \,dy.
\end{gather*}
Now we introduce spherical coordinates for $(y_2,..,y_d) \in \real^{d-1}$ and observe that $y_1 - \sqrt{y_1^2+r^2} = -r^2/(y_1+ \sqrt{y_1^2+r^2})$. So,
\begin{align*}
    \mathrm{I}
    =
    c \int_1^{\infty} \int_0^{\infty} e^{-\frac{\kappa r^2}{y_1+ \sqrt{y_1^2+r^2}}} \frac{r^{d-2}}{(y_1^2+r^2)^{\delta/2}} \,dr \,dy_1
    \leq
    c \int_1^{\infty} \int_0^{\infty} e^{-\frac{\kappa r^2}{y_1+ \sqrt{y_1^2+y_1r^2}}} r^{d-2} y_1^{-\delta} \,dr \,dy_1;
\end{align*}
changing variables  in the inner integral according to $r = \sqrt{y_1}u$ yields that the double integral factorizes
\begin{gather*}
    \int_0^{\infty} e^{-\frac{\kappa u^2}{1+ \sqrt{1+u^2}}}\, u^{d-2}\, du \cdot
    \int_1^{\infty} y_1^{\frac 12 (d-1)-\delta} \, dy_1.
\end{gather*}
The first integral is finite and the second is finite if, and only if $\gamma > \frac 12(d+1)$. This implies that $\mathrm{I}$ is finite for this range of $\gamma$.

\begin{example}[Quasi-relativistic Schr\"odinger Operators]\label{ex:relativistic}
Let $\sigma$ be a function on the unit sphere $\dsphere \subset \real^d$ such that $\sigma(-\theta) = \sigma(\theta)$ and
$0<\inf_{\theta \in \dsphere} \sigma(\theta) \leq \sup_{\theta \in \dsphere} \sigma(\theta) < \infty$. Moreover, let $\alpha \in (0,2)$, $\kappa >0$ and $\gamma > \frac 12 (d+1)$.
Consider
\begin{gather*}
\nu(x) = \sigma\left(\frac{x}{|x|}\right) g(|x|) \qquad \text{where} \qquad g(r) \asymp f(r) := e^{-\kappa r} r^{-d-\alpha}(1 \vee r)^{d+\alpha-\gamma}.
\end{gather*}
We assume that $A \equiv 0$ in \eqref{eq:Lchexp}. This setting covers two important classes of L\'evy processes whose L\'evy measure decays exponentially.
\begin{enumerate}
\item\label{ex:relativistic-a} \emph{Relativistic symmetric $\alpha$-stable processes} ($\kappa = m^{1/\alpha}$, $m >0$, $\gamma = \frac 12(d+\alpha+1)$), see~\cite{bib:R};
\item\label{ex:tempered-b} (\emph{Exponentially}) \emph{tempered symmetric $\alpha$-stable processes} ($\kappa >0$, $\gamma = d+\alpha$), see~\cite{bib:Ros}.
\end{enumerate}
Assumption \eqref{A1} holds, we have checked (\ref{A1}.d) above, and \eqref{A2} follows from~\cite[Theorem~2]{bib:KS2013}; as before, it holds true for every fixed $t_{\bee}>0$ with appropriate constants $C_4, C_5$, depending on $t_{\bee}$. Let
\begin{gather} \label{eq:exp_pot}
    V(x) = (1 \vee |x|)^{\beta},
\end{gather}
for some $\beta >0$. Since $|\log f(r)| = \kappa r + \gamma \log r$ for $r \geq 1$, we may choose $h(r):= (r/\kappa)^{\beta}$ and $g(r):=\big((r \vee 1) + (\gamma/\kappa) \log (r \vee 1)\big)^{\beta}$ so that $g(r) = h(|\log f(r)|)$ for $r\geq 1$. The assumptions \eqref{A3} and \eqref{A4} hold for this choice of $g$ and $h$ with $R_0=1$. With a straightforward calculation we can check that $C_6 = \big(\kappa e/(\gamma + \kappa e)\big)^{\beta}$ and $C_7=(2+(\gamma/\kappa)\log 2)^{\beta}$. In particular, $g(r) \casymp{C_6} (1 \vee r)^{\beta}$, for $r \geq 0$. The threshold functions defined in \eqref{def:initial_time} and \eqref{def:moving_boundary} are given by
\begin{gather*}
    \ttau{r} = \kappa^{\beta} u(r)^{1-\beta}
		\quad \text{and} \quad
    \ittau(\tau)
    =u^{-1}\left(\left(\frac{\tau}{\kappa^{\beta}}\right)^{\frac{1}{1-\beta}}\right)
    \quad\text{where\ \ } u(r)= |\log f(r)| = \kappa r + \gamma \log r.
\end{gather*}
Clearly, $\ittau(\tau) \asymp \big(\tau/\kappa^{\beta}\big)^{\frac{1}{1-\beta}}$, $\tau \geq \kappa$. Set $\mathsf{K_2} := 3 \mathsf{K} = 12 C_6 C_7^2$ and $\mathsf{K_4} := C_6 \mathsf{K_2} = 12 C_6^2 C_7^2$.

We have the following large time estimates. As before, parts~\ref{lt-a} and~\ref{lt-b1} follow immediately from Corollaries~\ref{cor:aIUC} and~\ref{cor:cor_prog} above, while part~\ref{lt-b2} will be a consequence of Corollary~\ref{cor:cor5} stated below.
\begin{enumerate}
\item\label{lt-a}
    If $\beta \geq 1$, then we are in the \emph{aIUC-regime}, and for every $t > 30t_{\bee} + \kappa \mathsf{K}_2$ we have
    \begin{gather*}
        u_t(x,y) \asymp e^{-\lambda_0 t} \frac{e^{-\kappa|x|-\kappa|y|}}{(1+|x|)^{\gamma+\beta}(1+|y|)^{\gamma+\beta}},
        \quad x,y \in \real^d.
    \end{gather*}

\item\label{lt-b}
    If $\beta \in (0,1)$, then we are in the \emph{non-aIUC-regime}, and there exists a constant $C\geq 1$ such that for every $t > \max\{30t_{\bee},\, \mathsf{K_2}\ttau{n_0+4}\}$ we have
    \begin{enumerate}
    \item\label{lt-b1}
        for $|x| \wedge |y| < \ittau(t/\mathsf{K_2})$,
        \begin{gather*}
            u_t(x,y)
            \casymp{C} e^{-\lambda_0 t} \frac{e^{-\kappa|x|-\kappa|y|}}{(1+|x|)^{\gamma+\beta}(1+|y|)^{\gamma+\beta}}.
        \end{gather*}
		In particular, the semigroup $\left\{U_t:t \geq 0\right\}$ is \emph{pIUC}.

    \item\label{lt-b2}
        for $|x|, |y| \geq \ittau(t/\mathsf{K_2})$,
        \begin{align*}
        \frac{1}{C} \frac{1}{|x|^{\beta}|y|^{\beta}}&\left(\frac{e^{-\lambda_0 t-\kappa|x|-\kappa|y|}}{|x|^{\gamma}|y|^{\gamma}}\vee  \frac{e^{-\mathsf{K_4} t (|x|\wedge|y|)^{\beta} - \kappa|x-y|}}{(1+|x-y|)^{\gamma}} \vee H(\mathsf{K_4} t,x,y)\right)
        \leq u_t(x,y) \\
        &\leq C  \frac{1}{|x|^{\beta}|y|^{\beta}}\left(\frac{e^{-\lambda_0 t-\kappa|x|-\kappa|y|}}{|x|^{\gamma}|y|^{\gamma}}\vee  \frac{e^{-\frac{t}{\mathsf{K_4}} (|x|\wedge|y|)^{\beta} - \kappa|x-y|}}{(1+|x-y|)^{\gamma}} \vee H(t/\mathsf{K_4},x,y)\right),
        \end{align*}
		where
		\begin{align}\label{eq:def_H}
            H(\tau,x,y)
            := \int_{n_0+2 \leq |z| \leq |x| \wedge |y|} \frac{e^{-\kappa(|x-z| + |z-y|)}}%
            {(1\vee|x-z|)^{\gamma}(1\vee|z-y|)^{\gamma}}\,e^{-\tau |z|^{\beta}}\,dz.
		\end{align}
		For $d=1$ these bounds can be simplified:
        \begin{align*}
            \frac{1}{C} \frac{1}{|x|^{\beta}|y|^{\beta}}&\left(\frac{e^{-\lambda_0 t-\kappa|x|-\kappa|y|}}{|x|^{\gamma}|y|^{\gamma}}\vee  \frac{e^{-\mathsf{K_4} t (|x|\wedge|y|)^{\beta} - \kappa|x-y|}}{(1+|x-y|)^{\gamma}} \right)
            \leq u_t(x,y) \\
            &\leq C  \frac{1}{|x|^{\beta}|y|^{\beta}}\left(\frac{e^{-\lambda_0 t-\kappa|x|-\kappa|y|}}{|x|^{\gamma}|y|^{\gamma}}\vee  \frac{e^{-\frac{t}{\mathsf{K_4}} (|x|\wedge|y|)^{\beta} - \kappa|x-y|}}{(1+|x-y|)^{\gamma}}\right).
        \end{align*}
    \end{enumerate}
\end{enumerate}
Also in the multidimensional case one can give (upper) estimates for $H(t/\mathsf{K_4},x,y)$ which will still lead to, say, exponential decay of $u_t(x,y)$, but we may loose the sharp two-sided estimates in \ref{lt-b2}. Since such estimates depend very much on the particular setting, we do not give further details here. See, however, Corollary~\ref{cor:cor5}.\ref{cor:cor5-a}.

\bigskip
If the growth order of the potential $V$ at infinity is slower than that in \eqref{eq:exp_pot} (e.g.\ $\log|x|^{\beta}$), then the corresponding Schr\"odinger heat kernels enjoy two-sided estimates similar to those in part~\ref{lt-b} above with appropriate threshold functions $\ittau(t/\mathsf{K_2})$.
\end{example}

In order to complete Example~\ref{ex:relativistic}, we need Corollary~\ref{cor:cor5} which is based on the following two-sided estimates for the function $F(\tau,x,y)$ defined in \eqref{def:def_F}.

\begin{lemma}\label{lem:expon_est_F}
    Let $d=1$ and let $f:(0,\infty) \to (0,\infty)$ be a decreasing profile such that \eqref{eq:exp_decay} holds with $\gamma > 1$. Assume that $g(r) = h\left(\kappa r + \gamma \log r\right)$, $r \geq R_0$, with an increasing function $h:[\kappa R_0 + \gamma \log R_0,\infty) \to (0,\infty)$ such that $h(s)/s$ decreases to $0$ as $s \to \infty$ and $h(\kappa s + \gamma \log s)/ \log s$ is monotone for $s \geq R_0  \vee e$. There are constants $\Ci,\Cii$ such that for every $\tau \geq 3 \ttau{n_0+4}$ and $|x|, |y| \geq \ittau(\frac 13\tau)$
    \begin{align*}
        \Ci  \left(e^{-\tau g(n_0+3)} \right. &\left. f(|x|)f(|y|) \vee e^{-\tau g(|x|\wedge|y|)} f_1(|x-y|)\right) \leq F(\tau,x,y) \\
        &\leq \Cii  \left(e^{- \frac{1}{3}\tau g(n_0+2)} f(|x|)f(|y|) \vee e^{- \frac{1}{3}\tau g(n_0+2)} e^{-\frac{1}{3}\tau g(|x|\wedge|y|)} f_1(|x-y|) \right).
    \end{align*}
\end{lemma}
\begin{proof}
The lower bound is easy, cf.\ the last lines of the proof of Lemma~\ref{lem:progress_est} and the second line of the proof of Lemma~\ref{lem:expl_est_F}.

The proof of the upper bound is similar to the proof of Lemma~\ref{lem:expl_est_F}. The only difference is in the upper estimate of the integral
\begin{gather*}
    \mathrm{I}_{2}:= \int_{\ittau(\frac 13\tau) \leq |z| < |y|}f_1(|x-z|) f_1(|y-z|) e^{- \tau g(|z|)}\,dz
\end{gather*}
in the case when $|y| > \ittau(\frac 13\tau)$. By symmetry, we may assume that $|x| \geq |y|$ and $x > 0$. We have two cases: $\ittau(\frac 13\tau) < y \leq x$ and $-x \leq y < - \ittau(\frac 13\tau)$.

\medskip\noindent
\text{Case 1}: $\ittau(\frac 13\tau) < y \leq x$. We have $|x-z|>|x-y|$ and, by monotonicity, $f_1(|x-y|) \geq f_1(|x-z|)$. Together with the last assertion in Lemma~\ref{lem:non-aIUC} which says that the function $r \mapsto e^{- \frac{1}{3}\tau g(r)}/f(r)$ is increasing on $\big[\ittau(\frac 13\tau), \infty\big)$ and \textup{(\ref{A1}.d)}, this gives
\begin{align*}
    \mathrm{I}_2
    &\leq c_1 f_1(|x-y|) e^{- \frac{1}{3}\tau  g(n_0+2)} \int_{\ittau(\frac 13\tau) < |z| < |y|} f_1(|y-z|) \frac{e^{- \frac{1}{3}\tau g(|z|)}}{f(|z|)} f(|z|)\,dz \\
    &\leq c_2 f_1(|x-y|) e^{- \frac{1}{3}\tau  g(n_0+2)} \frac{e^{- \frac{1}{3}\tau g(|y|)}}{f(|y|)}\int_{\ittau(\frac 13\tau) < |z| < |y|} f_1(|y-z|)) f(|z|)\,dz \\
    &\leq c_3  e^{- \frac{1}{3}\tau  g(n_0+2)} f_1(|x-y|) e^{- \frac{1}{3}\tau g(|y|)},
\end{align*}
which is the required estimate.

\medskip\noindent
\text{Case 2}: $-x \leq y < - \ittau(\frac 13\tau)$. For $\ittau(\frac 13\tau) \leq |z| < |y|$ we have $|x-z|+|y-z|= |x| + |y| = |x-y|$, which gives
\begin{align*}
    f_1(|x-z|) f_1(|y-z|)
    &\leq c_4 e^{-\kappa(|x-z| + |z-y|)} (1 \wedge |x-z|^{-\gamma})(1 \wedge |z-y|^{-\gamma}) \\
    &= c_4 e^{-(|x| + |y|)} (1 \wedge |x-z|^{-\gamma})(1 \wedge |z-y|^{-\gamma}) \\
	&= c_4 e^{-|x-y|} (1 \wedge |x-z|^{-\gamma})(1 \wedge |z-y|^{-\gamma}).	
\end{align*}
This shows
\begin{equation}\label{aux:I2}\begin{aligned}
    \mathrm{I}_2
    &= \int_{\ittau(\frac 13\tau) \leq |z| < |y|}f_1(|x-z|) f_1(|y-z|) e^{- \tau g(|z|)}\,dz \\
    &\leq c_4 e^{-\kappa(|x|+|y|)} \widetilde F(\tau,x,y) \\
    &= c_4 e^{-\kappa|x-y|} \widetilde F(\tau,x,y),
\end{aligned}\end{equation}
where
\begin{gather*}
    \widetilde F(\tau,x,y)
    := \int_{n_0+2 \leq |z| < |x| }\widetilde f_1(|x-z|) \widetilde f_1(|y-z|) e^{- \tau g(|z|)}\,dz
    \quad \text{with\ \ } \widetilde f(r)= r^{-\gamma}.
\end{gather*}
We still have to estimate the function $\widetilde F(\tau,x,y)$ from above. Define $\widetilde h(s) := h(\kappa e^{s/\gamma} + s)$ and observe that for $r \geq R_0$ we have $g(r)= h(\kappa r + \gamma \log r) =: \widetilde h(|\log \widetilde f(r)|)$. By assumption, $h(s)$ is increasing and $h(\kappa s + \gamma \log s)/\log s$ is monotone for $s \geq R_0 \vee e$. This shows that $\widetilde h(s)$ is increasing and $\widetilde h(s)/s$ is monotone for $s \geq \gamma \log R_0 \vee \gamma$.  It implies that the pairs $\widetilde f$ and $g$ show all possible behaviours which were discussed in Remark~\ref{rem:regimes}. In particular, the assumptions of Lemma~\ref{lem:aiuc_est}.\ref{lem:aiuc_est-a}, \ref{lem:progress_est}.\ref{lem:progress_est-b} and \ref{lem:expl_est_F} are satisfied for $\widetilde f$ and $g$. If we combine these lemmas, we obtain the upper estimate of the function $\widetilde F(\tau,x,y)$ in all possible regions, and so
\begin{gather*}
    \widetilde F(\tau,x,y)
    \leq c_5 \left(e^{- \frac{1}{3}\tau g(n_0+2)} \widetilde f(|x|) \widetilde f(|y|) \vee e^{- \frac{1}{3}\tau g(n_0+2)} e^{-\frac{1}{3}\tau g(|x|\wedge|y|)} \widetilde f_1(|x-y|) \right).
\end{gather*}
Plugging this into \eqref{aux:I2}, gives, in view of \eqref{eq:exp_decay},
\begin{gather*}
    \mathrm{I}_2
    \leq c_6 \left(e^{- \frac{1}{3}\tau g(n_0+2)} f(|x|) f(|y|) \vee e^{- \frac{1}{3}\tau g(n_0+2)} e^{-\frac{1}{3}\tau g(|x|\wedge|y|)}f_1(|x-y|) \right),
\end{gather*}
and the proof is finished.
\end{proof}

We are now ready to state our last corollary. Recall that $\ttau{r}$ and $\ittau(\tau)$ are defined in \eqref{def:initial_time} and \eqref{def:moving_boundary}.
\begin{corollary}\label{cor:cor5}
    Assume \eqref{A1}--\eqref{A4} with $R_0 >0$, $t_{\bee}>0$, let $f$ be of exponential type \eqref{eq:exp_decay} and $\lim_{r\to\infty}g(r)/|\log f(r)| = 0$.
\begin{enumerate}
    \item\label{cor:cor5-a}
    There are constants $\Ci, \Cii >0$ such that for every $t >\max\{30t_{\bee},\,  \mathsf{K_2}\ttau{n_0+4}\}$ and $|x|, |y| \geq  \ittau(t/\mathsf{K_2})$
    \begin{align*}
        \frac{\Ci }{g(|x|)g(|y|)}&\left(\frac{e^{-\lambda_0 t-\kappa|x|-\kappa|y|}}{|x|^{\gamma}|y|^{\gamma}}\vee  \frac{e^{-\mathsf{K_2} t g(|x|\wedge|y|) - \kappa|x-y|}}{(1+|x-y|)^{\gamma}} \vee H(\mathsf{K_2} t,x,y)\right)
        \leq u_t(x,y) \\
        &\leq   \frac{\Cii }{g(|x|)g(|y|)}\left(\frac{e^{-\lambda_0 t-\kappa|x|-\kappa|y|}}{|x|^{\gamma}|y|^{\gamma}}\vee  \frac{e^{-\frac{t}{\mathsf{K_2}} g(|x|\wedge|y|) - \kappa|x-y|}}{(1+|x-y|)^{\gamma}} \vee H(t/\mathsf{K_2},x,y)\right),
    \end{align*}
	where
	\begin{align} \label{eq:def_H_gen}
        H(\tau,x,y)
        := \int_{n_0+2 \leq |z| \leq |x| \wedge |y|} \frac{e^{-\kappa(|x-z| + |z-y|)} }{(1\vee|x-z|)^{\gamma}(1\vee|z-y|)^{\gamma}}
        \,e^{-\tau g(|z|)}\,dz.
	\end{align}
    Moreover, there exists $\Ciii>0$ such that for every $\tau > 3\ttau{n_0+4}$ and $|x|, |y| \geq  \ittau(\frac 13\tau)$
    \begin{align}\label{eq:H-est}
        H(\tau,x,y)
        \leq \Ciii  &\left[\frac{e^{-\frac{1}{3}\tau g(n_0+2)-\kappa|x|-\kappa|y|}}{|x|^{\gamma}|y|^{\gamma}} \right. \\
        &\nonumber\mbox{} + \left.\left(\frac{e^{-\frac{1}{3}\tau g(|x| \wedge |y|)-\frac{\kappa}{2}|x-y|}}{(1\vee|x-y|)^{\gamma}}
        \wedge \frac{e^{-\kappa|x-y|}}{(1\vee|x-y|)^{\gamma}}
        \int_{\ittau(\frac 13\tau) \leq |z| \leq |x| \wedge |y|} e^{-\frac{1}{3}\tau g(|z|)}\, dz \right) \right].
    \end{align}

    \item\label{cor:cor5-b}	
    If, in addition, $\gamma > 1 = d$ and the function $r \mapsto g(r) / \log r$ is monotone on $[R_0 \vee e, \infty)$, then there exist  constants $\Ci, \Cii >0$ such that for every $t > \max\{30t_{\bee},\, \mathsf K_2 \ttau{n_0+4}\}$ and $|x|, |y| \geq  \ittau(t/\mathsf{K_2})$
    \begin{align*}
        \frac{\Ci }{g(|x|)g(|y|)}
        &\left(\frac{e^{-\lambda_0 t-\kappa|x|-\kappa|y|}}{|x|^{\gamma}|y|^{\gamma}}
        \vee  \frac{e^{-\mathsf{K_2} t g(|x|\wedge|y|) - \kappa|x-y|}}{(1+|x-y|)^{\gamma}}\right)
        \leq u_t(x,y) \\
        &\leq   \frac{\Cii }{g(|x|)g(|y|)}\left(\frac{e^{-\lambda_0 t-\kappa|x|-\kappa|y|}}{|x|^{\gamma}|y|^{\gamma}}
        \vee  \frac{e^{-\frac{t}{\mathsf{K_2}} g(|x|\wedge|y|) - \kappa|x-y|}}{(1+|x-y|)^{\gamma}}\right),
    \end{align*}
\end{enumerate}			
\end{corollary}

\begin{proof}
We begin with~\ref{cor:cor5-a} and show that for every $n_0 > R_0$ there are constants $c_1, c_2 >0$ such that for all $\tau > 0$ and $|x|, |y| \geq  n_0+3$
\begin{align} \label{eq:F-by-H}
    c_1 \bigg( e^{-\tau g(|x|\wedge|y|)} &\frac{e^{-\kappa(|x-y|)}}{(1\vee|x-y|)^{\gamma}}  \vee H(\tau,x,y) \bigg)
    \leq F(\tau,x,y) \\
    &\nonumber\leq c_2 \bigg( e^{-\tau g(|x|\wedge|y|)} \frac{e^{-\kappa(|x-y|)}}{(1\vee|x-y|)^{\gamma}} \vee H(\tau,x,y) \bigg),
\end{align}
where $F(\tau,x,y)$ is defined in \eqref{def:def_F}. By symmetry, we may assume that $|y| \leq |x|$. Clearly,
\begin{gather*}
    F(\tau,x,y)
    \asymp H(\tau,x,y)  +  \int_{|y| \leq |z| \leq |x|}  f_1(|x-z|) f_1(|y-z|) e^{- \tau g(|z|)}\,dz.
\end{gather*}
From this, it is easy to get the lower bound in \eqref{eq:F-by-H}, cf.~the second line of the proof of Lemma~\ref{lem:expl_est_F}.

For the proof of the upper bound, we bound the term $e^{- \tau g(|z|)}$ under the integral by $e^{- \tau g(|y|)}$, and then we use \eqref{eq:conv_f1}.

The estimates of the heat kernel in part~\ref{cor:cor5-a} follows from a combination of the estimates in Theorem \ref{th:th1}.\ref{th:th1-c} and \eqref{eq:F-by-H} --- note that $t > \max\{30t_{\bee},\, \mathsf{K_2}\ttau{n_0+4}\}$ so that $\ittau(t/\mathsf{K_2}) > n_0+3$, and take $\tau = t/\mathsf{K_2}$ in the upper bound and $\tau = \mathsf{K_2} t$ in the lower bound.

Now we show the estimates \eqref{eq:H-est}. We have
\begin{gather*}
    H(\tau,x,y)
    \leq c_3 \left(\int_{n_0+2 < |z| <  \ittau(\frac 13\tau)}  +  \int_{ \ittau(\frac 13\tau) \leq |z| \leq |y|}  \right)f_1(|x-z|) f_1(|y-z|) e^{- \tau g(|z|)}\,dz
    =: \mathrm{I}_1+\mathrm{I}_2.
\end{gather*}
With the argument from the proof of Lemma~\ref{lem:expl_est_F}, we get
\begin{gather*}
    \mathrm{I}_1
    \leq c_4 e^{- \frac{1}{3}\tau g(n_0+2)} f(|x|) f(|y|)
    = e^{- \frac{1}{3}\tau g(n_0+2)}e^{-\kappa|x|-\kappa|y|}|x|^{-\gamma}|y|^{-\gamma}
\intertext{and}
    \mathrm{I}_2
    \leq c_5  f_1\left(\frac{|x-y|}{2}\right) e^{- \frac{1}{3}\tau g(|x| \wedge |y|)}
    \leq c_6 e^{-\frac{1}{3}\tau g(|x| \wedge |y|)-\frac{\kappa}{2}|x-y|} (1\vee|x-y|)^{-\gamma}.
\end{gather*}
On the other hand, we can also use \eqref{eq:conv_f1_pw} to get
\begin{gather*}
    \mathrm{I}_2
    \leq c_7 \frac{e^{-\kappa|x-y|}}{(1\vee|x-y|)^{\gamma}} \int_{\ittau(\tau) \leq |z| \leq |x| \wedge |y|} e^{-\frac{1}{3}\tau g(|z|)}\,dz.
\end{gather*}
If we combine these estimates, we get the upper bound in \eqref{eq:H-est}.

Part~\ref{cor:cor5-b} follows from Theorem \ref{th:th1}.\ref{th:th1-c} and Lemma \ref{lem:expon_est_F}.
\end{proof}

\end{document}